\documentclass{amsart}

\usepackage{amsmath,amscd,amsfonts,amsthm,amssymb,amscd,mathrsfs}
\usepackage{appendix}
\usepackage{bm}
\usepackage{cite}
\usepackage{geometry}
\usepackage{graphicx}
\usepackage[colorlinks,linkcolor=blue,citecolor=red]{hyperref}
\hypersetup{CJKbookmarks}
\usepackage{setspace}
\usepackage{stmaryrd}
\usepackage{url}

\allowdisplaybreaks[4]

\newtheorem{definition}{Definition}[section]

\newtheorem{theorem}{Theorem}[section]
\newtheorem{lemma}[theorem]{Lemma}
\newtheorem{proposition}{Proposition}[section]
\newtheorem{corollary}{Corollary}[theorem]

\theoremstyle{definition}
\newtheorem{notation}{Notation}[section]
\newtheorem{remark}{Remark}[section]
\newtheorem{task}{Task}[section]

\numberwithin{equation}{section}

\DeclareMathOperator{\Tr}{Tr} 
\newcommand{\lHS}{\llbracket}
\newcommand{\rHS}{\rrbracket}
\newcommand{\Ft}{\widehat}
\newcommand{\Gp}{\mathbb{G}}
\newcommand{\DuGp}{\Ft{\mathbb{G}}}
\newcommand{\SU}{\mathrm{SU}}
\newcommand{\Op}{\mathbf{Op}}
\newcommand{\size[1]}{\langle{#1}\rangle}
\newcommand{\dist}{\mathrm{dist}}
\newcommand{\Hh[1]}{\mathcal{H}_{#1}}
\newcommand{\I[1]}{\mathrm{Id}_{\mathcal{H}_{#1}}}
\newcommand{\Df}{\mathbf{D}}
\newcommand{\Pa[1]}{\partial_{#1}}
\newcommand{\hN}{\frac{1}{2}\mathbb{N}_0}
\newcommand{\Ind}{\{+,-,0\}}

\author{Chengyang Shao}
\title{Para-differential Calculus on Compact Lie Groups and Spherical Capillary Water Waves}

\begin{document}
\maketitle
\begin{abstract}
This paper provides a para-differential calculus toolbox on compact Lie groups and homogeneous spaces. It helps to understand non-local, nonlinear partial differential operators with low regularity on manifolds with high symmetry. In particular, the paper provides a para-linearization formula for the Dirichlet-Neumann operator of a distorted 2-sphere, a key ingredient in understanding long-time behaviour of spherical capillary water waves. As an initial application, the paper provides a new proof of local well-posedness for spherical capillary water waves equation under weaker regularity assumptions compared to previous results.
\end{abstract}
\begin{spacing}{1.2}

\section{Introduction}
This is the sequel of the author's previous paper \cite{Shao2023Toolbox}, being the second split from \cite{Shao2023}. In this paper, we reduce the spherical capillary water waves equation into a para-differential form, thus giving a new proof of the local well-posedness of the Cauchy problem. This reduction sets stage for the study of long-time behaviour of the system.

\subsection{Equation of Motion for a Water Drop}
We are interested in the initial value problem for the motion of a water droplet under zero gravity, which is the starting point of a program proposed by the author in \cite{Shao2022}. Let us first describe the physical scenario. We pose the following assumptions on the fluid motion we aim to describe:
\begin{itemize}
    \item (A1) The perfect, incompressible, irrotational fluid of constant density $\rho_0$ occupies a smooth, compact region in $\mathbb{R}^3$.
    \item (A2) There is no gravity or any other external force in presence.
    \item (A3) The air-fluid interface is governed by the Young-Laplace law, and the effect of air flow is neglected.
\end{itemize}
Hydrodynamics of water droplet governed by (A1)-(A3) is a long-standing interest for hydrophysicists and astronautical engineers. To mention a few, hydrophysicists Tsamopoulos-Brown \cite{TB1983}, Natarajan-Brown \cite{NB1986} and Lundgren-Mansour \cite{LM1988} all carried out initial ``weakly nonlinear analysis" towards the fluid motion satisfying assumptions (A1)-(A3). There are also numbers of visual materials on such experiments conducted in spacecrafts by astronauts\footnote{See for example \url{https://www.youtube.com/watch?v=H_qPWZbxFl8&t} or \url{https://www.youtube.com/watch?v=e6Faq1AmISI&t}.}.

We assume that the boundary of the fluid region has the topological type of a smooth compact orientable surface $M$, and is described by a time-dependent embedding $\iota(t,\cdot):M\to\mathbb{R}^3$. We will denote a point on $M$ by $x$, the image of $M$ under $\iota(t,\cdot)$ by $M_t$, and the region enclosed by $M_t$ by $\Omega_t$. The unit outer normal will be denoted by $N(\iota)$. We also write $\bar\nabla$ for the flat connection on $\mathbb{R}^3$.

Adopting assumption (A3), we have the Young-Laplace equation:
$$
\sigma_0 H_\iota=p_i-p_e,
$$
where $H_\iota$ is the (scalar) mean curvature of the embedding, $\sigma_0$ is the surface tension coefficient (which is assumed to be a constant), and $p_i,p_e$ are respectively the inner and exterior air pressure at the boundary; they are scalar functions on the boundary and we assume that $p_e$ is a constant. Under assumptions (A1) and (A2), we obtain Bernoulli's equation, sometimes referred as the pressure balance condition, on the evolving surface:
\begin{equation}\label{BernEq}
\left.\frac{\partial\Phi}{\partial t}\right|_{M_t}+\frac{1}{2}|\bar\nabla\Phi|_{M_t}|^2-p_e=\frac{\sigma_0}{\rho_0}H_\iota,
\end{equation}
where $\Phi$ is the velocity potential of the velocity field of the air. Note that $\Phi$ is determined up to a function in $t$, so we shall leave the external (constant) pressure $p_e$ around for convenience reason. According to assumption (A1), the function $\Phi$ is a harmonic function within the region $\Omega_t$, so it is uniquely determined by its boundary value, and the velocity field within $\Omega_t$ is $\bar\nabla\Phi$. The kinetic equation on the free boundary $M_t$ is naturally obtained as
\begin{equation}\label{VelEq}
\frac{\partial\iota}{\partial t}\cdot N(\iota)
=\bar\nabla\Phi|_{M_t}\cdot N(\iota).
\end{equation}

We would like to discuss the conservation laws for (\ref{BernEq})-(\ref{VelEq}). The conservation of volume $\text{Vol}(\Omega_t)=\text{Vol}(\Omega_0)$ is a consequence of incompressibility. Since the flow is Eulerian without any external force, the center of mass must move at a uniform speed along a fixed direction, i.e.
\begin{equation}\label{CenterofMass}
\frac{1}{\mathrm{Vol}(\Omega_0)}\int_{\Omega_t}Pd\mathrm{Vol}(P)=V_0t+C_0,
\end{equation}
with Vol being the Lebesgue measure, $P$ marking points in $\mathbb{R}^3$, $V_0$ and $C_0$ being the velocity and starting position of center of mass respectively. Furthermore, the total momentum is conserved, and since the flow is a potential incompressible one, the conservation of total momentum is expressed as
\begin{equation}\label{ConsMomentum}
\int_{M_t}\rho_0\Phi N(\iota)d\mathrm{Area}(M_t)\equiv\rho_0\mathrm{Vol}(\Omega_0) V_0.
\end{equation}
Most importantly, as Zakharov pointed out in \cite{Zakharov1968}, (\ref{BernEq})-(\ref{VelEq}) is a Hamilton system, with Hamiltonian
\begin{equation}\label{Hamiltonian}
\sigma_0\mathrm{Area}(\iota)+\frac{1}{2}\int_{\Omega_t}\rho_0|\bar\nabla\Phi|^2d\mathrm{Vol}
=\sigma_0\mathrm{Area}(M_t)+\frac{1}{2}\int_{M_t}\rho_0\Phi|_{M_t}\left(\bar\nabla\Phi|_{M_t}\cdot N(\iota)\right) d\mathrm{Area},
\end{equation}
i.e. potential energy proportional to surface area plus kinetic energy of the fluid.

We explain why the scenario of oscillating almost spherical water drop is of particular interest. If the system is static, then the kinetic equation (\ref{VelEq}) implies that the outer normal derivative of the velocity potential $\Phi$ is zero, so $\Phi\equiv\text{const}$. The Bernoulli equation on boundary (\ref{BernEq}) then implies that the embedded surface $\iota(M)$ has constant mean curvature. Although the topology of $M$ is not prescribed, the famous rigidity theorem of Alexandrov asserts that $\iota(M)$ must be an Euclidean sphere (see for example \cite{MP2019}). Thus, Euclidean ball is the \emph{only} static configuration for the physical system, and oscillating almost spherical water drop is the \emph{only} possible perturbative oscillation near a static configuration. We emphasize that surface tension plays an important role: since there is no gravity in presence, it is the only constraining force preventing the fluid from spreading in the space. This is in contrast to the familiar gravity or gravity-capillary waves oscillating near the horizontal level.

\subsection{Spherical Capillary Water Waves Equation}
It is not hard to verify that the system (\ref{BernEq})-(\ref{VelEq}) is invariant if $\iota$ is composed with a diffeomorphism of $M$. We may thus regard it as a \emph{geometric flow}. If we are only interested in perturbation near a given configuration, we may reduce system (\ref{BernEq})-(\ref{VelEq}) to a non-degenerate dispersive differential system concerning two scalar functions defined on $M$, just as Beyer and G\"{u}nther did in \cite{BeGu1998}. In fact, during a short time of evolution, the interface can be represented as the graph of a function defined on the initial surface: if $\iota_0:M\to\mathbb{R}^3$ is a fixed embedding close to the initial embedding $\iota(0,x)$, we may assume that $\iota(t,x)=\iota_0(x)+\zeta(t,x)N_0(x)$, where $\zeta$ is a scalar ``height" function defined on $M_0$ and $N_0$ is the outer normal vector field of $M_0$. See Figure \ref{Height}.

With this observation, we shall transform the system (\ref{BernEq})-(\ref{VelEq}) into a non-local system of two real scalar functions $(\zeta,\phi)$ defined on $M$, where $\zeta$ is the ``height" function described as above, and $\phi(t,x)=\Phi(t,\iota(t,x))$ is the boundary value of the velocity potential, pulled back to the underlying manifold $M$.

\begin{figure}[h]\label{Height}
\centering
\includegraphics[width=0.4\textwidth,angle=0]{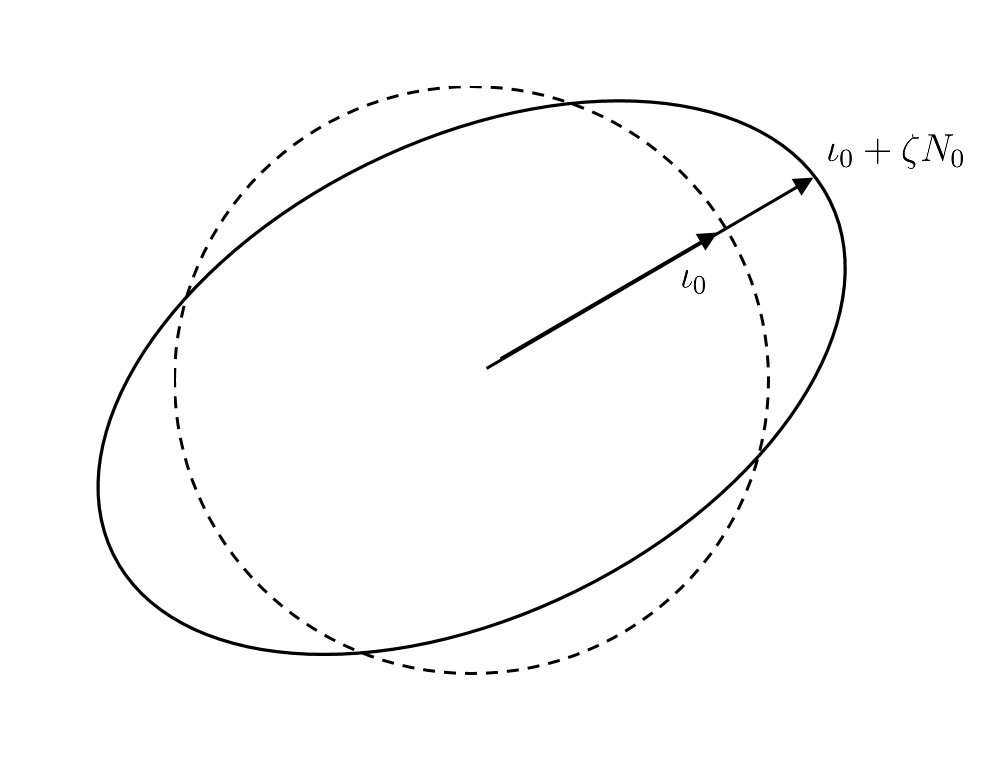}
\caption{The surface defined by a height function.}
\end{figure}

Set
$$
B_\zeta:\phi\to\bar\nabla\Phi|_{M_t}
$$
to be the operator mapping the pulled-back Dirichlet boundary value $\phi$ to the boundary value of the gradient of $\Phi$. Define the \emph{Dirichlet-Neumann operator} $D[\zeta]\phi$ corresponding to the region enclosed by the image of $\iota=\iota_0+\zeta N_0$ as the weighted outer normal derivative:
$$
D[\zeta]\phi:=\frac{\bar{\nabla}\Phi\cdot N(\iota)}{N_0\cdot N(\iota)}.
$$
Thus the kinetic equation (\ref{VelEq}) becomes
$$
\frac{\partial\zeta}{\partial t}=D[\zeta]\phi.
$$

We also need to calculate the restriction of $\partial_t\Phi$ on $M_t$ in terms of $\phi$ and $\iota$. By the chain rule,
$$
\begin{aligned}
\left.\frac{\partial\Phi}{\partial t}\right|_{M_t}
&=\frac{\partial\phi}{\partial t}-\bar\nabla\Phi|_{M_t}\cdot\frac{\partial\iota}{\partial t}\\
&=\frac{\partial\phi}{\partial t}-\left(\bar\nabla\Phi|_{M_t}\cdot N_0\right)\frac{\partial\zeta}{\partial t}\\
&=\frac{\partial\phi}{\partial t}-\left(\bar\nabla\Phi|_{M_t}\cdot N_0\right)\cdot D[\zeta]\phi.
\end{aligned}
$$
We thus arrive at the following nonlinear system:
\begin{equation}\label{EQ(M)}\tag{EQ(M)}
\left\{
\begin{aligned}
\frac{\partial\zeta}{\partial t}
&=D[\zeta]\phi,\\
\frac{\partial\phi}{\partial t}
&=\left(B_\zeta\phi\cdot N_0\right)\cdot D[\zeta]\phi-\frac{1}{2}|B_\zeta\phi|^2+\frac{\sigma_0}{\rho_0}H(\zeta)+p_e,
\end{aligned}
\right.
\end{equation}
where $H(\zeta)$ is the (scalar) mean curvature of the surface given by the height function $\zeta$.

For $M=\mathbb{S}^2$, the case that we shall discuss in detail, we name the system as \emph{spherical capillary water waves equation}. To simplify our discussion, we will be working under the center of mass frame, and require the mean of $\phi$ vanishes for all $t$. This could be easily accomplished by absorbing the mean into $\phi$ since the equation is invariant by a shift of $\phi$. The quadratic terms in the right-hand-side of (\ref{EQ(M)}) are also computed explicitly using the Riemann connection $\nabla_0$ on $\mathbb{S}^2$ corresponding to the standard spherical metric $g_0$. In a word, from now on, we will be focusing on the non-dimensional capillary spherical water waves equation
\begin{equation}\label{EQ}\tag{EQ}
\left\{
\begin{aligned}
\frac{\partial\zeta}{\partial t}
&=D[\zeta]\phi,\\
\frac{\partial\phi}{\partial t}
&=-\frac{|\nabla_0\phi|_{g_0}^2}{2(1+\zeta)^2}
+\frac{\left((1+\zeta)^2D[\zeta]\phi+\nabla_0\zeta\cdot\nabla_0\phi\right)^2}{2\big((1+\zeta)^2+|\nabla_0\zeta|_{g_0}^2\big)(1+\zeta)^2}
+\frac{\sigma_0}{\rho_0}H(\zeta)+p_e.
\end{aligned}
\right.
\end{equation}
By suitable spatial scaling, we may assume that $\sigma_0/\rho_0=1$, the total volume of the fluid is $4\pi/3$, and $p_e=-2$, so that the conservation of volume is expressed as
\begin{equation}\label{Vol}
\frac{1}{3}\int_{\mathbb{S}^2}(1+\zeta)^3d\mu_0\equiv\frac{4\pi}{3},
\end{equation}
where $\mu_0$ is the standard measure on $\mathbb{S}^2$. The inertial movement of center of mass (\ref{CenterofMass}) and conservation of total momentum (\ref{ConsMomentum}) under our center of mass frame are expressed respectively as
\begin{equation}\label{SphereCons.}
\int_{\mathbb{S}^2}(1+\zeta)^4N_0d\mu_0=0,
\quad
\int_{\mathbb{S}^2}\phi N(\iota)d\mu(\iota)=0,
\end{equation}
where $\mu(\iota)$ is the induced surface measure. Further, the Hamiltonian of the system is
\begin{equation}\label{SHamilton}
\mathbf{H}[\zeta,\phi]=\mathrm{Area}(\iota)+\frac{1}{2}\int_{\mathbb{S}^2}\frac{\phi\cdot D[\zeta]\phi}{N_0\cdot N(\iota)} d\mu(\iota),
\end{equation}
and for a solution $(\zeta,\phi)$ there holds $\mathbf{H}[\zeta,\phi]\equiv\text{Const.}$.

The system (\ref{EQ}) resembles the well-known Zakharov-Craig-Schanz-Sulem formulation of gravity-capillary water waves equation \cite{CSS1992}. Formal linearization of (\ref{EQ}) around the (unique) static solution $(\zeta,\phi)=(0,0)$ indicates that the system is a dispersive one of order 3/2. In fact, if we denote by $\mathcal{E}_n$ the space of degree $n$ spherical harmonics, then the Dirichlet-Neumann operator $D[0]$ acts as the multiplier $n$ on $\mathcal{E}_n$, and the linearization $H'(0)$ acts as the multiplier $-(n-2)(n+1)$ on $\mathcal{E}_n$. Consequently, with $\Pi^{(n)}$ being the projection to $\mathcal{E}_n$, we can introduce the diagonal unknown 
\begin{equation}\label{Formal1}
u=\Pi^{(0)}\zeta+\Pi^{(1)}\zeta
+\sum_{n\geq2}\sqrt{(n-1)(n+2)}\cdot\Pi^{(n)}\zeta+i\sum_{n\geq1}\sqrt{n}\cdot\Pi^{(n)}\phi,
\end{equation}
so that the linearization (\ref{EQ}) around the static solution is a linear dispersive equation
\begin{equation}
\frac{\partial u}{\partial t}+i\Lambda u=0,
\end{equation}
where the $3/2$-order elliptic operator $\Lambda$ is given by a multiplier
\begin{equation}\label{DispersiveRelation}
\Lambda=\sum_{n\geq2}\sqrt{n(n-1)(n+2)}\cdot \Pi^{(n)}=:\sum_{n\geq0}\Lambda(n)\Pi^{(n)}.
\end{equation}
The system (\ref{EQ}) for small amplitude oscillation can be formally re-written as
\begin{equation}\label{Formal4}
\frac{\partial u}{\partial t}+i\Lambda u=\mathfrak{N}(u),
\end{equation}
with $\mathfrak{N}(u)=O(u^{\otimes2})$ vanishing quadratically as $u\simeq0$. In \cite{Shao2022}, the author provided formal analysis regarding small amplitude solutions and pointed out obstructions in understanding the long time behaviour.

From a hydrophysical point of view, the linearized equation of (\ref{EQ}) and the dispersive relation (\ref{DispersiveRelation}) dates back to Lord Rayleigh \cite{Rayleigh1879} (see also the well-known textbook \cite{Lamb1932}, Section 275). It has been studied by several widely cited papers on hydrodynamics \cite{TB1983} \cite{NB1986} \cite{LM1988}, regarding the following topics: Poincar\'{e}-Lindstedt series of periodic solutions (assuming its existence, which is not guaranteed), possible ``chaotic" behaviour, and numerical simulation.

From a mathematical point of view, it is already known that the general free-boundary problem for Euler equation is locally well-posed, due to the work of Coutand-Shkoller \cite{CoSh2007} and Shatah-Zeng \cite{ShZe2008}. The curl equation ensures that if the flow is curl free at the beginning, then it remains so during the evolution. This justifies the local well-posedness of Cauchy problem of (\ref{EQ}), at least for sufficiently regular initial data. On the other hand, Beyer-G\"{u}nther \cite{BeGu1998} also showed that the Cauchy problem of (\ref{EQ}) is locally well-posed, without referring to general free-boundary problem. They transformed the problem into an ODE problem defined on graded Hilbert spaces. See also \cite{Lannes2005} and \cite{MingZhang2009} for proof of local well-posedness for water waves near horizontal level using Nash-Moser type theorems. The potential-theoretic approach of Wu \cite{Wu1997}-\cite{Wu1999} may also be transplanted to this case.

In summary, the Cauchy problem for (\ref{EQ}) is locally well-posed, at least for sufficiently regular initial data. But this is all we can assert for the motion of a water droplet under zero gravity. None of the mathematical results mentioned above goes beyond the time regime guaranteed by energy estimate. 

By a formal analysis conducted in \cite{Shao2022}, the author observed that the geometry of the sphere plays a fundamental role in the long-time behaviour of (\ref{EQ}), and leads to very different oscillatory behaviour compared to gravity or gravity-capillary water waves near horizontal level. To analyze such oscillations in detail, we need to take into account the spectral properties of the sphere. This calls for the development of new analytic tools that explicitly reflects the geometry of the underlying space.

\subsection{Toolbox of Para-differential Calculus on Compact Lie Groups}
Our goal is to develop an analytic toolbox for PDEs on compact Lie groups and homogeneous spaces. In particular, the toolbox should enable us to reduce the spherical capillary water waves system (\ref{EQ}) to a para-differential form convenient for the study of its long-time behaviour. 

Following \cite{Fis2015}, we transplant the sequence of ideas for para-differential calculus on $\mathbb{R}^n$ to a compact Lie group $\Gp$ (see for example Chapter 8-10 of \cite{Hormander1997}). The Fourier modes $e^{i\xi\cdot x}$ are replaced by irreducible unitary representations $\xi$, so that a symbol $a(x,\xi)$ is a field acting on the representation space of $\xi$. Due to unbounded degree of degeneracy of Laplace eigenfunctions (see the Weyl type lemma \ref{Weyl}), it is necessary to consider symbols as \emph{matrices} acting on representation spaces instead of \emph{scalar-valued functions}. 

We start by defining symbols and symbol classes on $\Gp$ in Section \ref{2}, then Littlewood-Paley decomposition in Section \ref{3}. Littlewood-Paley characterization for Sobolev and Zygmund space will be given here. The Stein theorem \ref{SteinTheorem} for pseudo-differential operators of type $(1,1)$ will be proved in Section \ref{3}. The spectral condition is introduced in Section \ref{3}, and Theorem \ref{Stein'} establishes the arbitrary $H^s$ boundedness for operators satisfying a spectral condition. Para-product estimates and Bony's para-linearization theorem are deduced as corollaries. In Section \ref{4}, para-differential operators are defined corresponding to rough symbols $a(x,\xi)$ with $C^r_*$ regularity in $x$. Formulas for symbolic calculus, i.e. composition and adjoint formulas are then proved for para-differential operators in Theorem \ref{Compo1}-\ref{ParaAdj}. That the commutator of para-differential operators of order $m$ and $m'$ is of $m+m'-1$ follows as a direct consequence.

None of the steps listed above are trivial. Let us discuss the major difficulties. As mentioned before, unlike Fourier modes on commutative groups, the product of two Laplace eigenfunctions on a non-commutative Lie group can \emph{never} be a Laplace eigenfunction. Furthermore, the dual of $\Gp$ does not form a group anymore, so the definition of ``differentiation" for a symbol $a(x,\xi)$ with respect to $\xi$ requires extra effort; see Section \ref{2} for the details. The spectral localization property for eigenfunctions on $\Gp$ will serve as the substitute for $e^{i\xi_1\cdot x}e^{i\xi_2\cdot x}=e^{i(\xi_1+\xi_2)\cdot x}$. In order to prove $L^2$-boundedness for operators satisfying a spectral condition, we actually need some information for the \emph{lower spectral bound} of the product. Corollary \ref{SpecPrd}, which has never appeared in previous literature on representation theory to the author's knowledge, explicitly describes such bounds with the aid of highest weight theory. The role it plays in establishing arbitrary $H^s$ boundedness for para-differential operators is clearly indicated during the proof of Theorem \ref{Stein'}. 

Additionally, functions of symbols thus cannot be manipulated so easily as in symbolic calculus on $\mathbb{R}^n$. We introduce the notion of \emph{quasi-homogeneous symbols} in Definition \ref{QuasiHomoSym}. They are essentially functions of classical differential symbols, hence still enjoys, to some extent, the good properties of the latter. It seems that this class of quasi-homogeneous symbols is broad enough to cover operators arising from PDEs that we are concerned with.

We emphasize that the algebraic structure of $\Gp$ is fundamental throughout our para-differential calculus. It enables us to write down the precise form of composition and adjoint operators. To this extent, this paper, together with the up-coming one (which will be the second split from \cite{Shao2023}), provide a novel formalism for non-local, \emph{quasi-linear} dispersive partial differential equations defined on certain manifolds with high symmetry. It can be expected that the formalism may be applied to other quasi-linear dispersive problems besides the spherical capillary waves equation.

\subsection{Para-differential Calculus Applied to the Spherical Water Waves Problem}
After presenting a toolbox of para-differential calculus on compact Lie groups in Section \ref{2}-\ref{4}, we claim a sequence of para-differential formulas based on this toolbox. The first one in them is the para-linearization formula for the Dirichlet-Neumann operator $D[\zeta]\phi$:
\begin{theorem}\label{Thm1}
Fix $s>3$. Suppose $\zeta\in H^{s+0.5}$, $\phi\in H^s$. Let $\nabla_0$ be the standard connection on $\mathbb{S}^2$, and let $\nabla_{G_0}$ be the standard connection on $\SU(2)$. Introduce the quantities
$$
\mathfrak{b}=\left(1+\frac{|\nabla_{0}\zeta|^2}{(1+\zeta)^{2}}\right)^{-1}\left(
D[\zeta]\phi+\frac{\nabla_{0}\zeta\cdot\nabla_{0}\phi}{(1+\zeta)^{2}}\right),
\quad
\mathfrak{v}=\frac{\nabla_0\phi-\mathfrak{b}\nabla_0\zeta}{(1+\zeta)^2},
$$
so that $\mathfrak{b}$ is the velocity of the boundary along the radial direction, $\mathfrak{v}$ is the velocity of the boundary along the tangential direction of the unperturbed sphere. The lift of $D[\zeta]\phi$ to $\SU(2)$ is para-linearized as follows:
$$
\big(D[\zeta]\phi\big)^\sharp
=T_{\lambda}\big(\phi^\sharp-T_{\mathfrak{b}^\sharp}\zeta^\sharp\big)
-T_{\mathfrak{v}^\sharp}\cdot\nabla_{G_0}\zeta^\sharp
+f_1(\zeta,\phi).
$$
The definition of lifting to $\SU(2)$ is given by (\ref{Fsharp}). The symbol $\lambda$ is defined in  (\ref{DNlambda}), satisfying
$$
T_\lambda\simeq |\nabla_{G_0}|.
$$
The error term $f_1(\zeta,\phi)$ projects to a function on $\mathbb{S}^2$ with $H^{s+0.5}$ norm controlled by $\big(\|\zeta\|_{H^{s+0.5}}+\|\phi\|_{H^s}\big)^2$ when $(\zeta,\phi)\simeq0$.
\end{theorem}

Section \ref{6} is devoted to the proof of Theorem \ref{Thm1}. The idea of the proof is quite straightforward: one factorizes the Laplacian as the product of two first order elliptic operators near the boundary, and extract the boundary value. Using this para-linearization formula, the system (\ref{EQ}) is converted to a quasi-linear para-differential system of order 1.5, resembling the formal linearization (\ref{Formal1})-(\ref{Formal4}):

\begin{theorem}\label{Thm2}
Suppose $s>3$, and $\zeta\in C^0_tH_x^{s+0.5}$, $\phi\in C^0_tH^s_x$, so that $\|\zeta\|_{H^{s+0.5}_x}$ remains small. Keep the notation of Theorem \ref{Thm1}, and introduce the ``good unknown" $w^\sharp=\phi^\sharp-T_{\mathfrak{b}^\sharp}\zeta^\sharp$. The spherical capillary water waves system (\ref{EQ}) is then equivalent to a quasi-linear para-differential system on $\SU(2)$:
\begin{equation}\label{EQSymm}
\partial_t\left(\begin{matrix}
T_p\zeta^\sharp \\ T_q w^\sharp \end{matrix}\right)
=\left(\begin{matrix}
 & T_\gamma \\
-T_\gamma & 
\end{matrix}\right)\left(\begin{matrix}
T_p\zeta^\sharp \\ T_q w^\sharp \end{matrix}\right)
-T_{\mathfrak{v}^\sharp}\cdot\nabla_{G_0}\left(\begin{matrix}
T_p\zeta^\sharp \\ T_q w^\sharp \end{matrix}\right)
+f_3(\zeta,\phi).
\end{equation}
Here 
\begin{itemize}
\item The para-differential operator $T_p\simeq|\nabla_{G_0}|^{0.5}$, and $q$ is a scalar function;
\item The ``dispersive relation" $T_\gamma\simeq|\nabla_{G_0}|^{1.5}$ is almost self-adjoint:
$$
T_\gamma-T_\gamma^*\text{ has order }0;
$$
\item The ``transporting operator" $T_{\mathfrak{v}^\sharp}\cdot\nabla_{G_0}$ has purely imaginary principal symbol, so taking into account the almost self-adjointness of $T_\gamma$, (\ref{EQSymm}) is of dispersive type;
\item The error $f_3$ is quadratic: when $(\zeta,\phi)\simeq0$ in $H^{s+0.5}_x\times H^s_x$, $\|f_3(\zeta,\phi)\|_{H^s_x}\lesssim\left(\|\zeta\|_{H_x^{s+0.5}}+\|\phi\|_{H_x^s}\right)^2$.
\end{itemize}
\end{theorem}

Proof of Theorem \ref{Thm2} occupies Section \ref{7}. An initial application of Theorem \ref{Thm2} is the following local well-posedness theorem for (\ref{EQ}):
\begin{theorem}\label{Thm3}
Fix $s>3$. Suppose the initial data for the system (\ref{EQ}) satisfies $(\zeta(0),\phi(0))\in H^{s+0.5}\times H^s$, sufficiently small in norm. Then there is a time interval $[0,T]$, on which the Cauchy problem of (\ref{EQ}) admits a unique solution
$$
(\zeta,\phi)\in \left(C_TH^{s+0.5}_x\times C_TH^{s}_x\right)
\cap \left(C^1_TH^{s-1}_x\times C_TH^{s-1.5}_x\right).
$$
\end{theorem}
Proof of Theorem \ref{Thm3} is only sketched in Section \ref{8}, as the key computations for this are quite parallel to the counterparts in \cite{ABZ2011}; we confine ourselves elaborating what is different. Theorem \ref{Thm3} shares the same regime of regularity as the main result in \cite{ABZ2011} does for nearly horizontal water waves. The requirement on regularity is obviously milder than that in \cite{BeGu1998}, which is $(\zeta,\phi)\in H^{5.5}\times H^{5}$; or that in \cite{CoSh2007}, which is $(\zeta,\phi)\in H^{6.5}\times H^{6}$; or that in \cite{ShZe2008}, which is $(\zeta,\phi)\in H^7\times H^{6.5}$. However, this should not be surprising, as the results in \cite{ShZe2008} and \cite{CoSh2007} apply to \emph{general} motion of perfect fluid, without assuming irrotationality, while the argument in \cite{BeGu1998} deals with ODE on graded spaces so that some information will certainly be lost.

The formalism of our para-differential calculus is different from previous ones, which will be briefly reviewed below. In this paper, we extensively utilize the symmetry of the 2-sphere, on which (\ref{EQ}) is defined. The 2-sphere is considered as a $\SU(2)$-homogeneous space via Hopf fiberation, and the Cauchy problem of (\ref{EQ}) is lifted to $\SU(2)$, a compact Lie group. We then employ the para-differential calculus for rough symbols on compact Lie groups proposed in \cite{Shao2023Toolbox}. The advantage of this approach is that the lower order symbols, being non-neglible for quasi-linear problems of order $>1$, are manipulated in an explicit, global, coordinate-independent way. 

With para-differential calculus constructed, the para-linearization and symmetrization of (\ref{EQ}) becomes parallel to the gravity-capillary water waves problem for the flat water level. This procedure is accomplished in Section \ref{6}-\ref{7}, with method very similar to \cite{ABZ2011}. The system is then reduced to a quasi-linear, non-local, dispersive para-differential equation of order 3/2, with precise form given in formula (\ref{EQSymm}). The local existence result, sketched in Section \ref{8}, then follows from a standard fixed point argument. 

The complexity of this para-differential reduction raises the natural question of whether it is necessary for the study of the water drop problem at all. Given the previous local well-posedness results already mentioned, such reduction does not seem unavoidable, if one is concerned with the local Cauchy theory. For example, if one employs Delort's coordinate-free definition \cite{Delort2015} of para-differential operators on manifolds, it is expected that the same regime of regularity remains valid for the local theory for initial fluid-air interface with higher genus. Anyway, it would be hard to imagine that the regime of short-time regularity depends on the topology of the interface.

However, previous approaches of para-differential calculus may fail to incorporate spectral properties of the sphere, and thus may cause extra difficulty in understanding long-time solutions. For example, no symbolic calculus is developed in \cite{KR2006} at all, since no additional structure beyond a compact surface is assumed. Both constructions in \cite{Delort2015} and \cite{BGdP2021} result in \emph{calculus involving principal symbol alone}, similar to H\"{o}rmander calculus on curved manifolds. These approaches thus may not capture the lower order symbols in quasi-linear problems that are crucial for long-time evolution. Furthermore, much of the spectral information related to dispersive properties become implicit, although not lost, under the frameworks of them. 

It then seems feasible to utilize the representation-theory-based approach to para-differential calculus, since it clearly reflects the spectral properties of the dispersive system (\ref{EQ}), while giving all the lower-order terms in symbolic calculus very explicit, manipulable form. Evidence of this convenience is obvious, if one takes into account the extensive volume of literature on gravity-capillary water waves with periodic boundary condition. People usually reduce the gravity-capillary water waves equation to a para-differential form, and use it to obtain global or almost global results. See for example \cite{DIPP2017}, \cite{BD2018}, \cite{IoPu2019}, or the review \cite{IP2018}. Given that (\ref{EQSymm}) is quite similar to the para-differential form of gravity-capillary water waves equation, we regard it as the starting point for the study of long-time behaviour of small amplitude solutions of (\ref{EQ}).

\subsection{Spectral Localization and Dispersive PDEs on Compact Manifolds}
The study of nonlinear dispersive partial differential equations and harmonic analysis are closely linked to each other and mutually reinforcing. If discussed on a compact manifold, the geometry of the underlying space shall also come on the scene. We will be exploring an aspect of this topic in this paper. In order to extract the key ideas out of several important examples, we start our discussion by reviewing some previous results to illustrate the intersectionality of harmonic analysis on compact manifolds and nonlinear dispersive PDEs.

In the pioneering work \cite{BGT2004}, Burq, G\'{e}rad and Tzvetkov obtained Strichartz estimates for the unitary group $e^{it\Delta_g}$ on a compact Riemannian manifold $(M,g)$, and discussed its application to the Cauchy problem of semi-linear Schr\"{o}dinger equation on compact Riemannian manifold. The discussion relies on ``$L^p$ decoupling of Laplace eigenfunctions", discovered by Sogge \cite{Sogge1988}. As a result, they showed that on a compact Riemann surface $(M,g)$, the Cauchy problem of the cubic Schr\"{o}dinger equation 
$$
i\partial_tu+\Delta_gu=|u|^2u,
\quad
u(0)\in H^s(M)
$$
with $s>1/2$ is locally well-posed in $C^0_tH^s(M)$. They further showed in \cite{BGT2005} that the regime of local well-posedness can be extended to $s>1/4$ if $(M,g)$ is a Zoll surface (in particular, if it is the standard 2-sphere). A crucial harmonic analysis result needed by this is a bilinear eigenfunction estimate on Zoll surfaces, an inequality bounding the $L^2$ norm of the product of two eigenfunctions.

Hani continued the study in \cite{Hani2012}, showing that the Cauchy problem of the cubic Schr\"{o}dinger equation on $(M,g)$ is globally well-posed for $s>2/3$, regardless of the geometry. The proof is a modification of the ``I-method" due to Colliander-Keel-Staffilani-Takaoka-Tao (see \cite{CKSTT2002}). When transplanting the I-method from $\mathbb{R}^n$ or $\mathbb{T}^n$ to a general manifold, an immediate issue appears: one replaces Fourier analysis by spectral decomposition using Laplace eigenfunctions, but unlike Fourier modes, the product of two Laplace eigenfunctions is, in general, not an eigenfunction anymore. In order to overcome this difficulty, Hani established \emph{spectral licalization property} as a substitute. That is, the product of two eigenfunctions with eigenvalues $\lambda$ and $\mu$ respectively should sharply concentrate around the range $\sqrt{-\Delta_g}\in\big[|\sqrt{\lambda}-\sqrt{\mu}|,\sqrt{\lambda}+\sqrt{\mu}\big]$. This turns out to be an important ingredient for the proof.

Not surprisingly, the spectral localization property could be improved if certain geometric conditions are posed for the underlying manifold. Such improvement should then lead to stronger results for nonlinear dispersive PDEs. In fact, this is the case for compact Lie groups and homogeneous spaces. Highest weight theory for representations of a compact Lie group implies the following: the product of two Laplace eigenfunctions on a compact Lie group is a \emph{finite} linear combination of Laplace eigenfunctions. In \cite{BP2011} and \cite{BCP2015} (see also \cite{BBP2010}), precise form of this algebraic result plays a fundamental role for the search of periodic or quasi-periodic solutions of the semi-linear Schr\"{o}dinger equation
$$
i\partial_tu+\Delta u+\mu u=f(x,|u|^2)u
$$
on a compact Lie group or homogeneous space. Another example is \cite{DS2004}, where the underlying manifold is the standard sphere $\mathbb{S}^d$. The aforementioned spectral localization on Lie groups then becomes an algebraic property of spherical harmonic functions: the product of spherical harmonics of degree $p$ and $q$ is a linear combination of spherical harmonics of degree between $|p-q|$ and $p+q$. Using this fact, the authors casted a normal form reduction for semi-linear Klein-Gordon equations on $\mathbb{S}^d$.

Furthermore, the spectral localization property plays a fundamental role in the theory of pseudo-differential operators on compact Lie groups. Construction of pseudo-differential calculus on compact Lie groups depends very much on representation theory of compact Lie groups. Being the non-commutative analogy of Fourier theory on $\mathbb{R}^n$ or $\mathbb{T}^n$, spectral localization shows how matrix elements of different representations (``modes of different frequencies") interact. It thus occupies a similar part as the additive formula $e^{i\xi_1\cdot x}e^{i\xi_2\cdot x}=e^{i(\xi_1+\xi_2)\cdot x}$ does in commutative Fourier analysis and symbolic calculus.

Global symbolic calculus on compact Lie groups has been proposed by Ruzhansky-Turunen-Wirth \cite{RT2009}-\cite{RTW2014}, while the ideas seem to date back to Taylor \cite{Taylor1986}. However, as pointed out by Fischer \cite{Fis2015}, the pseudo-differential calculus in \cite{RT2009}-\cite{RTW2014} may lack rigorous justification. Fischer's work \cite{Fis2015} provides a mathematically rigorous and complete construction in which the symbolic calculus formulas are proved. In particular, Fischer proved that pseudo-differential operators of type $(1,0)$ defined via global symbols indeed coincides with the H\"{o}rmander class of $(1,0)$ pseudo-differential operators defined via local coordinates. To the author's knowledge, the first result on \emph{linear} dispersive PDEs on a compact Lie group with essential usage of global symbolic calculus is Bambusi-Langella \cite{BL2022}. The paper states a general frame work concerning ``abstract pseudo-differential operators" and obtained a growth estimate of the solution to \emph{linear} Schr\"{o}dinger equation with possibly time-dependent potential. Taking into account the pseudo-differential calculus constructed by \cite{Fis2015}, linear Schr\"{o}dinger equations on a compact Lie group thus realizes this abstract framework. 

\subsection{Para-differential Calculus on Manifolds} 
When coming to quasi-linear dispersive PDEs, however, some more technicalities beyond spectral localization will be necessarily required. Within our scope, we shall focus on one of them, namely para-differential calculus. 

Para-differential calculus originates from a systematic application of Littlewood-Paley decomposition to nonlinear PDEs. It first appeared in the form of \emph{para-product estimate}, leading to inequalities of the form $\|fg\|_{H^s}\lesssim |f|_{L^\infty}\|g\|_{H^s}+\|f\|_{H^s}|g|_{L^\infty}$ or composition estimates of the form $\|F(u)\|_{H^s}\lesssim \|u\|_{H^s}$. Such inequalities were crucial for local well-posedness of semi-linear dispersive PDEs; see for example, Saut-Temam \cite{SaTe1976} Linares-Ponce \cite{LP2014} for discussion of KdV equation, or Kato \cite{Kato1995} and Staffilani \cite{Staffilani1995} for discussion of semi-linear Schr\"{o}dinger equation. General estimates may also be found in \cite{Taylor2000}.

The defining character of such methodology was extracted as \emph{isolating the part with lowest regularity}. It has been extended to a class of operators called \emph{para-differential operators}; see Bony \cite{Bony1981} and Meyer \cite{Meyer1980}-\cite{Meyer1981}). Modelled on symbols $a(x,\xi)$ that are of limited regularity with respect to the spatial variable $x$, a symbolic calculus is still available for para-differential operators. Para-differential calculus thus becomes a powerful tool when usual pseudo-differential calculus is of limited usage due to low regularity.

Para-differential calculus has been successfully used in the study of quasi-linear PDEs (and also micro-local properties for fully nonlinear ones, see \cite{Bony1981} or H\"{o}rmander \cite{Hormander1997}; but we do not discuss that aspect). Taking fluid dynamics equations as an example, the gravity-capillary water waves equation was re-written into a para-differential form by Alazard-Burq-Zuilly \cite{ABZ2011}, and a local well-posedness result in the regime $s>3.5$ was deduced. A pleasant feature of this para-differential form is that the spectral properties of the dispersive relation for the problem are explicit. Not only is it an improvement of previous local well-posedness results (see for example \cite{Wu1997}-\cite{Wu1999}, \cite{Lannes2005}, \cite{CoSh2007}, \cite{ShZe2008}, \cite{MingZhang2009}), it also reduces the water waves equation into a form convenient for the study of long time behaviour. See the review \cite{IP2018} for detailed discussion, or \cite{BD2018} \cite{DIPP2017} \cite{IoPu2019} as examples of applying this para-differential form.

Being such a powerful tool, it is natural to ask whether para-differential calculus admits generalization to curved manifolds instead of $\mathbb{R}^n$ or $\mathbb{T}^n$. Let us briefly summarize previous attempts towards it. 

In \cite{KR2006}, Klainerman and Rodnianski extended the Littlewood-Paley theory to compact surfaces. They defined Littlewood-Paley decomposition of a tensor field via spectral cut-off of the tensorial Laplacian operator, studied the interaction of different frequencies, and deduced Bernstein type inequalities. This served as a toolbox in proving the $L^2$ bounded curvature conjecture for the Einstein equation. See \cite{KRS2015} for the details. The argument actually does not require para-differential calculus, but only an invariant version of Littlewood-Paley decomposition on compact surfaces. Not surprisingly, these properties can be reproduced and improved for compact Lie groups (see \cite{Fis2015}).

In \cite{BGdP2021}, Bonthonneau, Guillarmou and de Poyferr{\'e} developed a para-differential calculus for any compact manifold via local coordinate charts. This approach is a natural generalization of the H\"{o}rmander calculus on a curved manifold $M$: a $(1,1)$ symbol $a(x,\xi)$ is defined on the cotangent bundle of $M$, as a function that becomes a $(1,1)$ symbol under any local chart; para-differential operators corresponding to rough symbols are also defined by coordinate charts. A key feature is that symbolic calculus, just as the usual H\"{o}rmander calculus does, only preserves information at the order of principle symbol. The authors used this para-differential calculus to study microlocal regularity for an Anosov flow.

On the other hand, a global definition of para-differential operators on a compact manifold is introduced in Delort \cite{Delort2015}. Observing the commutator characterization of usual pseudo-differential operators (see e.g. Beals \cite{Beals1977}), Delort defined para-differential operators on a compact manifold via commutators with vector fields and Laplacian spectral projections. The most important result from this para-differential calculus is that, the commutator of para-differential operators of order $m$ and $m'$ is an operator of $m+m'-1$. With these preparations, Delort successfully obtained a normal form reduction for quasi-linear Hamiltonian Klein-Gordon equations on the standard sphere $\mathbb{S}^d$, and thus obtained extended lifespan estimate beyond the energy estimate regime.

Both constructions in \cite{BGdP2021} and \cite{Delort2015} share the common feature that formulas of symbolic calculus for para-differential operators involve only the \emph{principal symbol}, just as the usual H\"{o}rmander calculus on curved manifolds. Consequently, if one applies either one to a quasi-linear dispersive PDE, much of the spectral information related to the dispersive relation becomes implicit, although not lost. Besides, lower order information apart from the principal symbol is blurred. These disadvantages might cause extra difficulty in studying the long-time behavior of the solution.

\subsection{Explanation of the Toolbox}
Our goal is to develop an analytic toolbox for PDEs on compact Lie groups and homogeneous spaces. In particular, the toolbox should enable us to reduce the spherical capillary water waves system in \cite{Shao2022} to a para-differential form convenient for the study of its long-time behaviour. Observing the advantages of incorporating geometric information, our goal will be accomplished by employing the notion of symbols in \cite{Fis2015} as the narrative formalism for para-differential calculus. 

Thus we transplant the sequence of ideas for para-differential calculus on $\mathbb{R}^n$ to a compact Lie group $\Gp$ (see for example Chapter 8-10 of \cite{Hormander1997}). The Fourier modes $e^{i\xi\cdot x}$ are replaced by irreducible unitary representations $\xi$, so that a symbol $a(x,\xi)$ is a field acting on the representation space of $\xi$. Due to unbounded degree of degeneracy of Laplace eigenfunctions (see the Weyl type lemma \ref{Weyl}), it is necessary to consider symbols as \emph{matrices} acting on representation spaces instead of \emph{scalar-valued functions}. 

We start by defining symbols and symbol classes on $\Gp$ in Section \ref{2}, then Littlewood-Paley decomposition in Section \ref{3}. Littlewood-Paley characterization for Sobolev and Zygmund space will be given here. The Stein theorem \ref{SteinTheorem} for pseudo-differential operators of type $(1,1)$ will be proved in Section \ref{3}. The spectral condition is introduced in Section \ref{3}, and Theorem \ref{Stein'} establishes the arbitrary $H^s$ boundedness for operators satisfying a spectral condition. Para-product estimates and Bony's para-linearization theorem are deduced as corollaries. In Section \ref{4}, para-differential operators are defined corresponding to rough symbols $a(x,\xi)$ with $C^r_*$ regularity in $x$. Formulas for symbolic calculus, i.e. composition and adjoint formulas are then proved for para-differential operators in Theorem \ref{Compo1}-\ref{ParaAdj}. That the commutator of para-differential operators of order $m$ and $m'$ is of $m+m'-1$ follows as a direct consequence.

None of the steps listed above are trivial. Let us discuss the major difficulties. As mentioned before, unlike Fourier modes on commutative groups, the product of two Laplace eigenfunctions on a non-commutative Lie group can \emph{never} be a Laplace eigenfunction. Furthermore, the dual of $\Gp$ does not form a group anymore, so the definition of ``differentiation" for a symbol $a(x,\xi)$ with respect to $\xi$ requires extra effort; see Section \ref{2} for the details. The spectral localization property for eigenfunctions on $\Gp$ will serve as the substitute for $e^{i\xi_1\cdot x}e^{i\xi_2\cdot x}=e^{i(\xi_1+\xi_2)\cdot x}$. In order to prove $L^2$-boundedness for operators satisfying a spectral condition, we actually need some information for the \emph{lower spectral bound} of the product. Corollary \ref{SpecPrd}, which has never appeared in previous literature on representation theory to the author's knowledge, explicitly describes such bounds with the aid of highest weight theory. The role it plays in establishing arbitrary $H^s$ boundedness for para-differential operators is clearly indicated during the proof of Theorem \ref{Stein'}. 

Additionally, functions of symbols thus cannot be manipulated so easily as in symbolic calculus on $\mathbb{R}^n$. We introduce the notion of \emph{quasi-homogeneous symbols} in Definition \ref{QuasiHomoSym}. They are essentially functions of classical differential symbols, hence still enjoys, to some extent, the good properties of the latter. It seems that this class of quasi-homogeneous symbols is broad enough to cover operators arising from PDEs that we are concerned with.

We emphasize that the algebraic structure of $\Gp$ is fundamental throughout our para-differential calculus. It enables us to write down the precise form of composition and adjoint operators. To this extent, this paper provides a novel formalism for non-local, \emph{quasi-linear} dispersive partial differential equations defined on certain manifolds with high symmetry. It can be expected that the formalism may be applied to other quasi-linear dispersive problems besides the spherical capillary waves equation (\ref{EQ}).

\section{Preliminary: Representation of Compact Lie Groups}\label{2}
This section is devoted to preliminaries of representation theory of compact Lie groups. Notation will also be fixed in this section. Results from representation theory can be found in any standard textbook on representation theory, for example \cite{Procesi2007}. Concise descriptions can also be found in \cite{BP2011}.

\subsection{Representation and Spectral Properties}
To simplify our discussion, throughout this paper, we fix $\Gp$ to be a simply connected, compact, semisimple Lie group with dimension $n$ and rank $\varrho$, i.e. the dimension of any maximal torus in $\Gp$. By standard classification results of compact connected Lie groups, this is the essential part for our discussion. It is known from standard theory of Cartan subalgebra that all maximal tori in $\Gp$ are conjugate to each other, so the rank $\varrho$ is uniquely determined. The group $\Gp$ is in fact a direct product of groups of the following types: the special unitary groups $\SU(n)$, the compact symplectic groups $\mathrm{Sp}(n)=\mathrm{Sp}(2n;\mathbb{C})\cap\mathrm{U}(2n)$, the spin groups $\mathrm{Spin}(n)$, and the exceptional groups $G_2,F_4,E_6,E_7,E_8$.

We write $e$ for the identity element of $\Gp$. Let $\mathfrak{g}$ be the Lie algebra of $\Gp$, identified with the space of left-invariant vector fields on $\Gp$. If a basis $X_1,\cdots,X_n$ for $\mathfrak{g}$ is given, then for any multi-index $\alpha\in\mathbb{N}_0^n$, write 
\begin{equation}\label{NormalOrder}
X^\alpha=X_1^{\alpha_1}\cdots X_n^{\alpha_n}
\end{equation}
for the left-invariant differential operator composed by these basis vectors, in the ``normal" ordering. When $\alpha$ runs through all multi-indices, the operators $\{X^\alpha\}$ span the universal enveloping algebra of $\mathfrak{g}$, which is the content of the Poincaré–Birkhoff–Witt theorem; however, this fact is not needed within our scope. 

Since $\Gp$ is semi-simple and compact, the Killing form
$$
\Tr\big(\mathrm{Ad}(X)\circ\mathrm{Ad}(Y)\big),\quad X,Y\in\mathfrak{g}
$$ 
on $\Gp$ is non-degenerate and negative definite, where $\mathrm{Ad}$ is the adjoint representation. A Riemann metric on $\Gp$ can thus be introduced as the opposite of the Killing form. The Laplace operator corresponding to this metric, usually called \emph{Casimir element} in representation theory, is given by
$$
\Delta=\sum_{i=1}^n X_i^2,
$$
where $\{X_i\}$ is any orthonormal basis of $\mathfrak{g}$; the operator is independent from the choice of basis.

Let $\DuGp$ be the dual of $\Gp$, i.e. the set of equivalence classes of irreducible unitary representations of $\Gp$. For simplicity, we do not distinguish between an equivalence class $\xi\in\DuGp$ and a certain unitary representation that realizes it. The ambient space of $\xi\in\DuGp$ is an Hermite space $\Hh[\xi]$, with (complex) dimension $d_\xi$. If we equip the space $\Hh[\xi]$ with an orthonormal basis containing $d_\xi$ elements, then $\xi:\Gp\to\mathrm{U}(\Hh[\xi])$ can be equivalently viewed as a unitary-matrix-valued function $(\xi_{jk})_{j,k=1}^{d_\xi}$.

We shall equip the Lie group $\Gp$ with the normalized Haar measure. For simplicity, we will denote the integration with respect to this measure by $dx$. The \emph{Fourier transform} of $f\in\mathcal{D}'(\Gp)$ is defined by 
$$
\Ft f(\xi):=\int_\Gp f(x)\xi^*(x)dx\in\mathrm{End}(\Hh[\xi]).
$$
Conversely, for every field $a(\xi)$ on $\DuGp$ such that $a(\xi)\in\mathrm{End}(\Hh[\xi])$ for all $\xi\in\DuGp$, the \emph{Fourier inversion} of $a$ is defined by
$$
\check{a}(x)=\sum_{\xi\in\DuGp} d_\xi\Tr\big(a(\xi)\cdot\xi(x)\big),
$$
as long as the right-hand-side converges at least in the sense of distribution. 

Two different norms for $a(\xi)\in\mathrm{End}(\Hh[\xi])$ will be used in this paper, one being the operator norm
$$
\|a(\xi)\|:=\sup_{v\in\Hh[\xi],v\neq0}\frac{|a(\xi)v|}{|v|},
$$
one being the Hilbert-Schmidt norm
$$
\lHS a(\xi)\rHS:=\sqrt{\Tr(a(\xi)\cdot a^*(\xi))}.
$$
To simplify notation, when there is no risk of confusion, we shall omit the dependence of these norms on the representation $\xi$.

We fix the convolution on $\Gp$ to be \emph{right} convolution:
$$
(f*g)(x)=\int_\Gp f(y)g(y^{-1}x)dy.
$$
The only property not inherited from the usual convolution on Euclidean spaces is commutativity. Every other property, including the convolution-product duality, Young's inequality, is still valid. The only modification is that 
$$
\Ft{f*g}(\xi)=(\Ft g\cdot\Ft f)(\xi),
$$
which in general does not coincide with $(\Ft f\cdot\Ft g)(\xi)$ since the Fourier transform is now a matrix.

The Peter-Weyl theorem is fundamental for harmonic analysis on Lie groups:
\begin{theorem}[Peter-Weyl]\label{PeterWeyl}
Let $\mathcal{M}_\xi$ be the subspace in $L^2(\Gp)$ spanned by matrix entries of the representation $\xi\in\DuGp$. Then $\mathcal{M}_\xi$ is a bi-invariant subspace of $L^2(\Gp)$ of dimension $d_\xi^2$, and there is a Hilbert space decomposition
$$
L^2(\Gp)=\bigoplus_{\xi\in\DuGp}\mathcal{M}_\xi.
$$
If $f\in L^2(\Gp)$, the Fourier inversion
$$
(\Ft f)^\vee(x)=\sum_{\xi\in\DuGp} d_\xi\Tr\left(\Ft f(\xi)\cdot\xi(x)\right)
$$
converges to $f$ in the $L^2$ norm, and there holds the Plancherel identity
$$
\|f\|_{L^2(\Gp)}^2=\sum_{\xi\in\DuGp} d_\xi\big\lHS\Ft f(\xi)\big\rHS^2.
$$
\end{theorem}

In the following, standard constructions and results in highest weight theory will be directly cited. They are already collected in Berti-Procesi \cite{BP2011}.

Let $\mathfrak{t}\subset\mathfrak{g}$ be the Lie algebra of a maximal torus of $\Gp$, and write $\mathfrak{t}^*$ for its dual. Let $\{\bm{\alpha}_i\}_{i=1}^\varrho\subset\mathfrak{t}^*$ be the set of \emph{positive simple roots} of $\mathfrak{g}$ with respect to $\mathfrak{t}$. They form a basis for $\mathfrak{t}^*$. The \emph{fundamental weights} $\bm{\varpi}_1,\cdots,\bm{\varpi}_\varrho\in\mathfrak{t}^*$ admit several equivalent definitions, one of them being the unique set of vectors $\{\bm{\varpi}_i\}_{i=1}^\varrho$ satisfying 
$$
(\bm{\varpi}_i,\bm{\alpha}_j)=\frac{1}{2}\delta_{ij}|\bm{\alpha}_j|^2,
\quad
i,j=1,\cdots,\varrho.
$$
Here the inner product on the dual is inherited from the Killing form on $\mathfrak{g}$.

\begin{theorem}[Highest Weight Theorem]
Irreducible unitary representations of $\Gp$ are in 1-1 correspondence with the discrete cone
$$
\mathscr{L}^+:=\left\{\sum_{i=1}^\varrho n_i\bm{\varpi}_i,\, n_i\in\mathbb{N}_0\right\}.
$$
The correspondence assigns an irreducible unitary representation to its highest weight vector.
\end{theorem}
In particular, the trivial representation corresponds to the point $0\in\mathscr{L}^+$. From now on, we shall denote the assignment between representations $\xi\in \DuGp$ and highest weights $\bm{j}\in\mathscr{L}^+$ by $\bm{j}=\bm{J}(\xi)$, or $\xi=\Xi(\bm{j})$. The highest weight theorem implies a complete description of the spectrum of the Laplace operator $\Delta$ of $\Gp$.

\begin{theorem}\label{DeltaSpec}
Each space $\mathcal{M}_\xi$ is an eigenspace of the Laplace operator $\Delta$, with eigenvalue
$$
-|\bm{J}(\xi)+\bm{\varpi}|^2+|\bm{\varpi}|^2,
$$
where $\bm{\varpi}=\sum_{i=1}^\varrho \bm{\varpi}_i$. 
\end{theorem}
We write $\lambda_\xi=|\bm{J}(\xi)+\bm{\varpi}|^2-|\bm{\varpi}|^2$, so that $\{\lambda_\xi\}$ is the spectrum of the positive self-adjoint operator $-\Delta$. Imitating commutative harmonic analysis, we set the size functions $|\xi|:=\sqrt{|\lambda_\xi|}$, $\size[\xi]:=(1+\lambda_\xi)^{1/2}$. The operator $|\nabla|$, which is the square-root of $-\Delta$ defined via spectral analysis, will be frequently used. The values of $\size[\xi]$ are exactly the eigenvalues of $(1-\Delta)^{1/2}$. We will also use the following notation:
\begin{equation}\label{FourierSupp}
\mathfrak{S}[\gamma_1,\gamma_2]:=\left\{\xi\in\DuGp:\gamma_1\leq|\xi|\leq\gamma_2 \right\}.
\end{equation}

For $s\in\mathbb{R}$, the Sobolev space $H^s(\Gp)$ is defined to be the subspace of $f\in\mathcal{D}'(\Gp)$ such that $\|(1-\Delta)^{s/2}f\|_{L^2(\Gp)}<\infty$. By the Peter-Weyl theorem and the above characterization of Laplacian, the Sobolev norm is computed by 
$$
\|f\|_{H^s(\Gp)}^2=\sum_{\xi\in\DuGp} d_\xi\size[\xi]^{2s}\big\lHS\Ft f(\xi)\big\rHS^2.
$$

To proceed further, we need to introduce a notion of partial ordering on $\mathscr{L}^+$. Define the cone 
$$
\mathscr{R}^+:=\left\{\sum_{i=1}^\varrho n_i\bm{\alpha}_i,\, n_i\in\mathbb{N}_0\right\}.
$$
We then introduce the \emph{dominance order} on $\mathscr{L}^+$ as follows: for $\bm{j},\bm{k}\in\mathscr{L}^+$, we set $\bm{j}\prec \bm{k}$ if $\bm{j}\neq \bm{k}$ and $\bm{k}-\bm{j}\in\mathscr{R}^+$, and set $\bm{j}\preceq \bm{k}$ to include the possibility $\bm{j}=\bm{k}$. The relation $\prec$ defines a partial ordering on $\mathscr{L}^+$. With the aid of this root system, we obtain the following characterization of product of Laplace eigenfunctions (see, for example, \cite{Procesi2007}, Proposition 3 of page 345):

\begin{proposition}\label{SpecPrd0}
Let $u_{\bm{j}}\in\mathcal{M}_{\Xi(\bm{j})}$, $u_{\bm{k}}\in\mathcal{M}_{\Xi(\bm{k})}$ be Laplace eigenfunctions corresponding to highest weights $\bm{j},\bm{k}\in\mathscr{L}^+$ respectively. Then the product $u_{\bm{j}}u_{\bm{k}}$ is in the space
$$
\bigoplus_{\bm{l}\in\mathscr{L}^+:\,\bm{l}\preceq \bm{j}+\bm{k}}\mathcal{M}_{\Xi(\bm{l})}.
$$
\end{proposition}
We deduce from this characterization an important spectral localization property for products that will play a fundamental role in dyadic analysis:
\begin{corollary}\label{SpecPrd}
There is a constant $0<c\leq1$, depending only on the algebraic structure of $\Gp$, with the following property. If $u_1\in\mathcal{M}_{\xi_1}$, $u_2\in\mathcal{M}_{\xi_2}$, then the product $u_1u_2\in\bigoplus_{\xi}\mathcal{M}_\xi$, where the range of sum is for
$$
\xi\in\mathfrak{S}\Big[c\big||\xi_1|-|\xi_2|\big|,|\xi_1|+|\xi_2|\Big].
$$
Roughly speaking, in the frequency space, product of two eigenfunctions ``localizes" in between the sum and difference of their frequencies.
\end{corollary}
\begin{proof}
With a little abuse of notation, given a highest weight vector $\bm{j}\in\mathscr{L}^+$, we denote by $\bar{\bm{j}}$ the highest weight vector of $\Xi(\bm{j})^*$, the dual representation, which is also irreducible. It is known from representation theory that the highest weight of $\Xi(\bm{j})^*$ is just $-W_0(\bm{j})$, where $W_0$ is the element in the Weyl group with longest length. Since the Weyl group is orthogonal, there holds $|\bm{j}|=|\bar{\bm{j}}|$. By self-adjointness of $\Delta$, matrix entries of $\Xi(\bar{\bm{j}})$ correspond to the same eigenvalue of $\Delta$ as those of $\Xi(\bm{j})$ do.

Now suppose there are highest weights $\bm{j},\bm{k},\bm{l}\in\mathscr{L}^+$, and let $u_{\bm{j}}$, $u_{\bm{k}}$, $u_{\bm{l}}$ be linear combinations of matrix entries of $\Xi(\bm{j})$, $\Xi(\bm{k})$, $\Xi(\bm{l})$, respectively. From the above we know that e.g. $\bar{u}_{\bm{j}}\in\mathcal{M}_{\Xi(\bar{\bm{j}})}$. By Proposition \ref{SpecPrd0} and Schur orthogonality, the integral
$$
\int_\Gp u_{\bm{j}}u_{\bm{k}}\bar{u}_{\bm{l}} dx
$$
does not vanish only if the following relations are simultaneously satisfied:
$$
\bm{l}\preceq \bm{j}+\bm{k},
\quad \bm{j}\preceq \bar{\bm{k}}+\bm{l},
\quad \bar{\bm{k}}\preceq \bm{j}+\bar{\bm{l}}.
$$
Note that this is only a necessary condition. Denoting the pre-dual vector of $\bm{\alpha}_i$ by $\tilde{\bm{\alpha}}_i$ (i.e. $\tilde{\bm{\alpha}}_i\in\mathfrak{t}$, and $\bm{\alpha}_i(\tilde{\bm{\alpha}}_{i'})=\delta_{ii'}$), these relations are equivalent to
$$
\left.
\begin{aligned}
\bm{l}(\tilde{\bm{\alpha}}_i)&\leq \bm{j}(\tilde{\bm{\alpha}}_i)+\bm{k}(\tilde{\bm{\alpha}}_i)\\
\bm{j}(\tilde{\bm{\alpha}}_i)&\leq \bar{\bm{k}}(\tilde{\bm{\alpha}}_i)+\bm{l}(\tilde{\bm{\alpha}}_i)\\
\bar{\bm{k}}(\tilde{\bm{\alpha}}_i)&\leq \bm{j}(\tilde{\bm{\alpha}}_i)+\bar{\bm{l}}(\tilde{\bm{\alpha}}_i)\\
\end{aligned}
\right\}\quad i=1,\cdots,\varrho.
$$
Due to monotonicity of the eigenvalues with respect to the dominance order, the first inequality implies
$$
|\bm{l}+\bm{\varpi}|^2\leq|\bm{j}+\bm{k}+\bm{\varpi}|^2
$$
so we conclude that $|\bm{l}|\leq|\bm{j}|+|\bm{k}|$. Since $\{\bm{\alpha}_i\}_{i=1}^\varrho$ is a basis for $\mathfrak{t}^*$, it follows that $\{\tilde{\bm{\alpha}}_i\}_{i=1}^\varrho$ is a basis for $\mathfrak{t}$, so the second and third inequalities, together with $|\bm{l}|=|\bar{\bm{l}}|$, imply
$$
|\bm{j}-\bar{\bm{k}}|\lesssim |\bm{l}|.
$$
Using $|\bm{k}|=|\bar{\bm{k}}|$ once again, we obtain $\big||\bm{j}|-|\bm{k}|\big|\lesssim|\bm{l}|$. Thus the integral does not vanish only if 
$$
\big||\bm{j}|-|\bm{k}|\big|\lesssim|\bm{l}|\leq|\bm{j}|+|\bm{k}|.
$$
Using the formula for eigenvalues in Theorem \ref{DeltaSpec}, with $\bm{j}=\bm{J}(\xi_1)$, $\bm{k}=\bm{J}(\xi_2)$, $\bm{l}=\bm{J}(\xi)$, this is implies that the spectral projection of $u_{\bm{j}}u_{\bm{k}}$ to $\mathcal{M}_\xi$ is nontrivial only if
$$
\big||\xi_1|-|\xi_2|\big|
\lesssim |\xi|
\leq |\xi_1|+|\xi_2|.
$$
This is exactly what to prove.
\end{proof}
The dual representation of $\xi\in\DuGp$ is, in general, not equivalent to $\xi$, although they have the same dimension. This is why $\bar{\bm{j}}$ should be studied first. But the rank one case can be treated trivially: the special unitary group $\SU(2)$ is the only simply connected compact Lie group of rank 1. Irreducible unitary representations of $\SU(2)$ are uniquely labeled by their dimensions. For the Lie algebra $\mathfrak{su}(2)$ and the subalgebra corresponding to the subgroup $\mathrm{diag}(e^{i\psi},e^{-i\psi})$, the dominance order is equivalent to the usual ordering on natural numbers. Corollary \ref{SpecPrd} then becomes a well-known fact: the product of spherical harmonics of degree $p$ and $q$ is a linear combination of spherical harmonics of degree between $|p-q|$ and $p+q$, with coefficients given by the \emph{Clebsch–Gordan coefficients} (see e.g. \cite{Vilenkin1978}). This fact has been used by \cite{DS2004} as the substitute of spectral localization properties on spheres. On the other hand, the constant $c$ can be strictly less than 1. For example, the root system of $\SU(3)$ forms a hexagonal lattice, and $c=1/2$ in that case.

In estimating the magnitude of certain symbols, we will also be employing the following asymptotic formula of Weyl type. It is the discrete version of volume growth estimate for Euclidean balls:
\begin{lemma}\label{Weyl}
As $t\to\infty$,
$$
\sum_{\xi\in \mathfrak{S}[0,t]}d_\xi^2
\simeq
t^n.
$$
\end{lemma}
\begin{proof}
By Schur orthogonality, recalling $|\xi|=\sqrt{\lambda_\xi}$, we find that 
$$
\sum_{\xi\in \mathfrak{S}[0,t]}d_\xi^2
=\sum_{|\xi|\leq t}\dim\mathcal{M}_\xi
=\sum_{\lambda\in\sigma(-\Delta):\lambda\leq t^2}1,
$$
i.e. the number of eigenvalues of $-\Delta$ with magnitude less than $t^2$ (counting multiplicity). Using the Weyl law, we find that the ratio between
$$
\sum_{|\xi|\leq t}d_\xi^2
\quad\text{and}\quad
\frac{\omega_n}{(2\pi)^n}\mathrm{Vol}(\Gp)t^n
$$
has limit 1 as $t\to\infty$. Here $\omega_n$ is the volume of Euclidean $n$-unit ball.
\end{proof}

\subsection{Symbol and Symbolic Calculus}
A global notion of symbol can be defined on Lie groups due to their high symmetry. An explicit symbolic calculus was formally constructed by a series of works of Ruzhansky, Turunen, Wirth. Fischer's work \cite{Fis2015} provided a complete study of the object. We will closely follow \cite{Fis2015}.

We fix our compact Lie group $\Gp$ as in the previous subsection. The starting point will be a Taylor's formula on $\Gp$.
\begin{proposition}[Taylor expansion]\label{TaylorGp}
Let $q=(q_1,\cdots,q_M)$ be an $M$-tuple of smooth functions on $\Gp$, all vanishing at $e\in\Gp$, such that $(\nabla q_i)_{i=1}^M$ has rank $n$. For a multi-index $\alpha\in\mathbb{N}_0^m$, set $q^\alpha:=q_1^{\alpha_1}\cdots q_M^{\alpha_M}$. Corresponding to each multi-index $\alpha$, there is a left-invariant differential operator $X_q^{(\alpha)}$ of order $|\alpha|$ on $\Gp$, such that the following Taylor's formula holds for every smooth function $f$ on $\Gp$ and every $N\in\mathbb{N}_0$:
$$
f(xy)=\sum_{|\alpha|<N}q^\alpha(y^{-1})X_q^{(\alpha)}f(x)
+R_{N}(f;x,y).
$$
Here the remainder $R_{N}(f;x,y)$ depends linearly on derivatives of $f$ of order $\leq N$, is smooth in $x,y$, and satisfies $R_{N}(f;x,y)=O\big(|f|_{C^N}\dist(y,e)^N\big)$.
\end{proposition}

\begin{proof}[Sketch of proof]
Suppose without loss of generality that the first $n$ functions are of maximal rank at $e$. By the implicit function theorem, the mapping $x\to(q_1(x),\cdots,q_M(x))\in\mathbb{R}^M$ is a smooth embedding in some neighbourhood of $e\in\Gp$. Consider first $x=e$. Fix a basis $X_1,\cdots,X_n$ of $\mathfrak{g}$, such that $X_iq_j(e)=\delta_{ij}$ for $i,j\leq n$. For multi-indices $\alpha\in\mathbb{N}_0^M$, write $X^\alpha=X_1^{\alpha_1}\cdots X_n^{\alpha_n}$ (the ``normal" order); if the first $n$ components of $\alpha$ all vanish, then just set $X^\alpha=0$. Choose constants $c_\alpha$ inductively so that
$$
\left(\sum_{\alpha}c_{\alpha}X^\alpha\right) q^\beta(e)=\delta_{\alpha\beta},
\quad 
\beta_i\neq0\text{ only for }1\leq i\leq n.
$$
Then with $X^{(\alpha)}_q:=\sum_{\alpha}c_{\alpha}X^\alpha$, the formula holds for $x=e$ as an implication of the usual Taylor's formula. In order to pass to general $x\in \Gp$, just use the left translation operator $L_x$ to act on both sides, and notice that $X^{(\alpha)}_q$ are left-invariant differential operators.
\end{proof}
\begin{remark}
If $f$ takes value in a finite dimensional normed space $V$, the Taylor's formula given above is still valid. If we define the $C^N$ norm of $f$ by
$$
|f|_{C^N;V}:=\sup_{v^*\in V^*:v^*\neq0}\frac{\big|\langle f,v^*\rangle\big|_{C^N}}{|v^*|},
$$
then the estimate $|R_{N}(f;x,y)|_{V}=O\big(|f|_{C^N;V}\dist(y,e)^N\big)$ is valid, with the implicit constant independent from the dimension of $V$.
\end{remark}

Note that the differential operators $X_q^{(\alpha)}$ are defined independently for every multi-index, so the equality $X_q^{(\alpha+\beta)}=X_q^{(\alpha)}X_q^{(\beta)}$ is, in general, not valid.

A \emph{symbol} $a$ on $\Gp$ is simply a field $a$ defined on $\Gp\times\DuGp$, such that $a(x,\xi)$ is a distribution of value in $\mathrm{End}(\Hh[\xi])$ for each $\xi\in\DuGp$. If a basis for $\Hh[\xi]$ is chosen, the value $a(x,\xi)$ can be simply understood as a matrix function of size $d_\xi\times d_\xi$. We define the \emph{quantization} of a symbol $a$ formally by
\begin{equation}\label{Op(a)}
\Op(a)f(x):=\sum_{\xi\in\DuGp} d_\xi\Tr\left(a(x,\xi)\cdot\Ft f(\xi)\cdot\xi(x)\right).
\end{equation}
Conversely, if $A:C^\infty(\Gp)\to\mathcal{D}'(\Gp)$ is a continuous operator, then it is the quantization of the symbol
\begin{equation}\label{Op-Symbol}
\sigma[A](x,\xi):=\xi^*(x)\cdot(A\xi)(x).
\end{equation}
Here $A\xi$ is understood as entry-wise action. In this case the series (\ref{Op(a)}) then converges in $\mathcal{D}'(\Gp)$. The \emph{associated right convolution kernel} for $a(x,\xi)$ is defined by
\begin{equation}\label{Symbol-Ker}
\mathcal{K}(x,y):=\left(a(x,\cdot)\right)^\vee(y)
=\sum_{\xi\in\DuGp}d_\xi\Tr\left(a(x,\xi)\cdot\xi(y)\right),
\end{equation}
where the Fourier inversion is taken with respect to $\xi$. Formally, the action $\Op(a)f$ can be written as a convolution:
\begin{equation}\label{Ker-Conv}
f(\cdot)*\mathcal{K}(x,\cdot)
=\int_{\Gp}f(y)\mathcal{K}(x,y^{-1}x)dy.
\end{equation}

An intrinsic notion of difference operators acting on symbols was introduced by Fischer \cite{Fis2015}, generalizing the differential operator with respect to $\xi$ in harmonic analysis on $\mathbb{R}^n$. For any (continuous) unitary representation $(\tau,\mathcal{H}_\tau)$, Maschke's theorem ensures that $\tau=\oplus_j\xi_j$ for finitely many $\xi_j\in\DuGp$. A symbol $a(x,\xi)$ can be naturally extended to any $\tau$ by $a(x,\tau)=\oplus_ja(x,\xi_j)$, up to equivalence of representations. The definition of difference operators is then given by
\begin{definition}
Given any extended symbol $a(x,\cdot)$ and representation $\tau$, the difference operator $\Df_\tau$ gives rise to a new extended symbol in the following manner:
$$
\Df_\tau a(x,\pi):=a(x,\tau\otimes\pi)-a(x,\I[\tau]\otimes\pi).
$$
\end{definition}
For a tuple of representations $\boldsymbol{\tau}=(\tau_1,\cdots,\tau_p)$, write $\Df^{\boldsymbol{\tau}}=\Df_{\tau_1}\cdots\Df_{\tau_p}$, and $\Df^{\boldsymbol{\tau}}a(x,\xi)$ is then an endomorphism of $\mathcal{H}_{\tau_1}\otimes\cdots\otimes\mathcal{H}_{\tau_p}\otimes\Hh[\xi]$.

Corresponding to the functions $q=\{q_i\}_{i=1}^M$ as in Proposition \ref{TaylorGp}, Ruzhansky and Turunen introduced the so-called \emph{admissible difference operators}, also generalizing the differential operator with respect to $\xi$ in harmonic analysis on $\mathbb{R}^n$, but in a more computable way compared to the intrinsic definition. We employ the following definitions from \cite{RT2009}, \cite{RT2013} and \cite{RTW2014}:

\begin{definition}\label{RTAdm}
An $M$-tuple of smooth functions $q=(q_i)_{i=1}^M$ on $\Gp$ is said to be RT-admissible if they all vanishing at $e\in\Gp$ and $(\nabla q_i)_{i=1}^M$ has rank $n$. If in addition the only common zero of $(q_i)_{i=1}^M$ is the identity element, then the $M$-tuple is said to be strongly RT-admissible.
\end{definition}

\begin{definition}\label{Difference}
Given a function $q\in C^\infty(\Gp)$, the corresponding RT-difference operator $\Df_q$ acts on the Fourier transform of a $f\in\mathcal{D}'(\Gp)$ by
$$
\Big(\Df_q\Ft f\Big)(\xi):=\Ft{qf}(\xi).
$$
The corresponding collection of RT-difference operators corresponding to an $M$-tuple $q=(q_i)_{i=1}^M$ is the set of difference operators
$$
\Df_q^\alpha:=\Df_{q^\alpha}=\Df_{q_1}^{\alpha_1}\cdots\Df_{q_M}^{\alpha_M}.
$$
If the tuple $q$ is RT-admissible (strongly RT-admissible), the corresponding collection of RT difference operators is said to be RT-admissible (strongly RT-admissible).
\end{definition}

We write $\Df_{q,\xi}$ for the action of a difference operator on the $\xi$ variable. Formally, we have
\begin{equation}\label{Diff_qa}
(\Df_{q,\xi}a)(x,\xi)=\int_\Gp q(y)\left(\sum_{\eta\in\DuGp} d_\eta\Tr\left(a(x,\eta)\cdot\eta(y)\right)\right)\xi^*(y)dy,
\end{equation}
so $\Df_{q,\xi}$ commutes with any differential operator acting on $x$. To compare with $\mathbb{R}^n$, we simply notice that given a symbol $a(x,\xi)$ on $\mathbb{R}^n$, the Fourier inversion of $\partial_\xi a(x,\xi)$ with respect to $\xi$ is $iy a^{\vee}(x,y)$, i.e. multiplication by a polynomial. The functions $q$ on $\Gp$ then play the role of monomials on $\mathbb{R}^n$.

With the aid of difference operators, \cite{Fis2015}  introduced symbol classes of interest.
\begin{definition}\label{S^mrd}
Let $m\in\mathbb{R}$, $0\leq\delta\leq\rho\leq1$. Fix a basis $X_1,\cdots,X_n$ of $\mathfrak{g}$, and define $X^\alpha$ as in Proposition \ref{NormalOrder}. The symbol class $\mathscr{S}^m_{\rho,\delta}(\Gp)$ is the set of all symbols $a(x,\xi)$, such that $a(x,\xi)$ is smooth in $x\in\Gp$, and for any tuple of representation $\boldsymbol{\tau}=(\tau_1,\cdots,\tau_p)$ and any left-invariant differential operator $X^\alpha$, there is a constant $C_{\alpha\boldsymbol{\tau}}$ such that
$$
\big\|X^\alpha\Df^{\boldsymbol{\tau}}a(x,\xi)\big\|
\leq C_{\alpha\boldsymbol{\tau}}\size[\xi]^{m-\rho p+\delta|\alpha|}.
$$
Here the norm is taken to be the operator norm of $\mathcal{H}_{\tau_1}\otimes\cdots\otimes\mathcal{H}_{\tau_p}\otimes\Hh[\xi]$.
\end{definition}
\begin{definition}\label{S^mrdAdmi}
Let $m\in\mathbb{R}$, $0\leq\delta\leq\rho\leq1$. Fix a basis $X_1,\cdots,X_n$ of $\mathfrak{g}$, and define $X^\alpha$ as in Proposition \ref{NormalOrder}. Let $q=(q_1,\cdots,q_M)$ be an RT-admissible $M$-tuple. The symbol class $\mathscr{S}^m_{\rho,\delta}(\Gp;\Df_q)$ is the set of all symbols $a(x,\xi)$, such that $a(x,\xi)$ is smooth in $x\in\Gp$, and 
$$
\big\|X^{\alpha}_x\Df_{q,\xi}^{\beta}a(x,\xi)\big\|
\leq C_{\alpha}\size[\xi]^{m+\delta|\alpha|-\rho|\beta|}.
$$
We can also introduce the norms
\begin{equation}\label{RhoDeltaNorm}
\mathbf{M}_{k,\rho;l,\delta;q}^m(a):=\sup_{|\alpha|\leq k}\sup_{|\beta|\leq l}
\sup_{x,\xi}\size[\xi]^{\rho|\beta|-m-\delta|\alpha|}
\left\|X^\alpha_x\Df_{q;\xi}^\beta a(x,\xi)\right\|.
\end{equation}
In particular, for $\rho=\delta=1$, we write $\mathbf{M}_{k,l;q}^m(a)$ for simplicity.
\end{definition}

In \cite{Fis2015}, Fischer proved that if $q=(q_i)_{i=1}^M$ is a strongly RT-admissible tuple, then the symbol class in Definition \ref{S^mrdAdmi} does not depend on the choice of $q$, and in fact gives rise to the usual H\"{o}rmander class of pseudo-differential operators.

\begin{theorem}[Fischer \cite{Fis2015}, Theorem 5.9. and Corollary 8.13.]\label{Fischer}
\hfill\par
(1) Suppose $0\leq\delta\leq\rho\leq1$. If $q=(q_i)_{i=1}^M$ is a strongly RT-admissible tuple, then a symbol $a\in \mathscr{S}^m_{\rho,\delta}$ if and only if all the norms $\mathbf{M}_{r,\rho;l,\delta;q}^m(a)$ are finite.

(2) Moreover, if $\rho>\delta$ and $\rho\geq1-\delta$, then the operator class $\Op \mathscr{S}^m_{\rho,\delta}$ coincides with the H\"{o}rmander class $\Psi^m_{\rho,\delta}$ of $(\rho,\delta)$-pseudo-differential operators defined via local charts.
\end{theorem}
Let us briefly sketch Fischer's proof of (1) in Theorem \ref{Fischer}. The following lemma is a key ingredient, and is worthy of a specific mention:
\begin{lemma}[\cite{Fis2015}, Lemma 5.10]\label{Fisqq'Lem}
Suppose $q,q'\in\mathcal{D}(\Gp)$, such that $q'/q$ extends to a smooth function on $\Gp$. If a symbol $a=a(\xi)$ is such that $\|\Df_qa(\xi)\|$ does not grow faster than a power of $\size[\xi]$, then 
$$
\|\Df_{q'}a(\xi)\|\lesssim_{q,q'}\|\Df_qa(\xi)\|.
$$
\end{lemma}
Using this lemma, Fischer showed that for any two strongly RT-admissible tuple $q,q'$, the collection of norms $\mathbf{M}_{r,\rho;l,\delta;q}^m(a)$ and $\mathbf{M}_{r,\rho;l,\delta;q'}^m(a)$ are equivalent to each other. 

A convenient choice of strongly RT-admissible tuple is necessary for calculations. Fischer found that matrix entries of the fundamental representations produce a strongly RT-admissible tuple in the following way: for each unitary representation $\tau$ of $\Gp$, set
$$
q_{jk}^{(\tau)}(x)=\tau_{jk}(x)-\delta_{jk},
$$
i.e. the matrix entries of $\tau-\I[\tau]$. Fischer showed in Lemma 5.11. of \cite{Fis2015} that, as $\tau$ exhausts all of the fundamental representations $\Xi(\bm{\varpi}_i),i=1,\cdots,\varrho$, the collection $q=\{q_{jk}^{(\tau)}\}$ is strongly RT-admissible. The advantage of this choice is that it is essentially intrinsically defined for $\Gp$, and exploits the algebraic structure of $\Gp$. From now on, we will just defined the \emph{fundamental tuple of $\Gp$} as
\begin{equation}\label{QFund}
Q:=\bigcup_{\tau:\text{ fundamental representation}}\left\{\tau_{ij}-\delta_{ij}:i,j=1,\cdots,d_\tau\right\}.
\end{equation}
Thus, as Fischer pointed out, the symbol class $\mathscr{S}^m_{\rho,\delta}(\Gp;\Df_Q)$ in Definition \ref{S^mrdAdmi}, with $Q$ being the fundamental tuple (\ref{QFund}), in fact coincides with the intrinsic definition of symbols \ref{S^mrd}. This is because for each representation $\tau$, the matrix entries of $\Df_\tau a$ are exactly given by $\Df_{\tau_{jk}-\delta_{jk}}a$, so as $\tau$ exhausts all of the fundamental representations, the corresponding difference operators should give rise to the class $\mathscr{S}^m_{\rho,\delta}(\Gp)$. It is legitimate to write $\mathscr{S}^m_{\rho,\delta}(\Gp)$ and omit the dependence on strongly RT-admissible difference operators. This is how (1) of Theorem \ref{Fischer} was proved.

A particular property of the fundamental tuple $Q$ deserves a specific mention. For harmonic analysis on $\mathbb{R}^n$, differentiation in frequency usually simplifies the computation drastically. For analysis on compact group, the dual $\DuGp$ is discrete, whence no natural notion of ``continuous differentiation" is available. But the RT difference operators can be constructed in such a manner that they possess \emph{Leibniz type property}: given a tuple $q=(q_i)_{i=1}^M$, the corresponding RT difference operators are said to possess Leibniz type property, if there are constants $c_{j,k}^i\in\mathbb{C}$ such that for symbols $a,b$, 
\begin{equation}\label{Leibniz}
\Df_{q_i}(ab)
=\Df_{q_i}a\cdot b+a\cdot\Df_{q_i}b
+\sum_{j,k=1}^Mc_{j,k}^i\Df_{q_j}a\cdot \Df_{q_k}b.
\end{equation}
Inductively, this implies 
$$
\Df_{q}^\alpha(ab)
=\sum_{\substack{0\leq|\beta|,|\gamma|\leq|\alpha|
\\|\alpha|\leq|\beta|+|\gamma|\leq2|\alpha|}}
c_{\beta\gamma}^\alpha\Df_{q}^{\beta}a\cdot \Df_{q}^{\gamma}b.
$$
Condition (\ref{Leibniz}) is equivalent to
$$
q_i(xy)=q_i(x)+q_i(y)+\sum_{j,k=1}c_{j,k}^iq_j(x)q_k(y),
$$
so the matrix entries of $\tau-\I[\tau]$, where $\tau$ exhausts all fundamental representations of $\Gp$, is a good choice, because due to the equality
$$
\tau(xy)=\tau(x)\tau(y),
$$
they give rise to strongly admissible RT difference operators that satisfies Leibniz type property: for $q_{ij}(x)=\tau_{ij}(x)-\delta_{ij}$,
\begin{equation}\label{LeibnizFund}
q_{ij}(xy)=q_{ij}(x)+q_{ij}(y)+\sum_{k=1}^{d_\tau}q_{ik}(x)q_{kj}(y).
\end{equation}

The advantage of the symbol class $\mathscr{S}^m_{\rho,\delta}(\Gp)$ is that a symbolic calculus is available. The proof of the following results is quite parallel as in Euclidean spaces (not without technicality in verifying the necessary kernel properties):
\begin{theorem}[\cite{Fis2015}, Corollary 7.9.]\label{RegCompo}
Suppose $0\leq\delta<\rho\leq1$. If $a\in\mathscr{S}^m_{\rho,\delta}$, $b\in\mathscr{S}^{m'}_{\rho,\delta}$, then the composition $\Op(a)\circ\Op(b)\in\Op\mathscr{S}^{m+m'}_{\rho,\delta}$, and in fact if $q$ is any strongly RT-admissible tuple and $X^{(\alpha)}_q$ is as in Proposition \ref{TaylorGp}, then the symbol $\sigma$ of $\Op(a)\circ\Op(b)$ satisfies
$$
\sigma(x,\xi)-\sum_{|\alpha|< N}\left(\Df_{q,\xi}^\alpha a\cdot X_{q,x}^{(\alpha)}b\right)(x,\xi)
\in\mathscr{S}^{m+m'-(\rho-\delta)N}_{\rho,\delta}.
$$
One can thus write
$$
\sigma(x,\xi)\sim\sum_{\alpha}\left(\Df_{q,\xi}^\alpha a\cdot X_{q,x}^{(\alpha)}b\right)(x,\xi).
$$
\end{theorem}
\begin{theorem}[\cite{Fis2015}, Corollary 7.6.]\label{RegAdj}
Suppose $0\leq\delta<\rho\leq1$. If $a\in\mathscr{S}^m_{\rho,\delta}$, then the adjoint $\Op(a)^*\in\Op\mathscr{S}^{m}_{\rho,\delta}$, and in fact if $q$ is any strongly RT-admissible tuple and $X^{(\alpha)}_q$ is as in Proposition \ref{TaylorGp}, then the symbol $a^{\bullet;q}$ of $\Op(a)^*$ satisfies
$$
a^{\bullet;q}(x,\xi)-\sum_{|\alpha|< N}\left(\Df_{q,\xi}^\alpha X_{q,x}^{(\alpha)}a^*\right)(x,\xi)
\in\mathscr{S}^{m+m'-(\rho-\delta)N}_{\rho,\delta}.
$$
One can thus write
$$
a^{\bullet;q}(x,\xi)\sim\sum_{\alpha}\left(\Df_{q,\xi}^\alpha X_{q,x}^{(\alpha)}a^*\right)(x,\xi).
$$
\end{theorem}

Fischer also proved that for $0\leq\delta\leq\rho\leq1$ with $\delta<1$, if $a\in\mathscr{S}^m_{\rho,\delta}$, then $\Op(a)$ maps $H^{s+m}$ continuously to $H^s$ for any $s\in\mathbb{R}$. The proof is still parallel to that on Euclidean spaces. 

Finally, to prove (2) in Theorem \ref{Fischer}, i.e. $\Op\mathscr{S}^m_{\rho,\delta}(\Gp)=\Psi^m_{\rho,\delta}(\Gp)$, the H\"{o}rmander class of operators for $\rho>\delta$ and $\rho\geq1-\delta$, Fischer showed that the commutators of $\Op(a)$ with vector fields and $C^\infty$-multiplications satisfy Beals'
characterization of pseudo-differential operators as stated in \cite{Beals1977}. In particular, the operator class $\Op\mathscr{S}^m_{1,0}(\Gp)$ coincides with the H\"{o}rmander class $\Psi^m(\Gp)$ defined via local charts. Thus, it is possible to manipulate pseudo-differential operators on $\Gp$ using globally defined symbolic calculus, and all the lower-order information lost in operation with principal symbols can be retained.

\subsection{Order of a Symbol}
Unlike the case of $\mathbb{R}^n$, symbolic calculus on the non-commutative Lie group $\Gp$ involves endomorphisms of the representation spaces, hence suffers from non-commutativity. Thus, for example, properties of the commutator\footnote{Note that this is not the commutator of pseudo-differential operators.} $(ab-ba)(x,\xi)=:[a,b](x,\xi)$ of two symbols $a(x,\xi)$ and $b(x,\xi)$ is not as clear as on $\mathbb{R}^n$ (it simply vanishes for symbols on $\mathbb{R}^n$). With the aid of Fischer's theorem, however, we are able to show that the commutator of symbols of order $m$ and $m'$ respectively ``is reduced by order 1".

We introduce a formal definition of order as follows.
\begin{definition}\label{2Order}
Let $m\in\mathbb{R}$. We say that a symbol $a(x,\xi)$ on $\Gp$, regardless of regularity in $x$, is of order $m$, if for some strongly RT-admissible tuple $q$, there always holds
$$
\sup_{x\in\Gp}\big\|\Df_q^\beta a(x,\xi)\big\|\lesssim \size[\xi]^{m-|\beta|}.
$$
By Fischer's Lemma \ref{Fisqq'Lem}, the class does not depend on the choice of strongly RT -admissible tuple.
\end{definition}

Thus the class of symbols of order $m$, in our convention, is the collection of symbols that ``possess best decays upon differentiation in $\xi$". It necessarily includes all the $\mathscr{S}^m_{1,\delta}(\Gp)$ with $0\leq\delta\leq1$. The property that we need to prove is 
\begin{proposition}\label{2OrderComm}
Let $a,b$ be symbols on $\Gp$ of order $m$ and $m'$ and type 1 respectively. Then the commutator $[a,b]$ is of order $m+m'-1$. In particular, if $\delta\in[0,1]$, $a\in\mathscr{S}^m_{1,\delta}(\Gp)$ and $b\in\mathscr{S}^{m'}_{1,\delta}(\Gp)$, then $[a,b]\in\mathscr{S}^{m+m'-1}_{1,\delta}(\Gp)$.
\end{proposition}
\begin{proof}
It is quite alright to disregard the dependence on $x$ and consider Fourier multipliers $a=a(\xi)$, $b=b(\xi)$ alone, as no differentiation in $x$ is involved. By Fischer's theorem \ref{Fischer}, the operators $A=\Op(a)$ and $B=\Op(b)$ are in fact in the H\"{o}rmander class of pseudo-differential operators $\Psi^m_{1,0}$ and $\Psi^{m'}_{1,0}$ respectively. The H\"{o}rmander calculus on principal symbols (as functions on the cotangent bundle) then asserts that $[A,B]\in\Psi^{m+m'-1}_{1,0}$, with principal symbol given by the Poisson bracket. But $[A,B]$ is nothing but $[\Op(a),\Op(b)]=\Op\big([a,b]\big)$ as $a,b$ are assumed to be Fourier multipliers. Fischer's theorem again implies that $[a,b]\in\mathscr{S}^{m+m'-1}_{1,0}$.  To extend this to $x$-dependent symbols $a(x,\xi)$ and $b(x,\xi)$, it suffices to note that, for example, given a vector field $X$, the Leibniz rule $X_x\big([a,b]\big)=[X_xa,b]+[a,X_xb]$ holds. The operator norm of $X_x\big([a,b]\big)$ at a given representation $\xi$ can thus be estiamted similarly.
\end{proof}

\section{(1,1) Pseudo-differential Operators on Compact Lie Group}\label{3}
\subsection{Littlewood-Paley Decomposition}
We start with a continuous version of Littlewood-Paley decomposition. Such construction has already been explored by Klainerman and Rodnianski in \cite{KR2006} for a general compact manifold. Here we aim to expoit the Lie group structures to deduce more. Fix an even function $\phi\in C_0^\infty(\mathbb{R})$, such that $\phi(\lambda)=1$ for $|\lambda|\leq1/2$, and $\phi(\lambda)=0$ for $|\lambda|\geq1$. Setting $\psi(\lambda)=-\lambda\phi'(\lambda)$, we obtain a continuous partition of unity
$$
1=\phi(\lambda)+\int_1^\infty\psi\Big(\frac{\lambda}{t}\Big)\frac{dt}{t}.
$$
The \emph{continuous Littlewood-Paley decomposition} of a distribution $f\in\mathcal{D}'(\Gp)$ will then be defined by 
\begin{equation}\label{LPCont}
f=\phi\big(|\nabla|\big)f+\int_1^\infty\psi_t\big(|\nabla|\big)f\frac{dt}{t},
\quad 
\psi_t(\cdot)=\psi\left(\frac{\cdot}{t}\right)
\end{equation}
We also write the \emph{partial sum operator} as
$$
\phi_T(|\nabla|)f:=\phi\left(\frac{|\nabla|}{T}\right)f
=\phi\big(|\nabla|\big)f+\int_1^T\psi_t\big(|\nabla|\big)f\frac{dt}{t}.
$$
This is the convention employed by H\"{o}rmander \cite{Hormander1997}, Chapter 9. It is easy to deduce the Bernstein inequality:
\begin{proposition}\label{Bernstein}
If $f\in\mathcal{D}(\Gp)$ has Fourier support contained in $\mathfrak{S}[0,t]$, then for any $s\in\mathbb{R}$,
$$
\|f\|_{H^s}\lesssim t^s\|f\|_{L^2}.
$$
\end{proposition}

For each $t\geq1$, the integrand $\psi(|\nabla|/t)f$ is a smooth function on $\Gp$, whose Fourier support is contained in $\mathfrak{S}[t/2,t]$, where the definition is given in (\ref{FourierSupp}). In particular, the mode of $\psi(|\nabla|/t)f$ corresponding to zero eigenvalue is always zero. The integral (\ref{LPCont}) converges at least in the sense of distribution, and just as in the case of Littlewood-Paley decomposition on Euclidean space, the speed of convergence reflects the regularity property of $f$. Sometimes the \emph{discrete Littlewood-Paley decomposition} is also employed: setting $\vartheta(\lambda)=\phi_1(\lambda)-\phi_1(\lambda)$ and $\vartheta_j(\lambda)=\vartheta(\lambda/2^j)$, then 
\begin{equation}\label{LPDisc}
f=\phi\big(|\nabla|\big)f+\sum_{j\geq0}\vartheta_j\big(|\nabla|\big)f.
\end{equation}
The continuous and discrete Littlewood-Paley decomposition are parallel to each other, with summation in place of integration or vice versa, but in different scenarios one version can make the computation simpler than the other.

Since the integrand in (\ref{LPCont}) is the convolution of $f$ with kernel 
$$
\left(\psi_t(|\xi|)\I[\xi]\right)^\vee(x)
=\sum_{\xi\in\Gp}d_\xi\psi_t(|\xi|)\Tr\left(\xi(x)\right),
$$
it is thus important to understand a kernel with finite Fourier support and is merely a scaling on each representation space. Notice that the order of convolution does not really matter for this specific case. We thus state the following fundamental lemma concerning kernels of spectral cut-off operators:

\begin{lemma}\label{CutOffKer}
Let $h\in C_0^\infty(\mathbb{R})$. Let $P$ be any classical differential operator with smooth coefficients of degree $N\in\mathbb{N}_0$. Then for $t\geq1$, the function
$$
\check h_t:=\left(h_t(|\xi|)\I[\xi]\right)^\vee
=\sum_{\xi\in\Gp}d_\xi h\left(\frac{|\xi|}{t}\right)\Tr\left(\xi\right)
$$
satisfies
$$
\|P\check h_t\|_{L^1}\lesssim_{N}t^{N}|h|_{L^\infty}.
$$
If $f\in C(\Gp)$ is such that for some $r>0$, $f(x)=O(\dist(x,e)^r)$ when $x\to e$, then
$$
\|f\check h_t\|_{L^1}\lesssim |h|_{L^\infty}t^{-r}.
$$
\end{lemma}
With a little abuse of notation, we denote the Fourier inversion of the symbol $h_t(|\xi|)\I[\xi]$ as $\check{h}_t$. The proof can be found in \cite{Fis2015}. It in fact only depends on the compact manifold structure of $\Gp$, and follows from properties of heat kernel on a compact manifold. As a corollary, the symbol $h_t(|\xi|)\I[\xi]$ is of class $\mathscr{S}^0_{1,0}$ if $h$ is smooth and has compact support. We deduce a corollary that will be used later.

\begin{corollary}\label{ConvVanish}
Let $h\in C_0^\infty(\mathbb{R})$. If $f\in \mathcal{D}(\Gp)$ is such that $f(x)=O(\dist(x,e)^N)$ for some $N\in\mathbb{N}_0$, then $h_t*f$ remains in a bounded subset of $\mathcal{D}(\Gp)$, for $t\geq1$, and
$$
\big|(\check h_t*f)(x)\big|
\lesssim_N |h|_{L^\infty}t^{-N}\sum_{k=0}^N t^k|f|_{C^k}\dist(x,e)^k.
$$
\end{corollary}
\begin{proof}
For simplicity we prove for $N=1$; the general case follows similarly. The order of convolution does not matter since $h_t(|\xi|)\I[\xi]$ is a scaling on each $\Hh[\xi]$. We use Taylor's formula to write
$$
\begin{aligned}
(\check h_t*f)(x)
&=\int_{\Gp}\check h_t(y)f(y^{-1}x)dy\\
&=\int_{\Gp}\check h_t(y)f(y)dy+\int_{\Gp}\check h_t(y)R_1(f;y^{-1},x)dy.
\end{aligned}
$$
The first term is controlled by $|h|_{L^\infty}|f|_{C^0}t^{-1}$ by Lemma \ref{Ker-Conv}, and the second is controlled by $|h|_{L^\infty}\dist(x,e)|f|_{C^1}$, by Taylor's formula.
\end{proof}

A more general multiplier theorem was deduced in \cite{Fis2015} from Lemma \ref{CutOffKer}:
\begin{theorem}\label{Multm}
Let $h$ be a smooth function on $[0,\infty)$. Define the multiplier norm
$$
\|h\|_{m;k}=\sup_{\lambda\geq0,l\leq k}\left(1+|\lambda|\right)^{-m+l}\big|h^{(l)}(\lambda)\big|.
$$
If $q\in C^\infty(\Gp)$ vanishes of order $N-1$ at $e$, where $N\geq1$ is an integer, then with 
$$
a_t(\xi)= h_t(|\xi|)\I[\xi],
$$
the symbol $a_t$ is of class $\mathscr{S}^m_{1,0}$, and the difference norm is estimated by
$$
\|\Df_qa_t(\xi)\|\lesssim_{q} \|h\|_{m;k}\size[\xi]^{m-N}t^{-m}.
$$
\end{theorem}

The Littlewood-Paley characterization of Sobolev space is obtained immediately: 
\begin{proposition}\label{LPHs}
Given $s\in\mathbb{R}$, a distribution $f$ belongs to $H^s(\Gp)$ if and only if for some non-vanishing $h\in C_0^\infty(0,\infty)$, there holds
$$
\left\|\phi\big(|\nabla|\big)f\right\|_{L^2}^2
+\int_0^\infty t^{2s-1}\left\|h_t\big(|\nabla|\big)f\right\|_{L^2_x}^2dt
<\infty,
$$
where $h_t(\lambda)=h(\lambda/t)$. The lower bound of integral does not cause singularity since $h=0$ near 0. Similarly, distribution $f$ belongs to $H^s(\Gp)$ if and only if for some non-zero $h\in C_0^\infty(0,\infty)$, there holds
$$
\left\|\phi\big(|\nabla|\big)f\right\|_{L^2}^2
+\sum_{j\geq0}2^{2sj}\left\|h_{2^j}\big(|\nabla|\big)f\right\|_{L^2}^2.
$$
The square root of either of the above quadratic forms is equivalent to $\|f\|_{H^s}$.
\end{proposition}
From Lemma \ref{CutOffKer}, we also obtain a characterization of the Zygmund space $C^r_*(\Gp)$ with $r\geq0$:

\begin{proposition}\label{LPZyg}
For $r\geq0$, a distribution $f\in\mathcal{D}'(\Gp)$ is in the Zygmund class $C^r_*(\Gp)$ if and only if
$$
\sup|\phi(|\nabla|)f|+\sup_{t\geq1}t^r\left|\psi_t\big(|\nabla|\big)f\right|_{L^\infty}<\infty.
$$
This quantity is equivalent to the Zygmund space norm defined via local coordinate charts.
\end{proposition}
\begin{proof}[Sketch of proof]
The proof does not differ very much from the Euclidean case. Suppose first $0<r<1$ is not an integer, and $f\in C^r$. Then 
$$
\psi_t\big(|\nabla|\big)f(x)
=\int_{\Gp}\check\psi_t(y^{-1}x)f(y)dy
=\int_{\Gp}\check\psi_t(y^{-1}x)\left(f(y)-f(x)\right)dy.
$$
Note that the order of convolution does not matter, since on the Fourier side, $\psi_t\big(|\nabla|\big)$ is just a scaling on each representation space, and the last step is because $\check\psi_t$ always has mean zero. Then
$$
|\psi_t\big(|\nabla|\big)f|_{L^\infty}
\lesssim_r |f|_{C^r}\int_{\Gp}\big|\check\psi_t(y^{-1}x)\big|\cdot\dist(x,y)^rdy,
$$
and the right-hand-side is controlled by $|f|_{C^r}t^{-r}$ by Lemma \ref{CutOffKer}. As for the opposite direction, if we know that $\big|\psi_t\big(|\nabla|\big)f\big|_{L^\infty}\lesssim t^{-r}$, then write
$$
\begin{aligned}
\int_1^\infty\left(\psi_t\big(|\nabla|\big)f(x)-\psi_t\big(|\nabla|\big)f(y)\right)\frac{dt}{t}
=\int_{t\leq\dist(x,y)^{-1}}+\int_{t>\dist(x,y)^{-1}}.
\end{aligned}
$$
When $x$ is close to $y$, the first integral is estimated, using Lemma \ref{CutOffKer} again, by
$$
\sum_{|\beta|=1}\int_1^{\dist(x,y)^{-1}}\left|X^\beta\psi_t\big(|\nabla|\big)f\right|_{L^\infty}\cdot \dist(x,y)\frac{dt}{t}
\,\lesssim_r\,
\sup_{t\geq1}t^r\left|\psi_t\big(|\nabla|\big)f\right|_{L^\infty}\cdot\dist(x,y)^r.
$$
The second integral is simply estimated by
$$
2\sup_{t\geq1}t^r\left|\psi_t\big(|\nabla|\big)f\right|_{L^\infty}\int_{t>\dist(x,y)^{-1}}t^{-r-1}dt
\,\simeq_r\,
\sup_{t\geq1}t^r\left|\psi_t\big(|\nabla|\big)f\right|_{L^\infty}\cdot\dist(x,y)^r.
$$
This proves the equivalence for $0<r<1$. As for the general case, we can simply apply the classical Schauder theory and interpolation for H\"{o}lder spaces; recall that interpolation of H\"{o}lder spaces results in the Zygmund space when the intermediate index is an integer.
\end{proof}

With the aid of Lemma \ref{Ker-Conv}, we immediately obtain
\begin{corollary}\label{LPZygCor}
Suppose $r\geq0$ and $f\in C^r_*$. Fix a basis $X_1,\cdots,X_n$ of $\mathfrak{g}$, and define $X^\alpha$ as in Proposition \ref{NormalOrder}. Then there holds, for $t\geq1$,
$$
\left|X^\alpha\phi_t\big(|\nabla|\big)f\right|_{L^\infty}
\lesssim_{r,\alpha}
\left\{
\begin{aligned}
    & t^{(|\alpha|-r)_+}|f|_{C^r_*} & \quad |\alpha|\neq r \\
    & \log(1+t)|f|_{C^r_*} & \quad |\alpha|=r
\end{aligned}
\right.
$$
where $s_+=\max(s,0)$. Furthermore, for $r>0$, $\left|f-\phi_t\big(|\nabla|\big)f\right|_{L^\infty}\lesssim t^{-r}|f|_{C^r_*}$ when $t\geq1$.
\end{corollary}

\begin{remark}\label{Zygmund<0}
Just as in the Euclidean case, we can \emph{define} Zygmund spaces with index $<0$ by Littlewood-Paley characterization: we say that $f\in C^{r}_*$ for $r<0$, iff
$$
|f|_{C^{-r}_*}
:=\sup|\phi(|\nabla|)f|
+\sup_{t\geq1}t^{r}\left|\psi_t\big(|\nabla|\big)f\right|_{L^\infty}<\infty.
$$
This gives a united definition of Zygmund spaces on $\Gp$. In particular, $|\nabla|^s$ is continuous from $C^{r+s}_*$ to $C^r_*$ for all $r\in\mathbb{R}$. Furthermore, for $r<0$, there holds a growth estimate
$$
\begin{aligned}
\left|\phi_t\big(|\nabla|\big)f\right|
\leq \sup|\phi(|\nabla|)f|+\int_1^t\left|\psi_\tau\big(|\nabla|\big)f\right|d\tau
\lesssim t^{-r}|f|_{C^r_*}.
\end{aligned}
$$
\end{remark}

\begin{remark}\label{VectZyg}
These results easily generalize to vector-valued functions. If $r>0$, $V$ is a finite dimensional normed space, $f\in C^r_*(\Gp;V)$, then the norm $|f|_{C^r_*;V}$ is defined as
$$
\sup_{v^*\in V^*,v^*\neq0}\frac{|\langle v^*,f\rangle|_{C^r_*}}{|v^*|_{V^*}}.
$$
We have e.g.
$$
|f|_{C^r_*;V}\simeq\sup|f|_{L^\infty;V}+\sup_{x,y\in\Gp}\frac{|f(x)-f(y)|_V}{\dist(x,y)^r},\quad 0<r<1,
$$
where $|f|_{L^\infty;V}=\sup|f|_V$, and the equivalence is \emph{independent} from $V$. Repeating the same argument as in the proof of Proposition \ref{LPZyg}, just with $\langle v^*,f\rangle$ in place of $f$, then taking supremum over $v^*\in V^*$, we arrive at
$$
|f|_{C^r_*;V}\simeq_r |\phi(|\nabla|)f|_{L^\infty;V}+\sup_{t\geq1}t^r\left|\psi_t\big(|\nabla|\big)f\right|_{L^\infty;V},
$$
$$
\left|X^\alpha\phi_t\big(|\nabla|\big)f\right|_{L^\infty;V}
\lesssim_{r,\alpha}
\left\{
\begin{aligned}
    & t^{(|\alpha|-r)_+}|f|_{C^r_*;V} & \quad |\alpha|\neq r \\
    & \log(1+t)|f|_{C^r_*;V} & \quad |\alpha|=r
\end{aligned}
\right.
$$
and $\left|f-\phi_t\big(|\nabla|\big)f\right|_{L^\infty;V}\lesssim t^{-r}|f|_{C^r_*;V}$ for $t\geq1$. Results for Zygmund spaces of index $\leq0$ in Remark \ref{Zygmund<0} can be generalized to vector-valued functions similarly. All the implicit constants above \emph{do not depend on} $V$.
\end{remark}

It is quite legitimate to refer the decomposition (\ref{LPCont}) or (\ref{LPDisc}) as \emph{geometric Littlewood-Paley decomposition}. In \cite{KR2006}, Klainerman and Rodnianski commented that, unlike Euclidean case, the geometric Littlewood-Paley decomposition on a general manifold does not exclude high-low interactions, although this drawback does not affect applications within their scope. As Berti-Procesi \cite{BP2011} pointed out, this is because that on a general compact manifold, the product of Laplacian eigenfunctions does not enjoy good spectral property (unlike the Fourier modes $e^{ix\xi}$), and this might cause transmission of energy between arbitrary modes for nonlinear dispersive systems. However, when working on a compact Lie group, this issue is resolved by spectral localization properties, Proposition \ref{SpecPrd0} and Corollary \ref{SpecPrd}.

\subsection{(1,1) Pseudo-differential Operator}
Even in the Euclidean case, the symbol class $\mathscr{S}^m_{1,1}(\mathbb{R}^n)$ exhibits exotic properties compared to smaller classes $\mathscr{S}^m_{1,\delta}(\mathbb{R}^n)$ with $\delta<1$, and ``must remain forbidden fruit" as commented by Stein \cite{SteMur1993} (Subsection 1.2., Chapter 7). We thus cannot expect a satisfactory calculus for general $\mathscr{S}^m_{1,1}(\Gp)$ symbols. But in analogy to the Euclidean case, a series of theorems and constructions still remain valid for $\mathscr{S}^m_{1,1}(\Gp)$.

\begin{theorem}[Stein]\label{SteinTheorem}
Suppose $a\in \mathscr{S}^m_{1,1}(\Gp)$. Then for $s>0$, $\Op(a)$ is a bounded linear operator from $H^{s+m}$ to $H^s$. The effective version reads
$$
\|\Op(a)f\|_{H^s}\leq C_{s;q}\mathbf{M}^m_{[s]+1,n+2;q}(a)\|f\|_{H^{s+m}},
$$
where $q$ is a strongly RT-admissible tuple, and the norm $\mathbf{M}^m_{k,l;q}$ is defined in (\ref{RhoDeltaNorm}).
\end{theorem}

Before we proceed to the proof, let us first state several lemmas that will be used later.

\begin{lemma}[Taylor \cite{Taylor2000}, Lemma 4.2 of Chapter 2]\label{l2Seq}
For $s<1$, there holds
$$
\sum_{j\in\mathbb{Z}}2^{2sj}\left|\sum_{k:k<j}2^{k-j}a_k\right|^2
\leq C_s\sum_{k\in\mathbb{Z}}2^{2sk}|a_k|^2.
$$
For $s>0$, there holds
$$
\sum_{j\in\mathbb{Z}}2^{2sj}\left|\sum_{k:k\geq j}a_k\right|^2\leq C_s\sum_{k\in\mathbb{Z}}2^{2sk}|a_k|^2.
$$
\end{lemma}
As a corollary, we also have the following lemma:
\begin{lemma}\label{H^ss>0}
Let $s>0$, $l\geq s+1$ be an integer. Fix a basis $X_1,\cdots,X_n$ of $\mathfrak{g}$, and define $X^\alpha$ as in Proposition \ref{NormalOrder}. There is a constant $C=C(s,l)$ with the following properties. If $\{f_k\}_{k\geq0}$ is a sequence in $H^l(\Gp)$, such that for any multi-index $\alpha$ with $|\alpha|\leq l$, there holds
$$
\|X^\alpha f_k\|_{L^2}\leq 2^{k(|\alpha|-s)}c_k,\quad (c_k)\in \ell^2(\mathbb{N}),
$$
then in fact $\sum_k f_k\in H^s(\Gp)$, and 
$$
\left\|\sum_{k\geq0}f_k\right\|_{H^s}^2
\leq C\sum_{k\geq0}c_k^2.
$$
\end{lemma}
\begin{proof}
We will use the discrete Littlewood-Paley decomposition (\ref{LPDisc}) to conduct the proof. If $j\leq k$, then by taking $\alpha=0$, obviously
$$
\|\vartheta_j\big(|\nabla|\big)f_k\|_{L^2}\leq\|f_k\|_{L^2}\leq 2^{-ks}c_k.
$$
If $j>k$, then by Bernstein's inequality (Proposition \ref{Bernstein}),
$$
\begin{aligned}
\|\vartheta_j\big(|\nabla|\big)f_k\|_{L^2}
&\leq C{2^{-lj}}\|\vartheta_j\big(|\nabla|\big)f_k\|_{H^l}\\
&\simeq C2^{-lj}\sum_{|\alpha|=l}\|X^\alpha f_k\|_{L^2}\\
&\leq C2^{-ks}2^{l(k-j)}c_k.
\end{aligned}
$$
Thus, for $f=\sum_kf_k$, we have
$$
2^{js}\|\vartheta_j\big(|\nabla|\big) f\|_{L^2}
\leq C\sum_{k:k<j}2^{(l-s)(k-j)}c_k
+\sum_{k:k\geq j}2^{(j-k)s}c_k.
$$
Using Lemma \ref{l2Seq}, we find that $\left\{2^{js}\|\vartheta_j\big(|\nabla|\big) f\|_{L^2}\right\}_{j\geq0}$ is an $\ell^2$ sequence. Thus $f\in H^s$.
\end{proof}

\begin{proof}[Proof of Stein's Theorem]
The proof here is very much like the one in \cite{Met2008}, Section 4.3. Writing $a_0(x,\xi)=a(x,\xi)\phi(|\xi|)$, consider the dyadic decomposition
$$
a(x,\xi)=a_0(x,\xi)+\sum_{j=1}^\infty a_j(x,\xi),
\quad\text{with}\quad 
a_j(x,\xi)=a(x,\xi)\vartheta_j(|\xi|).
$$
Define $A_j=\Op a_j(x,\xi)$. By formula (\ref{Op(a)})-(\ref{Ker-Conv}), it follows that $A_j$ is represented as a convergent integral operator
$$
A_ju(x)=\int_{\Gp}u(y)\mathcal{K}_j(x,y^{-1}x)dy,
$$
where the kernel
$$
\mathcal{K}_j(x,y)=\left(a_j(x,\cdot)\right)^\vee(y)
=\sum_{\xi\in\DuGp}d_\xi\Tr\left(a_j(x,\xi)\cdot\xi(y)\right).
$$
We fix $Q$ to be the fundamental tuple, as defined in (\ref{QFund}). For any multi-index $\beta$, by definition of RT difference operator,
$$
Q^\beta(y)\mathcal{K}_j(x,y)
=\sum_{\xi\in\DuGp}d_\xi\Tr\left((\Df_{Q^\beta,\xi}a_j)(x,\xi)\cdot\xi(y)\right).
$$
Using the Cauchy-Schwartz inequality for trace inner product (note that if $A$ is an $n\times n$ matrix then $\llbracket A\rrbracket^2\leq n\|A\|^2$), Weyl-type Lemma \ref{Weyl}, the Leibniz type property (\ref{Leibniz}) and Theorem \ref{Multm}, we obtain
\begin{equation}\label{SteinIneq1}
\begin{aligned}
\left|Q^\beta(y)\mathcal{K}_j(x,y)\right|
&\lesssim
\sum_{\xi\in \mathfrak{S}[0,2^{j+1}]}d_\xi^2\cdot\mathbf{M}^m_{0,|\beta|;Q}(a)\size[\xi]^{m-|\beta|}\\
&\lesssim 2^{jn}\cdot\mathbf{M}^m_{0,|\beta|;Q}(a)2^{j(m-|\beta|)}.
\end{aligned}
\end{equation}
Summing over $\beta$ whose whose length is $\leq L$ for $L\in\mathbb{N}$, using that the only common zero of the $Q_i$'s is the identity element, then Fischer's lemma \ref{Fisqq'Lem} for any other strongly RT-admissible tuple $q=(q_i)_{i=1}^M$, we find
\begin{equation}\label{SteinIneq2}
\begin{aligned}
\big(1+|2^j\dist(y,e)|^2\big)^{L}\cdot|\mathcal{K}_j(x,y)|^2
&\lesssim\left[\mathbf{M}^m_{0,L;Q}(a)\right]^22^{2jm+2jn}\\
&\simeq_q\left[\mathbf{M}^m_{0,L;q}(a)\right]^22^{2jm+2jn}.
\end{aligned}
\end{equation}
Hence we can estimate the weighted $L^2$ norm of the kernel:
$$
\begin{aligned}
\int_{\Gp}\big(1+|2^j\dist(x,y)|^2&\big)^{[n/2]+1}|\mathcal{K}_j(x,y^{-1}x)|^2dy\\
&\leq \int_{\Gp}\frac{\big(1+|2^j\dist(x,y)|^2\big)^{n+2}}{\big(1+|2^j\dist(x,y)|^2\big)^{n+1-[n/2]}}\cdot|\mathcal{K}_j(x,y^{-1}x)|^2dy\\
&\lesssim\left[\mathbf{M}^m_{0,n+2;q}(a)\right]^22^{2jm+jn}
\cdot 2^{jn}\int_{\Gp}\frac{dy}{\big(1+|2^j\dist(x,y)|^2\big)^{n+1-[n/2]}}\\
&\lesssim\left[\mathbf{M}^m_{0,n+2;q}(a)\right]^22^{2jm+jn}.
\end{aligned}
$$
Here we estimate the last integral by splitting it into two parts $\dist(x,y)<R$ and $\dist(x,y)\geq R$, where $R$ is the injective radius of $\Gp$; the first part can then be estimated as an Euclidean integral, while the second is just bounded by $2^{-nj}$, and they both absorbs the factor $2^{jn}$.

We then proceed to $X^\alpha\circ A_j$, whose integral kernel will be denoted as $\mathcal{K}_{j,\alpha}(x,y)$, which is just $\mathcal{K}_j(x,y)$ when $\alpha=0$. Using the quantization formulas (\ref{Op(a)})-(\ref{Op-Symbol}), we find that, with $\sigma_{X_i}(x,\xi)=\xi^*(x)\cdot(X_i\xi)(x)$ being the symbol of $X_i$ (which is of class $S^1_{1,0}(\Gp)$ by Fischer's theorem), the symbol of $X_i\circ A_j$ is 
$$
(X_{i;x}a_j)(x,\xi)+\sigma_{X_i}(x,\xi)\cdot a_j(x,\xi)
\in \mathscr{S}^{m+1}_{1,1}(\Gp).
$$
By induction on $|\alpha|$, we obtain 
$$
\int_{\Gp}\big(1+|2^j\dist(x,y)|^2\big)^{[n/2]+1}|\mathcal{K}_{j,\alpha}(x,y^{-1}x)|^2dy
\lesssim\left[\mathbf{M}^m_{|\alpha|,n+2;q}(a)\right]^22^{2j(m+|\alpha|)+jn}.
$$
Applying the Cauchy-Schwarz inequality to $\int \mathcal{K}_{j,\alpha}(x,y^{-1}x)u(y)dy$, we find that for any $u\in L^2$,
$$
\begin{aligned}
\int_{\Gp}|(X^\alpha \circ A_j)u(x)|^2dx
&\lesssim\left[\mathbf{M}^m_{|\alpha|,n+2;q}(a)\right]^22^{2j(m+|\alpha|)+jn}
\int_{\Gp}\int_{\Gp}\frac{|u(y)|^2}{\big(1+|2^j\dist(x,y)|^2\big)^{[n/2]+1}}dydx\\
&=C\left[\mathbf{M}^m_{|\alpha|,n+2;q}(a)\right]^22^{2j(m+|\alpha|)}\|u\|_{L^2}^2,
\end{aligned}
$$
or 
\begin{equation}\label{Aj}
\|X^\alpha\circ A_ju\|_{L^2}\lesssim_\alpha \mathbf{M}^m_{|\alpha|,n+2;q}(a) 2^{j(m+|\alpha|)}\|u\|_{L^2}.
\end{equation}

Now substitute $u$ by any Littlewood-Paley building block $\vartheta_k\big(|\nabla|\big) f$ of $f\in \mathcal{D}(\Gp)$ in (\ref{Aj}). Since the support of $a_j(x,\xi)$ with respect to $\xi$ is contained in $\mathfrak{S}[0,2^{j+1}]$, we have
$$
A_jf=\sum_{k:|k-j|\leq3}A_j\vartheta_k\big(|\nabla|\big)f.
$$
Applying (\ref{Aj}) with $u=\vartheta_k\big(|\nabla|\big) f$, we have
$$
\begin{aligned}
\|X^\alpha A_jf\|_{L^2}
&\lesssim_\alpha \mathbf{M}^m_{|\alpha|,n+2;q}(a) 2^{j(m+|\alpha|)}\sum_{k:|k-j|\leq3}\big\|\vartheta_k\big(|\nabla|\big) f\big\|_{L^2}\\
&\lesssim_\alpha \mathbf{M}^m_{|\alpha|,n+2;q}(a) 2^{j(|\alpha|-s)}\sum_{k:|k-j|\leq3}2^{k(s+m)}\big\|\vartheta_k\big(|\nabla|\big) f\big\|_{L^2}.
\end{aligned}
$$
Since $\sum_{k=1}^\infty2^{2k(s+m)}\big\|\vartheta_k\big(|\nabla|\big) f\big\|_{L^2}^2\simeq\|f\|_{H^{s+m}}^2$, we apply Lemma \ref{H^ss>0} to obtain
$$
\|\Op(a)f\|_{H^s}
=\left\|\sum_{j=0}^\infty A_jf\right\|_{H^s}
\leq C \mathbf{M}^m_{[s]+1,n+2;q}(a)\|f\|_{H^{s+m}}.
$$
\end{proof}
\begin{remark}
We point out that the proof generalizes without much modification to Zygmund spaces $C^r_*$ with $r>0$. Since this is not needed for our goal, we omit the details.
\end{remark}

We do not know whether the class $\Op\mathscr{S}^m_{1,1}$ given in Definition \ref{S^mrd}-\ref{S^mrdAdmi} coincides with any locally defined H\"{o}rmander class of operators or not. However, this is not a question of concern within our scope, since all we need is a calculus for a suitable subclass of symbols and operators, broad enough to cover operators arising from classical differential operators. 

\subsection{Spectral Condition}
The constructions and theorems so far rely mostly, if not completely, on the compact manifold structure of $\Gp$. From now on, we will exploit more algebraic properties of $\Gp$. In particular, the spectral localization property, namely Corollary \ref{SpecPrd}, will play an important role. Just as in the Euclidean case, we introduce a subclass of $\mathscr{S}^m_{1,1}(\Gp)$ that satisfies a spectral condition. The exposition in this subsection is parallel to that of \cite{Met2008}, Chapter 4, or \cite{Hormander1997}, Chapter 9-10. 

Before we pass to the definition of spectral condition, we need a lemma concerning difference operators acting on symbols with finite Fourier support. Being trivial for Euclidean or toric case, for non-commutative groups it relies on spectral localization.
\begin{lemma}\label{DiffVanish}
Suppose the symbol $a=a(\xi)$ has Fourier support contained in $\mathfrak{S}[\gamma_1,\gamma_2]$. Then with $Q$ being the fundamental tuple defined in (\ref{QFund}), there is a constant $C$ depending only on the algebraic structure of $\Gp$, such that for any $Q_i$, the difference $\Df_{Q_i}a$ vanishes for $|\xi|\notin[\gamma_1-C,\gamma_2+C]$.
\end{lemma}
\begin{proof}
We use formula (\ref{Diff_qa}):
$$
(\Df_{Q_i}a)(\xi)=\int_\Gp Q_i(y)\left(\sum_{\eta\in \mathfrak{S}[\gamma_1,\gamma_2]} d_\eta\Tr\left(a(\eta)\cdot\eta(y)\right)\right)\xi^*(y)dy.
$$
By the proof of Corollary \ref{SpecPrd}, the product $Q_i(y)\eta(y)$ is the sum of matrix entries of irreducible representations whose highest weights $\bm{j}\in\mathscr{L}_+$ satisfy
$$
\bm{J}(\eta)-\bm{\varpi}_i\preceq \bm{j}\preceq \bm{J}(\eta)+\bm{\varpi}_i.
$$
Such $\bm{j}$'s must fall into a bounded parallelogram in $\mathfrak{t}^*$ that contains $\bm{J}(\eta)$, so there must hold 
$$
|\bm{J}(\eta)|-C\leq|\bm{j}|\leq |\bm{J}(\eta)|+C,
$$
where $C\simeq\sum_{i=1}^\varrho|\bm{\varpi}_i|$. 
Hence the Fourier support of 
$$
Q_i(y)\left(\sum_{\eta\in \mathfrak{S}[\gamma_1,\gamma_2]} d_\eta\Tr\left(a(\eta)\cdot\eta(y)\right)\right)
$$
is contained in $\mathfrak{S}\big[\gamma_1-C,\gamma_2+C\big]$, with $C\simeq\sum_{i=1}^\varrho|\bm{\varpi}_i|$.
\end{proof}

We turn to the definition of spectral condition. Since no natural ``addition" is available on $\DuGp$, we restrict ourselves to a narrower definition of spectral condition:
\begin{definition}
Fix $\delta>0$. The subclass $\Sigma_\delta^m(\Gp)\subset \mathscr{S}^m_{1,1}(\Gp)$ consists of all symbols $a\in \mathscr{S}^m_{1,1}(\Gp)$ such that the partial Fourier transform of matrix entries of $a(x,\xi)$ satisfies
$$
\Ft{a_{ij}}(\eta,\xi)
:=\int_{\Gp}a_{ij}(x,\xi)\eta^*(x)dx=0
\quad\mathrm{if}\quad|\eta|\geq \delta\size[\xi],
\quad i,j=1,\cdots,d_\xi,
$$
when $\size[\xi]$ is large enough. This is called the spectral condition with parameter $\delta$.
\end{definition}

With a little abuse of notation, we may write 
$$
\Ft{a}(\eta,\xi)= \big[\Ft{a_{ij}}(\eta,\xi)\big]_{i,j=1}^{d_\xi}\in\mathrm{End}(\Hh[\eta]\otimes\Hh[\xi]),
$$
and abbreviate the spectral condition as $\Ft{a}(\eta,\xi)=0$ for $|\eta|\geq \delta\size[\xi]$. Since the symbol $a$ is smooth in $x$, the spectral condition is only important for high modes, i.e. when $\size[\xi]$ is large: if $a(x,\xi)$ vanishes for all but finitely many $\xi\in\DuGp$, then in fact $a\in \mathscr{S}^{-\infty}$, so the ``low frequency part of $a(x,\xi)$" does not cause any problem in regularity. This is why the spectral condition is required only for \emph{sufficiently large $\size[\xi]$.}

We list some basic properties of $\Sigma^m_\delta(\Gp)$ that will be used repeatedly.

\begin{lemma}\label{Sigmadelta}
(1) If $a\in\Sigma^m_\delta(\Gp)$, $b\in\Sigma^{m'}_{\delta'}(\Gp)$, then $a\cdot b\in\Sigma^{m+m'}_{\delta+\delta'}(\Gp)$.

(2) If $a\in\Sigma^m_\delta(\Gp)$, $X\in\mathfrak{g}$ is a left-invariant vector field, then the symbol of $X\circ\Op(a)$ is of order $m+1$, and still satisfies a spectral condition with parameter $\delta$.

(3) Suppose $a\in\Sigma^m_\delta$. Let $Q$ be the fundamental tuple as defined in (\ref{QFund}). Then given any multi-index $\alpha$, there is a constant $C_\alpha$ depending only on $\alpha$ and the algebraic structure of $\Gp$, such that the symbol $\Df_Q^\alpha a\in \mathscr{S}^{m-|\alpha|}_{1,1}$, and satisfies the spectral condition with parameter $C_\alpha\delta$.
\end{lemma}
\begin{proof}
Part (1) is a direct consequence of the spectral localization property, namely Corollary \ref{SpecPrd}. For part (2), write $A=\Op(a)$. For a left-invariant vector field $X\in\mathfrak{g}$, the symbol of $X\circ A$ reads
\begin{equation}\label{SymbXaj}
(X_{x}a)(x,\xi)+\sigma_X(x,\xi)\cdot a(x,\xi)\in \mathscr{S}^{m+1}_{1,1},
\end{equation}
as seen in the proof of Stein's theorem. By the right convolution kernel formula (\ref{Symbol-Ker})-(\ref{Ker-Conv}), we find that the right convolution kernel $\mathcal{K}(x,y)$ of $X$ is invariant under left translation with respect to $x$, so in fact $\mathcal{K}(x,y)\equiv \mathcal{K}(e,y)$, and the symbol $\sigma_X(x,\xi)$ is thus independent from $x$. As a result, it is legitimate to write it as $\sigma_X(\xi)\in\mathrm{End}(\Hh[\xi])$, and refer it as a \emph{non-commutative Fourier multiplier}. Consequently, the Fourier transform of (\ref{SymbXaj}) with respect to the $x$ variable evaluated at $\eta\in\DuGp$ is 
$$
\sigma_X(\eta)\cdot\Ft{a}(\eta,\xi)+\left(\sigma_X(\xi)\cdot a\right)^{\wedge}(\eta,\xi),
$$
which still vanishes for $|\eta|\geq\delta\size[\xi]$.

As for part (3), we just have to prove the claim for $|\alpha|=1$ as the rest follows by induction. Theorem \ref{Fischer} ensures that $\Df_Q a\in \mathscr{S}^{m-1}_{1,1}$. Here with a little abuse of notation we use $Q$ to denote any component of the fundamental tuple. To verify the spectral condition, we employ formula (\ref{Diff_qa}):
$$
(\Df_Qa)(x,\xi)=\int_\Gp Q(y)\left(\sum_{\zeta\in\DuGp} d_\zeta\sum_{i,j=1}^{d_\zeta}a_{ij}(x,\zeta)\zeta_{ij}(y)\right)\xi^*(y)dy.
$$
Taking Fourier transform with respect to $x$, evaluating at $\eta\in\DuGp$ and using the spectral condition, we obtain
\begin{equation}\label{Lem3.5Temp}
\Ft{\Df_Qa}(\eta,\xi)
=\sum_{\zeta\in\DuGp:\size[\zeta]\geq \delta^{-1}|\eta|}d_\zeta \sum_{i,j=1}^{d_\zeta}\Ft{a_{ij}}(\eta,\zeta)\otimes
\int_\Gp Q(y)\zeta_{ij}(y)\xi^*(y)dy
\in\mathrm{End}(\Hh[\eta]\otimes\Hh[\xi]).
\end{equation}
From the proof of Lemma \ref{DiffVanish}, we find that $Q(y)\xi^*(y)$ has Fourier support contained in $\mathfrak{S}\big[|\xi|-C,|\xi|+C\big]$, where $C$ depends only on the algebraic property of $\Gp$.

Now suppose $|\eta|\geq 2\delta\size[\xi]$. Then for each summand in (\ref{Lem3.5Temp}), the representation $\zeta$ satisfies $|\zeta|\geq 2\size[\xi]$. By Schur orthogonality, matrix entries of $\zeta$ is thus $L^2$ orthogonal to matrix entries of $\bar Q(y)\xi(y)$. This shows that for $|\eta|\geq 2\delta\size[\xi]$, each summand in (\ref{Lem3.5Temp}) vanishes. 
\end{proof}

We are now at the place to state the fundamental theorem for para-differential calculus: (1,1) symbols satisfying a suitable spectral condition give rise to operators bounded in all Sobolev spaces. 

\begin{theorem}\label{Stein'}
Let $\delta\in(0,1/2)$, $m\in\mathbb{R}$. Suppose $a\in \Sigma_\delta^m(\Gp)$. Then the operator $\Op(a)$ maps $H^{s+m}(\Gp)$ to $H^{s}(\Gp)$ continuously for all $s\in\mathbb{R}$. Quantitatively, with $s_+=0$ for $s\leq0$ and $s_+=s$ for $s>0$, for any strongly RT-admissible tuple $q$,
$$
\|\Op(a)f\|_{H^s}\leq C_{s,\delta;q}\mathbf{M}^m_{[s_+]+1,n+2;q}(a)\|f\|_{H^{s+m}}.
$$
\end{theorem}

Before we proceed to the proof, we still need a summation Lemma in Sobolev spaces.

\begin{lemma}\label{H^ssReal}
Let $s\in\mathbb{R}$, $l>s$ be a positive integer, and $\gamma>0$. Fix a basis $X_1,\cdots,X_n$ of $\mathfrak{g}$, and define $X^\alpha$ as in Proposition \ref{NormalOrder}. There is a constant $C=C(s,l,\gamma)$ with the following properties. If $\{f_k\}_{k\geq0}$ is a sequence in $H^l(\Gp)$, such that for any multi-index $\alpha$ with $|\alpha|\leq l$, there holds
$$
\|X^\alpha f_k\|_{L^2}\leq 2^{k(|\alpha|-s)}c_k,\quad (c_k)\in \ell^2(\mathbb{N}),
$$
and furthermore $\Ft{f}_k$ is supported in $\{\xi:\size[\xi]\geq\gamma 2^k\}$, then $\sum_k f_k\in H^s(\Gp)$ and 
$$
\left\|\sum_{k\geq0}f_k\right\|_{H^s}^2
\leq C\sum_{k\geq0}c_k^2.
$$
\end{lemma}
\begin{proof}
The proof is quite similar for that of Lemma \ref{H^ss>0}. The only difference is as follows: the spectral condition $\mathrm{supp}\Ft{f}_k\subset\{\xi:\size[\xi]\geq\gamma 2^k\}$ implies that for some integer $N\sim\log\gamma$, the Littlewood-Paley building block $\vartheta_j\big(|\nabla|\big)f_k$ vanishes for $j<k-N$. So for $f=\sum_kf_k$, we obtain similarly as in Lemma \ref{H^ss>0},
$$
2^{js}\|\vartheta_j\big(|\nabla|\big) f\|_{L^2}
\leq C\sum_{k:k<j}2^{(l-s)(k-j)}c_k
+\sum_{k:j\leq k<j+N}2^{(j-k)s}c_k.
$$
The first sum gives rise to terms of an $\ell^2$ sequence again by Lemma \ref{l2Seq}, while the second obviously gives rise to terms of an $\ell^2$ sequence. Thus again $\{2^{js}\big\|\vartheta_j\big(|\nabla|\big) f\big\|_{L^2}\}_{j\geq0}$ is an $\ell^2$ sequence. The rest of the proof is identical to that of Lemma \ref{H^ss>0}.
\end{proof}

\begin{proof}[Proof of Theorem \ref{Stein'}]
The proof is still very much similar to that of Theorem 4.3.4. in \cite{Met2008}, and proceeds basically parallelly as Stein's theorem. We still decompose $\Op(a)f$ into a dyadic sum $\sum_{j\geq0}A_jf$, where $A_j$ is the quantization of $a_j(x,\xi):=a(x,\xi)\vartheta_j(|\xi|)$. The estimate of operator norms for $X^\alpha\circ A_j$ is identical to that in the proof of Theorem \ref{SteinTheorem}. There are only some minor differences: since $\Ft{a}(\eta,\xi)=0$ for $|\eta|>\delta\size[\xi]$, by the quantization formula (\ref{Op(a)})-(\ref{Op-Symbol}) we have
$$
\Ft{A_jf}(\eta)=\int_{\Gp}\sum_{\xi\in \mathfrak{S}[2^{j-1},2^j]}d_\xi\Tr\left(\xi(x)a_j(x,\xi)\Ft{f}(\xi)\right)\eta^*(x)dx.
$$
By spectral localization property \ref{SpecPrd} and the spectral condition $\Ft{a_j}(\eta,\xi)=0$ for $|\eta|\geq\delta 2^j$ (since $|\xi|\leq 2^j$), we find that the integrand $\Tr\left(\xi(x)a_j(x,\xi)\Ft{f}(\xi)\right)$ has Fourier support (with respect to $x$) contained in 
$$
\left\{\eta\in\DuGp:|\eta|\geq c\left(\frac{1}{2}-\delta\right)2^j\right\},
$$
where $c\in(0,1)$ depends only on the algebraic structure of $\Gp$, i.e. $\Ft{A_jf}(\eta)=0$ for $|\eta|\leq c(1/2-\delta)2^j$. By Lemma \ref{Sigmadelta} and induction on $|\alpha|$, the symbol of $X^\alpha\circ a$ still satisfies the spectral condition with parameter $\delta$. To summarize, for any $f\in\mathcal{D}'(\Gp)$, the Fourier support of $X^\alpha A_jf$ is within $\{\eta\in\DuGp:|\eta|>c(1/2-\delta)2^j\}$. 

With exactly the same argument as in the proof of Theorem \ref{SteinTheorem}, the estimate
$$
\|X^\alpha A_jf\|
\lesssim_\alpha \mathbf{M}^m_{|\alpha|,n+2;q}(a) 2^{j(|\alpha|-s)}\sum_{k:|k-j|\leq3}2^{k(s+m)}\big\|\vartheta_k\big(|\nabla|\big) f\big\|_{L^2}
$$
is valid. We can then apply Lemma \ref{H^ssReal} to conclude.
\end{proof}
\begin{remark}\label{HorRegSym}
The proof for Theorem \ref{SteinTheorem} and \ref{Stein'} proceed with exactly the same idea as in the Euclidean case, since the product spectral localization property does not differ very much from the Euclidean case under dyadic decomposition. We thus expect that H\"{o}rmander's characterization of the class $\tilde\Psi^m_{1,1}(\mathbb{R}^n)$, namely (1,1) pseudo-differential operators that are continuous on all Sobolev spaces (see \cite{Hormander1988} or Chapter 9 of \cite{Hormander1997} for the details), admits a direct generalization to $\Gp$. But we content ourselves with operators satisfying the spectral condition, because this is what we actually need for para-differential calculus. We only write $\tilde{\mathscr{S}}^{m}_{1,1}(\Gp)$ for the subclass of $(1,1)$ symbols whose quantization are continuous on every Sobolev space, and do not characterize it explicitly. The inclusion $\Sigma_{<1/2}^m\subset\tilde{\mathscr{S}}^{m}_{1,1}$ is quite sufficient for most of our applications.
\end{remark}

\subsection{Para-product and Para-linearization on \texorpdfstring{$\Gp$}{a}}
As an initial application, we are able to recover the para-product estimate and Bony's para-linearization theorem on $\Gp$, just as in \cite{BGdP2021}, but this time with a global notion of para-product. We point out that in \cite{KR2006}, the authors also constructed para-product in terms of spectral calculus on a general compact surface. Since $\Gp$ obviously enjoys more structure than a general manifold, we are able to recover almost everything in the Euclidean setting.

Given two distributions $a$ and $u$, we define the \emph{para-product} $T_au$ by
\begin{equation}\label{T_au}
T_au
:=\int_1^\infty\left[\phi_{2^{-10}t}\big(|\nabla|\big)a\right]\cdot\psi_t\big(|\nabla|\big)u\frac{dt}{t},
\end{equation}
where $\phi$ and $\psi$ are as in the continuous Littlewood-Paley decomposition (\ref{LPCont}). Similarly as in the Euclidean case, the gap 10 is inessential. By the spectral localization property i.e. Corollary \ref{SpecPrd}, the integrand has Fourier support contained in $\mathfrak{S}[ct,c^{-1}t]$ for some constant $c\in(0,1)$, depending only on the algebraic structure of $\Gp$. The symbol of $T_a$ is 
$$
\sigma[T_a](x,\xi)
=\int_1^\infty\left[\phi_{2^{-10}t}\big(|\nabla|\big)a(x)\right]\cdot\psi_t\big(|\xi|\big)\cdot\I[\xi]\frac{dt}{t}.
$$
\begin{proposition}
If $a\in L^\infty$, the symbol $\sigma[T_a]$ of the para-product operator $T_a$ is of class $\Sigma^0_{2^{-10}}$.
\end{proposition}
\begin{proof}
The spectral condition is obvious since $\psi_t(|\xi|)\neq0$ only for $|\xi|\leq t$. We turn to estimate the operator norm of $X^\alpha_x\Df_{q,\xi}^\beta \sigma[T_a](x,\xi)(x,\xi)$ on $\Hh[\xi]$, for a given basis $X_1,\cdots,X_n$ of $\mathfrak{g}$ and $q$ is a strongly RT-admissible tuple. But this is immediate from Lemma \ref{CutOffKer} and Theorem \ref{Multm}, and quantatively $\mathbf{M}^0_{k,l;q}\big(\sigma[T_a]\big)\lesssim_{k,l}\size[\xi]^{k-l}|a|_{L^\infty}$.
\end{proof}

By Theorem \ref{Stein'}, the para-product operator $T_a$ is thus continuous from $H^s$ to itself for all $s\in\mathbb{R}$. This of course also admits a more direct proof. We can simply insert the multiplier corresponding to a function $h\in C_0^\infty(0,\infty)$, such that $h(\lambda)\equiv1$ for $\lambda\in[c,c^{-1}]$ and $h\geq0$. Then the Fourier support of $T_au$ is away from the trivial representation $\xi\equiv1$, and
$$
T_au
=\int_0^\infty h_t\big(|\nabla|\big)\Big(\phi_{2^{-10}t}\big(|\nabla|\big)a\cdot\psi_t\big(|\nabla|\big)u\Big)\frac{dt}{t}.
$$
Write $v(t)=\phi_{2^{-10}t}\big(|\nabla|\big)a\cdot\psi_t\big(|\nabla|\big)u$ for the integrand. It has Fourier support contained in $\mathfrak{S}[ct,c^{-1}t]$. By Bernstein's inequality and Lemma \ref{CutOffKer}, we find that
$$
\begin{aligned}
\|v(t)\|_{H^s}
&\lesssim t^{s}\left|\phi_{2^{-10}t}\big(|\nabla|\big)a\right|_{L^\infty}\big\|\psi_t\big(|\nabla|\big)u\big\|_{L^2}\\
&\lesssim t^{s}|a|_{L^\infty}\big\|\psi_t\big(|\nabla|\big)u\big\|_{L^2}.
\end{aligned}
$$
So by the Peter-Weyl theorem,
$$
\begin{aligned}
\|T_au\|_{H^s}^2
&\simeq\sum_{\xi\in\DuGp}d_\xi|\xi|^{2s}\left\llbracket\int_0^\infty h_t(|\xi|)\cdot
\Ft{v(t)}(\xi)\frac{dt}{t}\right\rrbracket^2\\
&\leq \sum_{\xi\in\DuGp}d_\xi\int_0^\infty
\big|h_t(|\xi|)\big|^2\frac{dt}{t}\cdot\int_0^\infty
|\xi|^{2s}\big\llbracket\Ft{v(t)}(\xi)\big\rrbracket^2\frac{dt}{t}\\
&\lesssim \int_0^\infty\|v(t)\|_{H^s}^2\frac{dt}{t}
\lesssim |a|_{L^\infty}^2\|u\|_{H^s}^2,
\end{aligned}
$$
where we used Proposition \ref{LPHs} at the last step. Note that the integrals in the above are not singular since the integrands all vanish near $t=0$.

Just as in the Euclidean case, we have the para-product decomposition on $\Gp$:
\begin{theorem}\label{ParaPrd}
If $s>0$, $a,u\in (L^\infty\cap H^s)(\Gp)$, then
$$
au=T_au+T_ua+R(a,u),
$$
where the smoothing remainder 
$$
\|R(a,u)\|_{H^s}\lesssim |a|_{L^\infty}\|u\|_{H^s}.
$$
\end{theorem}
\begin{proof}
It suffices to investigate the remainder $R(a,u)$. Using the continuous Littlewood-Paley decomposition (\ref{LPCont}),
$$
R(a,u)=\int_{\substack{2^{-10}\leq t_1/t_2\leq 2^{10}\\ 1\leq t_1,t_2}}\psi_{t_1}\big(|\nabla|\big)a\cdot\psi_{t_2}\big(|\nabla|\big)u\cdot \frac{dt_1dt_2}{t_1t_2}
+\phi\big(|\nabla|\big)a\cdot\phi\big(|\nabla|\big)u.
$$
Only the integral is non-trivial. We make a change of variable $(t_1,t_2)=t(\cos\varphi,\sin\varphi)$, so that it equals
$$
\int_{2^{-10}\leq \tan\varphi\leq 2^{10}}\left(\int_1^\infty\psi_{t\cos\varphi}\big(|\nabla|\big)a\cdot\psi_{t\sin\varphi}\big(|\nabla|\big)u\cdot \frac{dt}{t}\right)\sin\varphi d\varphi.
$$
Then $R(a,u)=\Op(R[a])u$, where the symbol
$$
R[a](x,\xi)=
\int_{2^{-10}\leq \tan\varphi\leq 2^{10}}\left(\int_1^\infty\psi_{t\cos\varphi}\big(|\nabla|\big)a
\cdot\psi_{t\sin\varphi}(|\xi|)\I[\xi]\cdot \frac{dt}{t}\right)\sin\varphi d\varphi.
$$
Just as in the proof of the previous proposition, we employ Lemma \ref{CutOffKer} and Theorem \ref{Multm} to conclude, for a strongly RT-admissible tuple $q$,
$$
\mathbf{M}^0_{k,l;q}(R[a])\lesssim_{k,l}\size[\xi]^{k-l}|a|_{L^\infty}.
$$
Thus by Stein's theorem, $\Op(R[a])$ maps $H^s$ to itself for $s>0$.
\end{proof}
\begin{remark}\label{R(a)Symbol}
If in addition $a\in C^r_*(\Gp)$, then from the spectral characterization of Zygmund spaces, we find that in fact $R[a]\in \mathscr{S}^{-r}_{1,1}$, so 
$$
\|R(a,u)\|_{H^s}\lesssim|a|_{C^r_*}\|u\|_{H^{s-r}}.
$$
\end{remark}

We also have a direct generalization of Bony's para-linearization theorem, proposed by Bony \cite{Bony1981}. Such a result is neither proved true or false for general compact manifolds in \cite{KR2006}. 
\begin{theorem}[Bony]\label{Bony}
Suppose $F\in C^\infty(\mathbb{C};\mathbb{C})$ (understood as smooth mapping on the plane instead of holomorphic function), $F(0)=0$. Let $r>0$, and suppose $u\in C^r_*$. Write $u_t=\phi(D/t)u$, where $\phi$ is the function in the Littlewood-Paley decomposition (\ref{LPCont}). Then with the symbol
$$
l_u(x,\xi):=\int_1^\infty F'(u_t(x))\cdot\psi_t\big(|\xi|\big)\cdot\I[\xi]\frac{dt}{t},
$$
there holds $F(u)=F(u_1)+\Op(l_u)u$, and we have the para-linearization formula
$$
F(u)=F(u_1)+T_{F'(u)}u+\Op(R[u])u
$$
with the symbol $R[u](x,\xi)\in \mathscr{S}^{-r}_{1,1}$. Consequently, if $u\in C^r_*\cap H^s$, then $F(u)-T_{F'(u)}u\in H^{s+r}$. In particular, for $s>n/2$, $F(u)-T_{F'(u)}u\in H^{2s-n/2}$.
\end{theorem}
\begin{proof}
The proof resembles that of Theorem 10.3.1. in \cite{Hormander1997} very much. We first set $u_t=\phi_t\big(|\nabla|\big)u$, then define a symbol
$$
l_u(x,\xi):=\int_1^\infty F'(u_t(x))\cdot\psi_t(|\xi|)\I[\xi]\frac{dt}{t}
$$
That $F(u)=F(u_1)+l_u(x,D)u$ is obtained by a direct computation already done when estimating the symbol of $T_a$, which also shows $l_u(x,\xi)\in \mathscr{S}^0_{1,1}$. Working in local coordinate charts, we find $F(u),F'(u)\in C^r_*$. 

Let us set $p(x,\xi)$ as the symbol of the para-product operator $T_{F'(u)}$. For any strongly RT-admissible tuple $q$, we estimate the $\mathbf{M}_{k,l;q}^{-r}$ norm of
$$
R[u](x,\xi):=l_u(x,\xi)-p(x,\xi)
=\int_1^\infty
\Big[F'(u_t(x))-\phi_{2^{-10}t}\big(|\nabla|\big)F'(u(x))\Big]\psi_t(|\xi|)\I[\xi]\frac{dt}{t}.
$$
In view of Lemma \ref{CutOffKer} and Corollary \ref{Fischer} once again, all we need to show is that the magnitude of
\begin{equation}\label{2Alpha}
X^\alpha_x\Big[F'(u_t)-\phi_{2^{-10}t}\big(|\nabla|\big)F'(u)\Big]
\end{equation}
has an upper bound $K_\alpha(|u|_{C^r_*}) t^{|\alpha|-r}$, where $K_\alpha$ is an increasing function. 

For $\alpha=0$, we find the quantity (\ref{2Alpha}) equals
$$
F'(u_t)-\phi_{2^{-10}t}\big(|\nabla|\big)F'(u)
=F'(u_t)-F'(u)+F'(u)-\phi_{2^{-10}t}\big(|\nabla|\big)F'(u).
$$
By Corollary \ref{LPZygCor}, we have $|u-u_t|\lesssim t^{-r}|u|_{C^r_*}$, and since $F'(u)\in C^r_*$, we find that (\ref{2Alpha}) is controlled by $t^{-r}$ for $\alpha=0$. 

For $|\alpha|>r$, the second term in (\ref{2Alpha}) is controlled by $t^{|\alpha|-r}|u|_{C^r_*}$ with the aid of Corollary \ref{LPZygCor}, since $F'(u)\in C^r_*$. We note that being $C^r_*$ is a local property, so it is legitimate to work in any local coordinate chart $\{x\}$ and estimate the usual local partial derivatives of the ambient function. Using the chain rule, we find that the first term in (\ref{2Alpha}) is a $C^\infty$-linear combination of terms of the form
$$
F^{(1+|\alpha|-|\beta|)}(u_t)\left(\partial_x^{\beta_1}u_t\right)^{j_1}\cdots\left(\partial_x^{\beta_k}u_t\right)^{j_k},
$$
where $\beta=\beta_1+\cdots+\beta_k$, and $j_1|\beta_1|+\cdots+j_k|\beta_k|=|\alpha|$. Using Corollary \ref{LPZygCor} again, we can bound this term by
$$
K(|u|_{C^r_*})\prod_{i:\beta_i> r}\left(|u|_{C^r_*}t\right)^{j_i(|\beta_i|-r)}
\cdot\prod_{i:\beta_i\leq r}\left(|u|_{C^r_*}\log(1+t)\right)^{j_i}
=O(t^{|\alpha|-r}).
$$
As for $0<|\alpha|\leq r$, we can just use the classical interpolation inequality for smooth functions with all derivatives bounded, in between 0 and $r+1$:
$$
|D^jf|_{L^\infty}\lesssim_{j,k}|f|_{C^k}^{j/k}\cdot|f|_{L^\infty}^{1-j/k},
\quad
0<j<k.
$$
This implies that (\ref{2Alpha}) is controlled by $t^{|\alpha|-r}$, hence ensures $R[u]\in \mathscr{S}^{-r}_{1,1}$. The claim of the theorem now follows from Stein's theorem.
\end{proof}

\section{Para-differential Operators on Compact Lie Group}\label{4}
\subsection{Rough Symbols and Their Smoothing}
Because of the spectral localization property on $\Gp$, the theory of para-differential operators resembles that on $\mathbb{R}^n$ very much. The details of the latter may be found in Chapter 9-10 in \cite{Hormander1997}. For completeness and convenience of future reference, we still write down the details here.

\begin{definition}\label{Rough}
For $r\geq0$ and $m\in\mathbb{R}$, define the symbol class $\mathcal{A}_r^m(\Gp)$ to be the collection of all symbols $a(x,\xi)$ on $\Gp\times\DuGp$, such that if $q=(q_i)_{i=1}^M$ is a strongly RT-admissible tuple, then
$$
\big\|\Df_{q,\xi}^\beta a(x,\xi)\big\|_{r;x}\leq C_\beta\size[\xi]^{m-|\beta|},\quad\forall\xi\in\DuGp.
$$
Here the norm $\|a(x,\xi)\|_{r;x}$ for $a(\cdot,\xi):\Gp\to\mathrm{End}(\Hh[\xi])$ is defined by
$$
\|a(x,\xi)\|_{r;x}:=
\sup_{B^*\in\mathrm{End}(\Hh[\xi])^*}\frac{\big|\langle B^*, a(x,\xi)\rangle\big|_{C^r_{*x}}}{\|B^*\|},
$$
as in Remark \ref{VectZyg}, that is we consider $\mathrm{End}(\Hh[\xi])$ merely as a normed space instead of a normed algebra. Introduce the following norm on $\mathcal{A}_r^m(\Gp)$:
$$
\mathbf{W}^{m;r}_{l;q}(a):=\sup_{\xi\in\DuGp}\sum_{|\beta|\leq l}\size[\xi]^{|\beta|-m}\big\|\Df_{q,\xi}^\beta a(x,\xi)\big\|_{r;x}
$$
\end{definition}
We note that Fischer's Lemma \ref{Fisqq'Lem} ensures that the definition of $\mathcal{A}^m_r(\Gp)$ does not depend on the choice of $q$, so we can choose any set of functions $q$ that facilitates our computation. In particular, the fundamental tuple $Q$ defined in (\ref{QFund}) is a convenient choice.

\begin{definition}[Admissible cut-off]\label{AdmCutoff}
An admissible cut-off function $\chi$ with parameter $\delta\in(0,1/2)$ is a smooth function on $\mathbb{R}\times\mathbb{R}$ such that
$$
\chi(\mu,\lambda)=\left\{
\begin{aligned}
1,&\quad |\mu|\leq\frac{\delta}{2}\size[\lambda],\\
0,&\quad |\mu|\geq\delta\size[\lambda],
\end{aligned}
\right.
$$
and there also holds
$$
|\partial_\mu^k\partial_\lambda^l\chi(\mu,\lambda)|
\lesssim_{k,l}\langle\lambda\rangle^{-k-l}.
$$
\end{definition}
An obvious choice of admissible cut-off function is the one that we used to construct para-product (\ref{T_au}):
\begin{equation}\label{AdmCutoffdelta}
\chi(\mu,\lambda)=\int_1^\infty\phi_{\delta t/2}(\mu)\psi_t(\lambda)\frac{dt}{t}.
\end{equation}
However, it will be shown shortly that there is a lot of flexibility in choosing an admissible cut-off function.

\begin{definition}[Para-differential Operator]
Let $\chi$ be an admissible cut-off function with some parameter $\delta$. For $a\in\mathcal{A}_r^m(\Gp)$, set $a^\chi (x,\xi)=\chi(|\nabla_x|,|\xi|)a(x,\xi)$, the regularized symbol corresponding to $a$. Define the para-differential operator $T_a^\chi$ corresponding to $a\in\mathcal{A}_r^m(\Gp)$ as
$$
T_a^\chi u(x):= \Op(a^\chi)u(x).
$$
\end{definition}

We immediately deduce that $a^\chi$ is a (1,1) symbol satisfying the spectral condition with parameter $\delta$ with improved growth estimate:
\begin{proposition}\label{ParaDiff1}
Suppose $r\geq0$. Fix a basis $X_1,\cdots,X_n$ of $\mathfrak{g}$, and define $X^\alpha$ as in Proposition \ref{NormalOrder}. Suppose $a\in\mathcal{A}_r^m(\Gp)$ and $\chi$ is an admissible cut-off function with parameter $\delta$. If $r=0$, then $a^\chi \in\Sigma^{m+\varepsilon}_\delta(\Gp)\subset \mathscr{S}^{m+\varepsilon}_{1,1}(\Gp)$. If $r>0$, then $a^\chi \in\Sigma^{m}_\delta(\Gp)\subset \mathscr{S}^{m}_{1,1}(\Gp)$ for some $\delta>0$. In fact, for any strongly RT-admissible tuple $q$, we have
$$
\big\|X_x^\alpha\Df_{q,\xi}^\beta a^\chi (x,\xi)\big\|
\lesssim_{\alpha,\beta;q}\left\{
\begin{aligned}
    & \mathbf{W}^{m;r}_{|\beta|;q}(a)\cdot\size[\xi]^{m+(|\alpha|-r)_+-|\beta|} &\quad |\alpha|\neq r \\
    & \mathbf{W}^{m;r}_{|\beta|;q}(a)\cdot\size[\xi]^{m-|\beta|}\log(1+\size[\xi]) &\quad |\alpha|=r
\end{aligned}
\right.,
$$
where $s_+=\max(s,0)$. The second inequality matters only when $r$ is a positive integer.
\end{proposition}
\begin{proof}
That $a^\chi$ satisfies the spectral condition with parameter $\delta$ is immediate from the definition: the Fourier cut-off operator $\chi(|\nabla_x|,|\xi|)$ annihilates all frequencies of size greater than $\delta\size[\xi]$. We turn to verify the estimates for the $\mathscr{S}^m_{1,1}$ norms of $a^\chi$. By Fischer's lemma i.e. Lemma \ref{Fisqq'Lem}, it suffices to prove the inequality for $q=Q$, the fundamental tuple (\ref{QFund}). The advantage is that the corresponding RT difference operators have Leibniz type property, and all the $\Df_Q^\beta a^\chi$'s also satisfy spectral conditions (with parameters proportional to $\delta$) by Lemma \ref{Sigmadelta}.

For $\beta=0$, the estimate for $\|X^\alpha_xa^\chi(x,\xi)\|$ follows from the vector-valued version of Corollary \ref{LPZygCor} (see Remark \ref{VectZyg}), where the normed space considered is $\Hh[\xi]$. To obtain the estimate for $\|X_x^\alpha\Df_{Q,\xi}^\beta a^\chi (x,\xi)\|
$ with $\beta\neq0$, we just have to apply Theorem \ref{Multm}.
\end{proof}

We notice that the norm of $a^\chi(x,\xi)$ grows slower upon differentiation in $x$ than generic $\mathscr{S}^m_{1,1}$ symbols. This property is called by H\"{o}rmander as having \emph{reduced order $m-r$} for symbols on $\mathbb{R}^n$. 

The next proposition shows that the para-differential operator $T_a^\chi=\Op(a^\chi)$ extracts the ``highest order differentiation" from $\Op(a)$. The para-product decomposition can be seen as a special case of it.

\begin{proposition}\label{a-a^chi}
Suppose $r>0$, $a\in\mathcal{A}_r^m(\Gp)$. Fix a basis $X_1,\cdots,X_n$ of $\mathfrak{g}$, and define $X^\alpha$ as in Proposition \ref{NormalOrder}. If $\chi$ is an admissible cut-off function, then $a-a^\chi \in\mathcal{A}_0^{m-r}$, and in fact for any strongly RT-admissible tuple $q$,
$$
\sup_{x\in\Gp}\big\|\Df_{q,\xi}^\beta(a- a^\chi)(x,\xi)\big\|
\lesssim_{\beta,q} \mathbf{W}^{m;r}_{|\beta|;q}(a)\cdot\langle\xi\rangle^{m-r-|\beta|}.
$$
\end{proposition}
\begin{proof}
We still just need to prove for $q=Q$, the fundamental tuple. By induction, it suffices to prove for $|\beta|=1$. By the Leibniz type property (\ref{Leibniz}), setting $M_i$ to be the dimension of the $i$'th fundamental representation,
$$
\Df_{Q_i,\xi}(a-a^\chi)(x,\xi)
=\big(\Df_{Q_i,\xi}a-(\Df_{Q_i,\xi}a)^\chi\big)(x,\xi)
-\sum_{j,k=1}^{M_i}c^i_{j,k}\Df_{Q_j,\xi}(\chi\cdot\I[\xi])\big(|\nabla_x|,|\xi|\big)\Df_{Q_k}a(x,\xi).
$$
By Lemma \ref{Sigmadelta}, $\Df_{Q_i}a$ satisfies the spectral condition with parameter proportional to $\delta$, so we may simply apply Corollary \ref{LPZygCor} (and Remark \ref{VectZyg}) to conclude that 
$$
\sup_{x\in\Gp}\left\|\big(\Df_{Q_i,\xi}a-(\Df_{Q_i,\xi}a)^\chi\big)(x,\xi)\right\|
\lesssim_{Q} \mathbf{W}^{m;r}_{1;Q}(a)\cdot\langle\xi\rangle^{m-r-1}.
$$
The sum $\sum_{j,k=1}^{M_i}$ needs greater care to handle. For a fixed $\eta\in\DuGp$, Theorem \ref{Multm} ensures that the symbol $\Df_{Q_j,\xi}(\chi\cdot \I[\xi])\big(|\eta|,|\xi|\big)$ is in a bounded subset in $\mathscr{S}^{m-1}_{1,1}$, and by Lemma \ref{DiffVanish} it does not vanish only if 
$$
\frac{1}{\delta}|\eta|-C\leq|\xi|\leq\frac{2}{\delta}|\eta|+C,
$$
where $C$ depends on the algebraic structure of $\Gp$ only. Consequently, $\Df_{Q_j,\xi}(\chi\cdot \I[\xi])\big(|\eta|,|\xi|\big)\neq0$ only if $|\eta|/\size[\xi]\simeq \delta$, so when the highest weight of $\xi$ is large enough, the Fourier support of 
$$
\Df_{Q_j,\xi}(\chi\cdot\I[\xi])\big(|\nabla_x|,|\xi|\big)\Df_{Q_k}a(x,\xi)
$$
with respect to $x$ is contained in $\mathfrak{S}[C\delta\size[\xi],C'\delta\size[\xi]]$, with $C,C'>0$ depending on the algebraic structure of $\Gp$ only. Corollary \ref{LPZygCor} and Remark \ref{VectZyg} then implies that
$$
\sup_{x\in\Gp}\left\|\Df_{Q_j,\xi}(\chi\cdot\I[\xi])\big(|\nabla_x|,|\xi|\big)\Df_{Q_k}a(x,\xi)\right\|
\lesssim_Q
\mathbf{W}^{m;r}_{1;Q}(a)\cdot\langle\xi\rangle^{m-r-1}.
$$
\end{proof}

We next estimate the difference $a^{\chi_1}-a^{\chi_2}$ for different choices of admissible cut-off functions. The result shows that the choice of $\chi$ is in fact very flexible:
\begin{proposition}\label{FreedomForCut-off}
Suppose $r\geq0$, $\chi_1,\chi_2$ are two admissible cut-off functions with parameters $0<\delta_1<\delta_2<1/2$ respectively. Fix a basis $X_1,\cdots,X_n$ of $\mathfrak{g}$, and define $X^\alpha$ as in Proposition \ref{NormalOrder}. Then for $a\in\mathcal{A}_r^m$, the difference $a^{\chi_1}- a^{\chi_2}\in \Sigma^{m-r}_{\delta_2}$. In fact for any strongly RT-admissible tuple $q$,
$$
\big\|X^\alpha_x\Df_{q,\xi}^\beta (a^{\chi_1}- a^{\chi_2})(x,\xi)\big\|
\lesssim_{\alpha,\beta,q}
\mathbf{W}^{m;r}_{|\beta|;q}(a)\cdot\langle\xi\rangle^{m-r+|\alpha|-|\beta|}.
$$
Thus for any $s\in\mathbb{R}$, $T_a^{\chi_1}-T_a^{\chi_2}$ maps $H^{s+m}$ continuously to $H^s$, and
$$
\|(T_a^{\chi_1}-T_a^{\chi_2})u\|_{H^s}
\lesssim_{s,q}\mathbf{W}^{m;r}_{n+2;q}(a)\|u\|_{H^{s+m-r}}.
$$
\end{proposition}
\begin{proof}
This follows immediately from an observation and Corollary \ref{LPZygCor} (also Remark \ref{VectZyg}): the Fourier support of $\chi_1(\eta,\xi)-\chi_2(\eta,\xi)$ with respect to $x$ is in the set $\mathfrak{S}[\delta_1\size[\xi],\delta_2\size[\xi]]$.
\end{proof}
Thus, given a rough symbol $a\in\mathcal{A}^m_r$ with $r>0$, it is legitimate to abbreviate the dependence on admissible cut-off function, and define \emph{the} para-differential operator $T_a$ modulo $\Op\big(\Sigma^{m-r}_{<1/2}\big)$.

\subsection{Symbolic Calculus of Para-differential Operators}
In this subsection, we state the core results in our toolbox: the composition and adjoint formula for para-differential operators. The spirit of proof does not differ much from that in Métivier's book \cite{Met2008}, but technically we have to be extra careful dealing with the operator norm of symbols. The spectral localization property i.e. Corollary \ref{SpecPrd} and results obtained in Subsection 4.1. will be used repeatedly.

\begin{definition}
Let $r>0$ be a real number. Let $q$ be a RT-admissible tuple, and $X_q^{(\alpha)}$ be the left-invariant differential operators as in Proposition \ref{TaylorGp}. If $a,b$ are any symbol of at least $C^r_*$ regularity in $x$, define
$$
(a\#_{r;q} b)(x,\xi):=\sum_{\alpha:|\alpha|\leq r}\Df_{q,\xi}^\alpha a(x,\xi)\cdot X_{q,x}^{(\alpha)} b(x,\xi).
$$
\end{definition}

We now state the first main theorem of this section.

\begin{theorem}[Composition of Para-differential Operators, I]\label{Compo1}
Suppose $r>0$, $m,m'$ are real numbers. Let $q$ be a RT-admissible tuple, whose components are linear combinations of the fundamental tuple $Q$ of $\Gp$, and let $X_q^{(\alpha)}$ be the left-invariant differential operators as in Proposition \ref{TaylorGp}. Given $a\in \mathscr{S}^m_{1,1}$, $b\in\mathcal{A}_r^{m'}$, for an admissible cut-off function $\chi$ with parameter $\delta$ sufficiently small (depending on the magnitude of $r$ and the algebraic structure of $\Gp$ only),
$$
\Op(a)\circ\Op(b^\chi)
-\Op(a\#_{r;q} b^\chi)\in \Op \mathscr{S}^{m+m'-r}_{1,1}(\Gp).
$$
More precisely, the operator norm of $\Op(a)\circ\Op(b^\chi)
-\Op(a\#_{r;q} b^\chi)$ for $H^{s+m+m'-r}\to H^s$ is bounded by
$$
C_s\mathbf{M}^{m}_{l,n+2;q}(a)\mathbf{W}^{m';r}_{l;q}(b),
$$
where the integer $l$ does not depend on $a,b$.
\end{theorem}
\begin{proof}
We observe a simple fact: for symbols $\sigma_1,\sigma_2$ of type (1,1), the symbol of $\Op(\sigma_1)\circ\Op(\sigma_2)$ is
\begin{equation}\label{OpaOpb}
\int_{\Gp}\mathcal{K}_{\sigma_1}(x,y)\xi^*(y)\sigma_2(xy^{-1},\xi)dy,
\end{equation}
where $\mathcal{K}_{\sigma_1}$ is the convolution kernel of $\Op(\sigma_1)$, defined by (\ref{Symbol-Ker})-(\ref{Ker-Conv}).

Let us now examine the symbol $\sigma$ of $\Op(a)\circ\Op(b^\chi)$. We start from $b^\chi(x,\xi)$: since $b^\chi(x,\xi)$ satisfies the spectral condition with parameter $\delta$, we can choose another admissible cut-off function $\chi_1(\mu,\lambda)$ with parameter $2\delta$, and find 
$$
\Ft{b^\chi}(\eta,\xi)=\chi_1(|\eta|,|\xi|)\Ft{b^\chi}(\eta,\xi.
$$
Converting back to the physical space, this implies
$$
b^\chi(x,\xi)=\int_\Gp \check{\chi}_1(z^{-1}x,|\xi|)b^\chi(z)dz.
$$
Here with a little abuse of notation, we write $\check{\chi}_1(y,|\xi|)$ for the Fourier inversion of $\chi_1(|\eta|,|\xi|)\I[\eta]$ with respect to $\eta$. Note that $\chi_1(|\eta|,|\xi|)\I[\eta]$ is a scaling on each representation space $\Hh[\eta]$, so the order of convolution is irrelevant. By formula (\ref{OpaOpb}), 
$$
\begin{aligned}
\sigma(x,\xi)
&=\int_{\Gp}\mathcal{K}_{a}(x,y)\xi^*(y)\left(\int_\Gp \check{\chi}_1(z^{-1}xy^{-1},|\xi|)b^\chi(z,\xi)dz\right)dy\\
&=\int_{\Gp}\mathcal{K}_{a}(x,y)\xi^*(y)\left(\int_\Gp \check{\chi}_1(zy^{-1},|\xi|)b^\chi(xz^{-1},\xi)dz\right)dy.
\end{aligned}
$$
Using Proposition \ref{TaylorGp} with $N$ being the smallest integer $>r$, we compute the inner integral as 
$$
\begin{aligned}
\int_\Gp \check{\chi}_1(zy^{-1},|\xi|)b^\chi(xz^{-1},\xi)dz
&=\sum_{|\alpha|\leq r}\left(\int_\Gp \check{\chi}_1(zy^{-1},|\xi|)q^\alpha(z)dz\right)X^{(\alpha)}_{q,x}b^\chi(x,\xi)\\
&\quad+\int_\Gp \check{\chi}_1(zy^{-1},|\xi|)R_N\big(b^\chi(\cdot,\xi);x,z^{-1}\big)dz
\end{aligned}
$$
The integrals in the sum $\sum_{|\alpha|\leq r}$ are of convolution type. By Proposition \ref{SpecPrd0}, the Fourier support of $q^\alpha$ for $|\alpha|\leq r$ is within $\mathfrak{S}[0,C]$, where $C$ depends on $r$ and the algebraic structure of $\Gp$ only. Thus, for $|\xi|\geq 2C/\delta$, we have
$$
\int_\Gp \check{\chi}_1(zy^{-1},|\xi|)q^\alpha(z)dz=q^\alpha(y),
$$
and consequently,
\begin{equation}\label{sigmaTemp1}
\begin{aligned}
\sigma(x,\xi)
&=\sum_{|\alpha|\leq r}\int_{\Gp}\mathcal{K}_{a}(x,y)\xi^*(y)q^\alpha(y)X^{(\alpha)}_{q,x}b^\chi(x,\xi)dy\\
&\quad+\int_{\Gp}\mathcal{K}_{a}(x,y)\xi^*(y)
\left(\int_\Gp \check{\chi}_1(zy^{-1},|\xi|)R_N\big(b^\chi(\cdot,\xi);x,z^{-1}\big)dz\right)dy\\
&=:\sum_{|\alpha|\leq r}\Df_{q,\xi}^\alpha a(x,\xi)\cdot X_{q,x}^{(\alpha)}b^\chi(x,\xi)+\varsigma_N(x,\xi).
\end{aligned}
\end{equation}

The sum in the right-hand-side of (\ref{sigmaTemp1}) is just $a\#_{r;q}b^\chi$, which is in $\mathscr{S}^{m+m'}_{1,1}$. There is then only one assertion to be verified: the term $\varsigma_N(x,\xi)$ is of class $\mathscr{S}^{m+m'-r}_{1,1}$, from which it follows that the symbol $\sigma$ really is of class $\mathscr{S}^{m+m'}_{1,1}$. Note that the lower frequency terms give smoothing symbols in $\mathscr{S}^{-\infty}$, so we just have to verify the assertion for $|\xi|\geq 2C/\delta$. 

Let us first show that $\varsigma_N(x,\xi)\lesssim\size[\xi]^{m+m'-r}$. We can change the order of integration as follows:
\begin{equation}\label{sigmaTemp2}
\begin{aligned}
\varsigma_N(x,\xi)
=\int_{\Gp}
\left(\int_\Gp \mathcal{K}_a(x,y)\xi^*(y)\check{\chi}_1(zy^{-1},|\xi|)dy\right)R_N\big(b^\chi(\cdot,\xi);x,z^{-1}\big)dz.
\end{aligned}
\end{equation}
Just as in the proof of Stein's theorem, we decompose $\mathcal{K}_a=\sum_{j=0}^\infty \mathcal{K}_{a_j}$, where $\mathcal{K}_{a_j}$ is the convolution kernel of $a_j(x,\xi):=a(x,\xi)\vartheta_j\big(|\xi|\big)$, with $\vartheta$ as in (\ref{LPDisc}). For the inner integral with respect to $y$, each summand
$$
P_j(x,z):=\int_\Gp \mathcal{K}_{a_j}(x,y)\xi^*(y)\check{\chi}_1(zy^{-1},|\xi|)dy
$$
is of convolution type. By spectral property i.e. Corollary \ref{SpecPrd}, the Fourier support of $\mathcal{K}_{a_j}(x,y)\xi^*(y)$ with respect to $y$ is in the set 
$$
\mathfrak{S}\Big[c\big|2^j-|\xi|\big|,2^j+|\xi|\Big].
$$
Hence $\Ft{P_j}(x,\eta)$ (the Fourier transform is taken with respect to $z$) does not vanish only for 
$$
c\big|2^j-|\xi|\big|\leq 2\delta\size[\xi].
$$
If we choose $\delta\ll c$, this implies that $2^j$ must be comparable with $\size[\xi]$. As a result,
$$
\varsigma_N(x,\xi)
=\sum_{j:2^j\simeq\size[\xi]}\int_{\Gp}\mathcal{K}_{a_j}(x,y)\xi^*(y)
\left(\int_\Gp \check{\chi}_1(zy^{-1},|\xi|)R_N\big(b^\chi(\cdot,\xi);x,z^{-1}\big)dz\right)dy.
$$
Just as inequalitites (\ref{SteinIneq1})-(\ref{SteinIneq2}) in the proof of Stein's theorem, the estimate
\begin{equation}\label{sigmaTemp3}
\big(1+|2^j\dist(y,e)|^2\big)^{L/2}\cdot|\mathcal{K}_{a_j}(x,y)|
\lesssim\mathbf{M}^m_{0,L;Q}(a)\cdot2^{jm+jn}
\end{equation}
is valid for $L\in\mathbb{N}$. The remainder in Taylor's formula can be estimated using the remark following Proposition \ref{ParaDiff1}: since $N>r$, we have, by the improved growth estimate (Proposition \ref{ParaDiff1}),
$$
\begin{aligned}
\big\|R_N\big(b^\chi(\cdot,\xi);x,z^{-1}\big)\big\|
&\lesssim \|b^\chi(x,\xi)\|_{C^N_x;\mathrm{End}(\Hh[\xi])}\cdot\dist(z,e)^N \\
&\lesssim \mathbf{W}^{m;r}_{0;Q}(b)\cdot\size[\xi]^{m'+N-r}\dist(z,e)^N.
\end{aligned}
$$
Here the norm is the operator norm on $\mathrm{End}(\Hh[\xi])$, and the implicit constants do not depend on $\xi$. By Corollary \ref{ConvVanish}, noting the integrand is smooth and equals $O(\dist(z,e)^N)$, we thus estimate
\begin{equation}\label{sigmaTemp4}
\begin{aligned}
\int_\Gp \check{\chi}_1(zy^{-1},|\xi|)R_N\big(b^\chi(\cdot,\xi);x,z^{-1}\big)dz
=\sum_{k=0}^N\upsilon_k(x,y,\xi),
\end{aligned}
\end{equation}
where each $\upsilon_k(x,y,\xi)$ is smooth in $x$, and satisfies
$$
\|\upsilon_k(x,y,\xi)\|
\lesssim\mathbf{W}^{m;r}_{0;Q}(b)\cdot\size[\xi]^{m'+k-r}\dist(y,e)^k.
$$
Combining every inequality (\ref{sigmaTemp3})-(\ref{sigmaTemp4}), using Lemma \ref{Fisqq'Lem}, we estimate the operator norm of $\varsigma_N(x,\xi)$ as (noting that $\|\xi^*(y)\|\equiv1$)
$$
\begin{aligned}
\|\varsigma_N(x,\xi)\|
&\lesssim \sum_{j:2^j\simeq\size[\xi]}\sum_{k=0}^N
\left\|\int_{\Gp}\mathcal{K}_{a_j}(x,y)\xi^*(y)\upsilon_k(x,y,\xi)dy\right\|\\
&\lesssim \sum_{k=0}^N\mathbf{M}^m_{0,n+1;Q}(a)\cdot\mathbf{W}^{m;r}_{0;Q}(b)\cdot\size[\xi]^{m'+k-r}
\sum_{j:2^j\simeq\size[\xi]}2^{jm+jn}
\int_{\Gp}\frac{\min\left(\dist(y,e)^k,1\right)dy}{\big(1+|2^j\dist(y,e)|^2\big)^{(n+1)/2}}\\
&\lesssim \mathbf{M}^m_{0,n+1;Q}(a)\cdot\mathbf{W}^{m;r}_{0;Q}(b)\cdot\size[\xi]^{m+m'-r}.
\end{aligned}
$$
This shows that the second sum in (\ref{sigmaTemp1}) is controlled by $\size[\xi]^{m+m'-r}$.

Finally, we use induction on the order of differentiation in (\ref{OpaOpb}). Noticing that $\Df_{Q^\beta,\xi}\big(\xi^*(y)\big)=Q^\beta(y)\xi^*(y)$, we have, for example, given any $X\in\mathfrak{g}$,
\begin{equation}\label{sigmaTemp5}
X_x\sigma(x,\xi)
=\int_{\Gp}X_x\mathcal{K}_a(x,y)\xi^*(y)b^\chi(xy^{-1},\xi)dy
+\int_{\Gp}\mathcal{K}_a(x,y)\xi^*(y)X_x\big(b^\chi(xy^{-1},\xi)\big)dy
\end{equation}
and by Leibniz type property of the fundamental tuple $Q$,
\begin{equation}\label{sigmaTemp6}
\begin{aligned}
\Df_{Q_i,\xi}\sigma(x,\xi)
&=\int_{\Gp}\mathcal{K}_a(x,y)Q_i(y)\xi^*(y)b^\chi(xy^{-1},\xi)dy+\int_{\Gp}\mathcal{K}_a(x,y)\xi^*(y)(\Df_{Q_i,\xi}b^\chi)(xy^{-1},\xi)dy\\
&\quad+\sum_{j,k}c_{j,k}^i\int_{\Gp}\mathcal{K}_a(x,y)Q_j(y)\xi^*(y)(\Df_{Q_k,\xi}b^\chi)(xy^{-1},\xi)dy.
\end{aligned}
\end{equation}
Each of the above integrals can be estimated in exactly the same method, with $m$ replaced by $m+1$ in (\ref{sigmaTemp5}) and by $m-1$ in (\ref{sigmaTemp6}); we just have to observe that $X_x\mathcal{K}_a(x,y)$ is the convolution kernel of $X_xa(x,\xi)$, and $\mathcal{K}_a(x,y)Q_i(y)$ is the convolution kernel of $\Df_{Q_i,\xi}a(x,\xi)$.
\end{proof}

Now if $\Op(a)$ is in fact a para-differential operator corresponding to some rough symbol of class $\mathcal{A}^m_r$, we can actually strengthen Theorem \ref{Compo1}, and deduce the composition formula of para-differential operators: 
\begin{theorem}[Composition of Para-differential Operators, II]\label{Compo2}
Suppose $r>0$, $m,m'$ are real numbers. Let $q$ be a RT-admissible tuple, whose components are are linear combinations of the fundamental tuple $Q$ of $\Gp$, and let $X_q^{(\alpha)}$ be the left-invariant differential operators as in Proposition \ref{TaylorGp}. Given $a\in \mathcal{A}_r^{m}$, $b\in\mathcal{A}_r^{m'}$, it follows that
$$
T_a\circ T_b-T_{a\#_{r,q}b}
\in\Op\Sigma_{<1/2}^{m+m'-r}.
$$
More precisely, the operator norm of $T_a\circ T_b-T_{a\#_{r,q}b}$ for $H^{s+m+m'-r}\to H^s$ is bounded by
$$
C_s\mathbf{W}^{m;r}_{l;q}(a)\mathbf{W}^{m';r}_{l;q}(b),
$$
where the integer $l$ does not depend on $a,b$.
\end{theorem}

Before we prove Theorem \ref{Compo2}, we state two auxiliary results.
\begin{lemma}\label{CompoAux1}
Suppose $r>0$, $m,m'$ are real numbers, and $a\in \mathcal{A}_r^{m}$, $b\in\mathcal{A}_r^{m'}$. There are constants $\delta_0\in(0,1/2)$ and $C>0$ depending only on the algebraic structure of $\Gp$ with the following property: if $\chi$ is an admissible cut-off function with parameter $\delta<\delta_0$, the symbol of $T_a^\chi\circ T_b^\chi$ is of class $\Sigma^{m+m'-r}_{C\delta}$.
\end{lemma}
\begin{lemma}\label{CompoAux2}
There is a constant $C>1$, depending on the algebraic structure of $\Gp$ only, with the following property. Let $A$ be a finite dimensional normed algebra. Suppose $0\leq r_1\leq r_2$, with $r_2>0$. Suppose $u\in C^{r_1}_*$, $v\in C^{r_2}_*$ are $A$-valued functions. Suppose $\vartheta\in C_0^\infty(\mathbb{R})$ is such that $\vartheta(\lambda)=1$ for $|\lambda|\leq1$ and vanishes for $|\lambda|\geq2$. Then for $t\geq1$,
$$
\left\|\vartheta\left(\frac{|\nabla|}{Ct}\right)(uv)-\vartheta\left(\frac{|\nabla|}{t}\right)u\cdot \vartheta\left(\frac{|\nabla|}{t}\right)v\right\|_{L^\infty;A}
\lesssim_{r_1,r_2}t^{-r_1}\|u\|_{C^{r_1}_*;A}\|v\|_{C^{r_2}_*;A}.
$$
The implicit constants do not depend on $A$.
\end{lemma}
\begin{proof}[Proof of Lemma \ref{CompoAux1}]
By formula (\ref{OpaOpb}), the symbol $\sigma$ of $T_a^\chi\circ T_b^\chi$ is 
$$
\begin{aligned}
\sigma(x,\xi)
&=\int_{\Gp}\mathcal{K}_{a^\chi}(x,y)\xi^*(y)b^\chi(xy^{-1},\xi)dy\\
&=\sum_{\zeta\in\DuGp}d_\zeta\int_\Gp\Tr\Big[a^\chi(x,\zeta)\zeta(y)\Big]\xi^*(y)b^\chi(xy^{-1},\xi)dy.
\end{aligned}
$$
Let us fix $\xi\in\DuGp$. By the spectral localization property of products (Corollary \ref{SpecPrd}), the Fourier support of $\zeta\otimes\xi^*$ is in 
$$
\mathfrak{S}\Big[c\big||\zeta|-|\xi|\big|,|\zeta|+|\xi|\Big],
$$
where $0<c<1$ depends on the algebraic structure of $\Gp$ only. Since the integral is of convolution type, we find that the Fourier support of $\int_\Gp (\zeta\otimes\xi)^*(y)b^\chi(xy^{-1},\xi)dy$ with respect to $x$ is contained in the set
$$
\mathfrak{S}\Big[c\big||\zeta|-|\xi|\big|,|\zeta|+|\xi|\Big]\cap \mathfrak{S}[0,\delta\size[\xi]],
$$
so the integral does not vanish only for $|\zeta|\leq C\size[\xi]$, where $C$ depends on the algebraic structure of $\Gp$ only. Now applying the spectral condition satisfied by $a^\chi(x,\zeta)$, using the spectral localization property of products again, we obtain that the Fourier support of $\sigma(x,\xi)$ is contained in the set $\mathfrak{S}[0,C\delta\size[\xi]]$. That $\sigma\in\mathscr{S}^{m+m'}_{1,1}$ for sufficiently small $\delta$ is a direct consequence of Theorem \ref{Compo1}.
\end{proof}
The proof of Lemma \ref{CompoAux2} is identical to that of Proposition 8.6.9. of \cite{Hormander1997}, with the application of Corollary \ref{SpecPrd}, Corollary \ref{LPZygCor} and its vector-valued version as described in Remark \ref{VectZyg}.

\begin{proof}[Proof of Theorem \ref{Compo2}]
We notice a simple fact: if $q$ is an RT-admissible tuple and $X_q^{(\alpha)}$ are the corresponding left-invariant differential operators as in Proposition \ref{TaylorGp}, then for $a\in\mathcal{A}^m_r$, $b\in\mathcal{A}^{m'}_r$, it follows that $a\#_{r;q}b$ is an element in $\bigcup_{j\leq r}\mathcal{A}_{r-j}^{m+m'-j}$ since by the product rule in Zygmund spaces,
\begin{equation}\label{Prd_r}
\sum_{|\alpha|=j}\Df_{q,\xi}^\alpha a(x,\xi)\cdot X_{q,x}^{(\alpha)} b(x,\xi)\in\mathcal{A}_{r-j}^{m+m'-j}.
\end{equation}

First of all, we need to justify the notation that disregards the admissible cut-off function. In fact, Lemma \ref{Sigmadelta}, Proposition \ref{FreedomForCut-off} and formula (\ref{Prd_r}) ensures that $T_{a\#_{r;q}b}$ is defined modulo $\Op\Sigma_{<1/2}^{m+m'-r}$, provided that the parameter of the admissible cut-off function is sufficiently small (depending on the algebraic structure of $\Gp$ only). We also compute, for admissible cut-off functions $\chi_1,\chi_2$ with parameter $\delta_1<\delta_2$ respectively,
$$
\Op(a^{\chi_1})\circ\Op(b^{\chi_1})-\Op(a^{\chi_2})\circ\Op(b^{\chi_2})
=\Op(a^{\chi_1}-a^{\chi_2})\circ\Op(b^{\chi_1})+\Op(a^{\chi_2})\circ\Op(b^{\chi_1}-b^{\chi_2}).
$$
By Proposition \ref{FreedomForCut-off}, $a^{\chi_1}-a^{\chi_2}\in\Sigma_{\delta_2}^{m-r}$, $b^{\chi_1}-b^{\chi_2}\in\Sigma_{\delta_2}^{m'-r}$. Thus if $\delta_1,\delta_2$ are sufficiently small (depending on the algebraic structure of $\Gp$ only), Lemma \ref{CompoAux1} then implies that the right-hand-side is in $\Op\Sigma_{C\delta_2}^{m+m'-r}\subset\Op\Sigma_{<1/2}^{m+m'-r}$. Thus $T_a\circ T_b$ is defined modulo $\Op\Sigma_{<1/2}^{m+m'-r}$. Applying Lemma \ref{Sigmadelta} and Theorem \ref{Compo1} once again, we find that if $\chi$ is an admissible cut-off function with sufficiently small parameter $\delta$ (depending on the algebraic structure of $\Gp$ only), then
$$
T_a^\chi\circ T_b^\chi-\Op\big(a^\chi\#_{r;q}b^\chi\big)
\in\Op\Sigma_{<1/2}^{m+m'-r}.
$$
From now on, we fix $\chi$ to be the function in (\ref{AdmCutoffdelta}).

It remains to show that $a^\chi\#_{r;q}b^\chi$ equals $(a\#_{r;q}b)^{\chi_1}$ modulo $\Op\Sigma_{<1/2}^{m+m'-r}$ for some suitable admissible cut-off function $\chi_1$. For simplicity we write $a^{(\alpha)}=\Df_{q,\xi}^\alpha a$, $b_{(\alpha)}=X_{q,x}^{(\alpha)} b$. We just need to show that for $|\alpha|\leq r$, each summand of $(a^\chi \#_r b^\chi )-(a\#_{r;q}b)^{\chi_1}$, being
$$
(a^\chi)^{(\alpha)}\cdot(b^\chi)_{(\alpha)}
-\left(a^{(\alpha)}\cdot b_{(\alpha)}\right)^{\chi_1},
$$
is in $\Sigma^{m+m'-r}_{<1/2}$. The only technical complexity comes from $(a^\chi)^{(\alpha)}$. In fact, from the proof of Proposition \ref{a-a^chi}, we know that
$$
\Df_{Q_i,\xi}a^\chi(x,\xi)
=(\Df_{Q_i,\xi}a)^\chi(x,\xi)
+\sum_{j,k=1}^{M_i}c^i_{j,k}\Df_{Q_j,\xi}(\chi\cdot\I[\xi])\big(|\nabla_x|,|\xi|\big)\Df_{Q_k}a(x,\xi),
$$
and belongs to $\Sigma_{C\delta}^{m-1}$ for some $C$ depending on the algebraic structure of $\Gp$ only. Thus, if $\delta$ is sufficiently small, we obtain by induction that
$$
(a^\chi)^{(\alpha)}\cdot(b^\chi)_{(\alpha)}=(a^{(\alpha)})^\chi\cdot(b_{(\alpha)})^\chi\,\mod\,\Sigma_{<1/2}^{m+m'-r},
$$
since $|\alpha|\leq r$. 

The claim is then reduced to checking
\begin{equation}\label{2summand}
(a^{(\alpha)})^\chi\cdot(b_{(\alpha)})^\chi
-\left(a^{(\alpha)}\cdot b_{(\alpha)}\right)^{\chi_1}
\in \Sigma^{m+m'-r}_{<1/2}.
\end{equation}
To verify this, we just need the technical Lemma \ref{CompoAux2}. Recall that we chose $\chi$ as in (\ref{LPCont}). We regard $\xi$ as fixed, then set $\chi_1(|\eta|,|\xi|)=\chi(|\eta|/C,|\xi|)$ where $C$ is as in Lemma \ref{CompoAux2}, and $\vartheta(|\eta|)=\chi_1(|\eta|,|\xi|)$. Since $a^{(\alpha)}\in \mathcal{A}_{r}^{m-|\alpha|}$, $b_{(\alpha)}\in \mathcal{A}_{r-|\alpha|}^{m}$, Lemma \ref{CompoAux2} gives
$$
\begin{aligned}
\left\|(a^{(\alpha)})^\chi\cdot(b_{(\alpha)})^\chi
-\left(a^{(\alpha)}\cdot b_{(\alpha)}\right)^{\chi_1}\right\|
&\lesssim 
\langle\xi\rangle^{-(r-|\alpha|)}\cdot\mathbf{W}^{m-|\alpha|;r}_{0;Q}(a)\langle\xi\rangle^{m-|\alpha|}
\cdot\mathbf{W}^{m';r-|\alpha|}_{0;Q}(b)\langle\xi\rangle^{m'}\\
&\lesssim \mathbf{W}^{m-|\alpha|;r}_{0;Q}(a)\cdot\mathbf{W}^{m';r-|\alpha|}_{0;Q}(b)\cdot\langle\xi\rangle^{m+m'-r}.
\end{aligned}
$$
Here the norm to be estimated is the operator norm on $\mathrm{End}(\Hh[\xi])$, and the implicit constants do not depend on $\xi$. Derivatives of (\ref{2summand}) can be estimated similarly. By Lemma \ref{Sigmadelta}, (\ref{2summand}) still satisfies a spectral condition. This completes the proof. 
\end{proof}

As a corollary, we obtain the commutator property of para-differential operators: the commutator of two para-differential operators of order $m$ and $m'$ gives rise to a para-differential operator of order $m+m'-1$. 

\begin{corollary}\label{ParaComm}
Suppose $r>0$, $m,m'$ are real numbers.  Given $a\in \mathcal{A}_r^{m}$, $b\in\mathcal{A}_r^{m'}$, the commutator $[T_a,T_b]$ is in the class $\Op\Sigma_{<1/2}^{m+m'-1}$. More precisely, the operator norm of $[T_a,T_b]$ for $H^{s+m+m'-1}\to H^s$ is bounded by
$$
C_s\mathbf{W}^{m;r}_{l;q}(a)\mathbf{W}^{m';r}_{l;q}(b),
$$
where the integer $l$ does not depend on $a,b$.
\end{corollary}
\begin{proof}
Imitating the proof of Proposition \ref{2OrderComm}, we find that the commutator symbol $[a,b](x,\xi)$ is of class $\mathcal{A}^{m+m'-1}_r$; we just need to repeat the proof of Proposition \ref{2OrderComm} and, in addition, estimate the $C^r_*$ norms with respect to $x$.

On the other hand, 
$$
\begin{aligned}
(a\#_{r;q} b)(x,\xi)-(b\#_{r;q} a)(x,\xi)
&=[a,b](x,\xi)\\
&\quad+\sum_{\alpha:1\leq|\alpha|\leq r}\left(\Df_{q,\xi}^\alpha a(x,\xi)\cdot X_{q,x}^{(\alpha)} b(x,\xi)
-\Df_{q,\xi}^\alpha b(x,\xi)\cdot X_{q,x}^{(\alpha)} a(x,\xi)\right).
\end{aligned}
$$
The sum is easily seen to give rise to symbols of order $m+m'-1$. Theorem \ref{Compo2} then shows that
$$
[T_a,T_b]
\in\Op\Sigma^{m+m'-1}_{<1/2}.
$$
\end{proof}

We can compare this with \cite{Delort2015}: in order to cast the normal form reduction for quasi-linear Hamiltonian Klein-Gordon equation, such commutator property is necessary. The definition of para-differential operator within out formalism is of course different from that in \cite{Delort2015}. While it is not known whether these two definitions coincide, this does not affect our goal, as a calculus is already available.

We also state the adjoint formula for para-differential operators and omit the proof.
\begin{theorem}[Adjoint of Para-differential Operator]\label{ParaAdj}
Suppose $r>0$, $m$ is a real number. Let $q$ be a RT-admissible tuple, whose components are from the fundamental tuple $Q$ of $\Gp$, and let $X_q^{(\alpha)}$ be the left-invariant differential operators as in Proposition \ref{TaylorGp}. Given $a\in \mathcal{A}_r^{m}$, the adjoint operator $T_a^*$ satisfies
$$
T_a^*-T_{a^{\bullet;r,q}}\in\Op\Sigma^{m-r}_{<1/2}
$$
where
$$
a^{\bullet;r,q}(x,\xi)=\sum_{|\alpha|\leq r}\left(\Df_{q,\xi}^\alpha X_{q,x}^{(\alpha)}a^*\right)(x,\xi).
$$
More precisely, the operator norm of $[T_a,T_b]$ for $H^{s+m+m'-1}\to H^s$ is bounded by
$$
C_s\mathbf{W}^{m;r}_{l;q}(a),
$$
where the integer $l$ does not depend on $a$.
\end{theorem}

\subsection{Symbols with Very Rough Regularity}
Sometimes it is also necessary to discuss symbols $a(x,\xi)$ which are merely $C^{-r}_*$ in $x$ with $r>0$. We have the following proposition:
\begin{proposition}\label{T_aNegIndex}
Suppose $r>0$. Fix a basis $X_1,\cdots,X_n$ of $\mathfrak{g}$, and define $X^\alpha$ as in Proposition \ref{NormalOrder}. Suppose $a\in\mathcal{A}_{-r}^m(\Gp)$\footnote{The definition of symbol class $\mathcal{A}_{r}^m(\Gp)$ of course can be directly extended beyond $r>0$.} and $\chi$ is an admissible cut-off function with parameter $\delta$. Then $a^\chi(x,\xi)\in\Sigma^{m+r}_{<1/2}$. In fact, for any strongly RT-admissible tuple $q$, we have
$$
\big\|X_x^\alpha\Df_{q,\xi}^\beta a^\chi (x,\xi)\big\|
\lesssim_{\alpha,\beta;q}
\mathbf{W}^{m;r}_{|\beta|;q}(a)\cdot\size[\xi]^{m+r+|\alpha|-|\beta|}.
$$
\end{proposition}
The proof is a mere application of the growth estimate indicated in Remark \ref{Zygmund<0}. Since such symbols are quite irrregular, we cannot expect that choosing a different admissible cut-off function would produce a negeligible error. But we may still refer $\Op(a^\chi)$ as a \emph{para-differential operator} $T_a$, which maps $H^{s}$ to $H^{s+r}$ continuously for any $s\in\mathbb{R}$. Fortunately, there is a simple trick to bypass its low regularity and incorporate it into the symbolic calculus that we just constructed. For example, if $0<r<1$ and $a\in\mathcal{A}^m_{-r}$, then we consider $\Delta^{-1}a$ (subtracting the mean value to make this well-defined), and simply notice that
$$
T_af
=\big[\Delta,T_{\Delta^{-1}a}\big] f-2T_{\nabla \Delta^{-1}a\cdot\nabla}f.
$$
The right-hand-side involves only regular symbols. By the formula of compositions we just proved, noting that $\Delta^{-1}a\in\mathcal{A}^{m+2}_{2-r}$, we find
$$
\big[\Delta,T_{\Delta^{-1}a}\big]
-2T_{\nabla \Delta^{-1}a\cdot\nabla}
$$
is in fact a para-differential operator of order $m+2-(2-r)=m+r$. 

\subsection{Analytic Function of Symbols and Para-Differential Operators}
In this subsection, we claim several propositions concerning analytic function of para-differential operators, to whose proof we apply the results obtained so far. They will be constantly needed for symbolic calculus in future applications.

Although a calculus is already available for our para-differential operators, uncertainty still appears when dealing with commutators. More precisely, if $a,b$ are only general rough symbols on $\Gp$ with order $m$ and $m'$ (in the sense of definition \ref{2Order}), then nothing is known for $[a,b]$ besides that it is a symbol of order $m+m'-1$. Such issue is of course trivially resolved in the commutative group case, but for a general compact Lie group, it would be feasible to look for a smaller subclass of symbols which possesses better commutator properties.

We start with function of symbols. The difficulty of non-trivial commutator can be partially resolved for a special class of symbols large enough for our application. Roughly speaking, such class consists of symbols of ``homogenized" classical differential operators. The prototype of such symbols is
$$
\sqrt{|\xi|^2+b(x,\xi)},
\quad 
b \text{ is the symbol of a vector field},
$$
i.e. the square-root of perturbed Laplacian. In the commutative group case, symbols like this may be manipulated as usual scalar-valued functions. We will show that such manipulation remains partially valid on our group $\Gp$.

\begin{definition}\label{QuasiHomoSym}
A quasi-homogeneous symbol of order $m$ on $\Gp$ takes the form
$$
\kappa(\xi)^mf\left(\frac{b(x,\xi)}{\kappa(\xi)}\right).
$$
Here the scalar Fourier multiplier $\kappa:\DuGp\to(0,+\infty)$ is the symbol of $|\nabla|=\sqrt{-\Delta}$, and $b(x,\xi)$ is the symbol of a vector field on $\Gp$, and the Borel function $f:\mathbb{C}\to\mathbb{C}$ is bounded. The action $f$ on each $\mathrm{End}(\Hh[\xi])$ is defined by spectral calculus of matrices.
\end{definition}

We first need some properties concerning the difference operator acting on the scalar symbol $\kappa(\xi)$. For simplicity, we fix a basis $\{X_i\}_{i=1}^n$ of $\mathfrak{g}$, so that $b(x,\xi)=\sum_{i=1}^nb_i(x)\sigma[X_i](\xi)$; and write $\{\Df_\mu\}$ for the RT-difference operators corresponding to the fundamental representations of $\Gp$. Throughout the rest of this subsection, summation with respect to Greek indices is interpreted as summation over these difference operators.
\begin{proposition}\label{Dkappa}
Let $\kappa$ be the symbol of $|\nabla|$ on $\Gp$.

(1) There necessarily holds $\kappa\in\mathscr{S}^1_{1,0}$, and $\Df_\mu \kappa$ commutes with every symbol of order $m$ up to an error of order $m-1$. Here the notion of order is as in Definition \ref{2Order}.

(2) For every real number $m\in\mathbb{R}$, the symbol $\Df_\mu\kappa^m-m\kappa^{m-1}\Df_\mu\kappa\in\mathscr{S}^{m-2}_{1,0}$.

(2) Furthermore, for any natural number $n$, there holds
$$
\left\|\Df_\mu\big(\kappa(\xi)^{-n}\big)-\frac{n\Df_\mu\kappa(\xi)}{\kappa(\xi)^{n+1}}\right\|
\lesssim \frac{n^2\big(1+C|\xi|^{-1}\big)^{n+1}}{|\xi|^{n+2}}.
$$
with the constants independent from $n,\xi$. Here as usual the norms are operator norms on $\Hh[\xi]$.
\end{proposition}
\begin{proof}
(1) The claim $\kappa\in\mathscr{S}^1_{1,0}$ follows from two premises: $|\nabla|$ is of class $\Psi^1_{1,0}(\Gp)$, the H\"{o}rmander class; the operator class $\Psi^m_{1,0}(\Gp)$ coincides with $\Op\mathscr{S}^m_{1,0}(\Gp)$ by Theorem \ref{Fischer}. We next notice that $\kappa$ commutes with any symbol $a$; thus by the Leibniz property,
$$
0=a\Df_\mu\kappa-\Df_\mu\kappa a
+\sum_{\nu,\nu'} c_{\mu\nu\nu'}\left(\Df_\nu a\Df_{\nu'}\kappa-\Df_\nu \kappa\Df_{\nu'}a\right).
$$
If $a$ is of order $m$, this obviously implies that $a\Df_\mu\kappa-\Df_\mu\kappa a$ is of order $m-1$.

(2) We shall extensively utilize Theorem \ref{Fischer}, namely $\Op\mathscr{S}^m_{1,0}=\Psi^m_{1,0}(\Gp)$. Since the Greek letter $\xi$ has already occupied the notation for representations, we use $(x,\omega)$ to denote elements in the cotangent bundle $\mathrm{T}^*\Gp$, instead of the commonly used $(x,\xi)$ in most literature. 

Consider any smooth function $b$ on $\Gp$, and then the commutator $\big[|\nabla|^m,b\big]\in\Psi^{m-1}_{1,0}$. By the global symbolic calculus formula, i.e. Theorem \ref{RegCompo}, we find that the invariant symbol of $\big[|\nabla|^m,b\big]$ is
$$
\sum_\mu \Df_\mu\kappa(\xi)^mX_\mu b(x)
\quad \mod\mathscr{S}^{m-2}_{1,0}.
$$
Here $\{X_\mu\}$ denotes the set of left-invariant vector fields adapting to the RT-difference operators $\Df_\mu$. In the special case $m=2$, this is just
$$
2\sum_\mu \kappa(\xi)\Df_\mu\kappa(\xi)X_\mu b(x)
\quad \mod\mathscr{S}^{0}_{1,0}.
$$
Since we have the precise equality $[\Delta,b]=2\nabla b\cdot\nabla+\Delta b$, we obtain
$$
\Op\left(\sum_\mu \kappa(\xi)\Df_\mu\kappa(\xi)X_\mu b(x)\right)
=-\nabla b\cdot\nabla
\quad \mod\Psi^{0}_{1,0}.
$$
By applying $m|\nabla|^{m-2}$ on both sides, we find, by the formula of global symbolic calculus once again,
$$
\begin{aligned}
\Op\left(\sum_\mu m\kappa(\xi)^{m-1}\Df_\mu\kappa(\xi)X_\mu b(x)\right)
&=m|\nabla|^{m-2}\big(\nabla b\cdot\nabla\big)
&\mod\Psi^{m-2}_{1,0}\\
&=m\nabla b\cdot\nabla|\nabla|^{m-2} 
&\mod\Psi^{m-2}_{1,0}.
\end{aligned}
$$

But in the formalism of H\"{o}rmander calculus, the commutator $\big[|\nabla|^m,b\big]$ has \emph{principal symbol} just being the Poisson bracket between $|\omega|_g^m$ and $b(x)$, namely
$$
\sum_{j=1}^n\partial_{\omega_j}|\omega|_{g}^m\partial_{x^j}b
=\sum_{j=1}^nm|\omega|_{g}^{m-2}{\omega_j}\partial_{x^j}b,
$$
where the local trivialization of the cotangent bundle is arbitrary, and $g$ is the bi-invariant metric on $\Gp$. Noting that the principal symbol of $\nabla b\cdot\nabla$ is just $\sum_{j=1}^n\omega_j\partial_{x^j}b$, this implies
$$
\big[|\nabla|^m,b\big]
=m\nabla b\cdot\nabla|\nabla|^{m-2}
\quad\mod\Psi^{m-2}_{1,0}.
$$
In conclusion, we find
$$
\Op\left(\sum_\mu m\kappa(\xi)^{m-1}\Df_\mu\kappa(\xi)X_\mu b(x)\right)
=\big[|\nabla|^m,b\big]
\quad\mod\Psi^{m-2}_{1,0},
$$
or in other words,
$$
\sum_\mu m\kappa(\xi)^{m-1}\Df_\mu\kappa(\xi)X_\mu b(x)
=\sum_\mu \Df_\mu\kappa(\xi)^{m}X_\mu b(x)
\quad\mod\mathscr{S}^{m-2}_{1,0}.
$$
Since $b$ is arbitrary, the equality $\Df_\mu\kappa^m=m\kappa^{m-1}\Df_\mu\kappa\mod\mathscr{S}^{m-2}_{1,0}$ must hold.

(3) Compared to (2), the proof is quite elementary. We use induction on $n$. For $n=1$ the claim is proved by noticing that $|\nabla|^{-1}=\Op(\kappa)$ is a pseudo-differential operator of order $-1$, and the formula
$$
0=\Df_\mu\big(\kappa^{-1}\cdot\kappa\big)=\Df_\mu\kappa^{-1}\kappa+\kappa^{-1}\Df_\mu\kappa
+\sum_{\nu,\nu'} c_{\mu\nu\nu'}\Df_\nu\kappa^{-1}\Df_{\nu'}\kappa.
$$
For convenience we omit the dependence on $\xi$ for the quantities we consider below. Setting $K_n=\sum_{\mu}\|\Df_\mu\kappa^{-n}\|$, we obtain by the Leibniz property again
$$
K_{n+1}\leq K_n|\xi|^{-1}+C_1|\xi|^{-(n+2)}+C_2K_n|\xi|^{-2}.
$$
Inductively, this implies
$$
K_{n}\leq C_1(n+1)|\xi|^{-(n+1)}\big(1+C_2|\xi|^{-1}\big)^{n+1},
$$
Using the Leibniz property again, we find
$$
\Df_\mu\kappa^{-n}=\kappa^{-(n-1)}\Df_\mu\kappa^{-1}+\kappa^{-1}\Df_\mu\kappa^{-(n-1)}
+\sum_{\nu,\nu'} c_{\mu\nu\nu'}\Df_\nu\kappa^{-(n-1)}\Df_{\nu'}\kappa^{-1}.
$$
Inductively this gives (noting that by our choice, $c_{\mu\nu\nu'}$ is either 0 or 1)
$$
\begin{aligned}
\left\|\Df_\mu\kappa^{-n}-\frac{n\Df_\mu\kappa^{-1}}{\kappa^{n+1}}\right\|
&\lesssim \sum_{k=0}|\xi|^{k-2}K_{n-1-k}(\xi)\\
&\lesssim \frac{n^2\big(1+C|\xi|^{-1}\big)^{n+1}}{|\xi|^{n+2}}.
\end{aligned}
$$
Finally it suffices to notice that $\Df_\mu\kappa^{-1}=-\kappa^{-2}\Df_\mu\kappa$ plus a symbol of order $-3$.
\end{proof}

We then verify that the notion of order for quasi-homogeneous symbols in Definition \ref{QuasiHomoSym} coincides with Definition \ref{2Order}. Furthermore, the difference of a quasi-homogeneous symbol can be computed by an approximate Leibniz rule.

\begin{proposition}\label{PSSymbol}
Let $\kappa$ be the symbol of $|\nabla|$ on $\Gp$, $b$ be the symbol of a vector field on $\Gp$. Let $f$ be a bounded holomorphic function defined near $z=0$ on the complex plane, covering the closed disk of radius $R_0:=\sup_{x,\xi}|\xi|^{-1}\|b(x,\xi)\|$.

(1) The quasi-homogeneous symbol $f(b/\kappa)$ is of order 0 in the sense of Definition \ref{2Order}, and the difference $\Df_\mu f(b/\kappa)$ is such that
$$
\Df_\mu f\left(\frac{b(x,\xi)}{\kappa(\xi)}\right)
-\frac{1}{\kappa(\xi)}f'\left(\frac{b(x,\xi)}{\kappa(\xi)}\right)\Df_\mu b
-\frac{\Df_\mu\kappa(\xi)}{\kappa(\xi)^2}f'\left(\frac{b(x,\xi)}{\kappa(\xi)}\right)b(x,\xi)
$$
is a symbol of order $-2$.

(2) If $a$ is the symbol of another vector field, then the commutator $[f(b/\kappa),a]$ is still a symbol of order $0$.

(3) If $b$ has at least $C^1$ coefficients, then for any left-invariant vector field $X$, the symbol $Xf(b/\kappa)$ is of order $0$, and 
$$
Xf\left(\frac{b}{\kappa}\right)-\frac{1}{\kappa}f'\left(\frac{b}{\kappa}\right)Xb
$$
is a symbol of order $-1$.
\end{proposition}

\begin{corollary}\label{CoroPSSymbol}
Under the same assumptions of Proposition \ref{PSSymbol}, if in addition the vector field $\Op(b)$ is of class $C^r_*$ for $r>0$, then the symbol $\kappa^m f(b/\kappa)$ is of class $\mathcal{A}^m_r$.
\end{corollary}

Thus the symbolic calculus formulas are easier to manipulate with for para-differential operators corresponding to quasi-homogeneous symbols.

\begin{proof}[Proof of Proposition \ref{PSSymbol}]
The proof is basically a repeated application of majorant power series. There is a simple fact that we will keep using: if $b$ is the symbol of a vector field, then $\Df_\mu b$ is just a scalar function of $x$, hence commutes with any symbol. In fact, if $X$ is a left-invariant vector field, then $\Df_\mu \sigma[X]=(XQ_\mu)(e)$ by formula (\ref{Diff_qa}), which is a number. In the following, it obviously suffices to prove the estimates for representations $\xi$ which are sufficiently ``large" in the dominance order (i.e. $|\xi|$ sufficiently large).

(1) Suppose 
$$
f(z)=\sum_{n=0}^\infty c_nz^n,
$$
and set $R$ to be the radius of convergence of this power series. By assumption, $R_0=\sup_{x,\xi}|\xi|^{-1}\|b(x,\xi)\|<R$. 

It suffices to prove that the expression in the statement has operator norm bounded by $|\xi|^{-2}$, because it is in fact still a power series in $b/\kappa$ multiplied by $\{\Df_\nu b\}_\nu$ and $\Df_\mu\kappa$. For convenience we omit the dependence on $\xi$ below.

We first estimate the norm $B_n:=\sum_{\mu}\big\|\Df_\mu(b^n)\big\|$. By the Leibniz property we find
$$
B_n
\leq C_1(R_0|\xi|)^{n-1}+(R_0|\xi|+C_2)B_{n-1}
,
$$
so some elementary manipulation with geometric progression gives 
$$
B_n\lesssim n|\xi|^{-1}(R_0|\xi|+C_2)^{n}.
$$
The implicit constant does not depend on $\xi$. Similarly as in Proposition \ref{Dkappa}, we may use this estimate inductively to $\Df_\mu(b^n)$ to conclude
$$
\begin{aligned}
\big\|\Df_\mu(b^n)-nb^{n-1}\Df_\mu b\big\|
\lesssim n^2|\xi|^{-2}(R_0|\xi|+C_2)^{n}
\end{aligned}
$$
Combining this and Proppsition \ref{Dkappa}, we find that the operator norm of the quantity in the statement has a majorant series
$$
\frac{1}{|\xi|^2}\sum_{n=1}^\infty n^2|c_n|\frac{(R_0|\xi|+C_2)^{n}+(R_0|\xi|+C)^n}{|\xi|^n}.
$$
For $|\xi|$ sufficiently large, Cauchy estimate for derivatives of a holomorphic function shows that the series is absolutely convergent and sums to a bounded function. This proves the claim.

(2) It suffices to notice that $[a,b]=ab-ba$ is still the symbol of a vector field. Furthermore, using the Leibniz property of Lie brackets, we have
$$
[a,b^n]=[a,b]b^{n-1}+b[a,b^{n-1}],
$$
implying 
$$
\big\|[a,b^n]\big\|\lesssim n(R_0|\xi|)^{n}.
$$
The rest of the estimate is identical to that in (1).

(3) The proof is of the same spirit of (2): in fact,
$$
X(b^n)=nb^{n-1}Xb+\sum_{k=1}^{n-1}b^{n-k-1}[Xb,b^k].
$$
Note that $[Xb,b^k]$ is still the symbol of a classical differential operator, whose norm at $\xi$ is controlled by
$$
\big\|[Xb,b^k]\big\|\lesssim_{\Gp} |b|_{C^1}k(R_0|\xi|)^{k},
$$
as in (2). Here $|b|_{C^1}$ stands for the $C^1$ norm of the coefficients. Thus
$$
\big\|X(b^n)-nb^{n-1}Xb\big\|\lesssim n^2(R_0|\xi|)^{n-1}|b|_{C^1}
$$
and the majorant series argument is sufficient to imply the desired result.
\end{proof}

We conclude this subsection with a proposition regarding symbols that involve very rough coefficients.
\begin{proposition}\label{SymbolVeryRough}
Suppose $\sum_{n=0}^\infty c_nz^n$ is a convergent power series near $z=0$. Fix $r>0$. Let $s>n/2$, $s>r$. Let $a,b$ be symbols of vector fields on $\Gp$, with $a$ of merely $C^{-r}_*$ regularity in $x$ and $b$ of $H^s$ regularity in $x$, with
$$
\sup_{\xi\in\DuGp}\frac{\|b(x,\xi)\|_{H^s_x}}{|\xi|}
$$
suitably small. Define 
$$
[b]_na:=ab^n+bab^{n-1}+\cdots+b^na.
$$
Then 
$$
\sum_{n=0}^\infty c_n\kappa^{-n}[b]_na
$$
is a symbol of class $\mathcal{A}^{0}_{-r}(\Gp)$.
\end{proposition}
\begin{proof}
Suppose $c$ is any $H^s$ vector field on $\Gp$. Let us recall the para-product decomposition Theorem \ref{ParaPrd} to analyze the product $ac$. With a little abuse of notation, we write $ac$ to represent any coefficient appearing in the product symbol $ac$. Then
$$
ac=T_ac+T_ca+R(a,c).
$$
By Proposition \ref{T_aNegIndex}, $T_a$ is a para-differential operator of order $r$, so 
$$
|T_ac|_{C^{-r}_*}
\leq 
C\|T_ac\|_{H^{s-r}}
\leq C|a|_{C^{-r}_*}\|c\|_{H^{s}}.
$$
Since $T_c$ is a para-differential operator of order 0, $|T_ca|_{C^{-r}_*}\leq C|a|_{C^{-r}_*}\|c\|_{H^{s}}$. Finally, from the proof of Theorem \ref{ParaPrd}, the symbol $R[a]\in\mathscr{S}^{r}_{1,1}$, so by Stein's theorem,
$$
|R(a,c)|_{C^{-r}_*}
\leq 
C\|R(a,c)\|_{H^{s-r}}
\leq C|a|_{C^{-r}_*}\|c\|_{H^{s}}.
$$
To summarize, $|ac|_{C^{-r}_*}\leq C|a|_{C^{-r}_*}\|c\|_{H^{s}}$.

We now use this to analyze the product $b^kab^{n-k}$. Since $H^s$ is a Banach algebra, we have
$$
|ab^{n-k}|_{C^{-r}_*}
\leq C|a|_{C^{-r}_*}\|b^{n-k}\|_{H^{s}}
\leq C_s^{n-k}|a|_{C^{-r}_*}\|b\|_{H^{s}}^{n-k},
$$
which further implies
$$
|b^kab^{n-k}|_{C^{-r}_*}
\leq C_s^{n}|a|_{C^{-r}_*}\|b\|_{H^{s}}^{n}.
$$
Thus if $\sup_{\xi\in\DuGp}|\xi|^{-1}\|b(x,\xi)\|_{H^s_x}$ is sufficiently small, then the series of matrices
$$
\sum_{n=0}^\infty c_n\kappa^{-n}[b]_na
$$
will have an absolutely convergent majorant series with respect to the $C^{-r}_*$ norm. This concludes the proof.
\end{proof}

\section{Global Symbolic Calculus on \texorpdfstring{$\SU(2)$}{a} and \texorpdfstring{$\mathbb{S}^2$}{b}}\label{5}
We apply the general theory developed so far to construct para-differetial calculus on $\mathbb{S}^2$, the standard 2-sphere. Although $\mathbb{S}^2$ cannot be endowed with a Lie group structure, it still inherits a symbolic calculus as $\SU(2)$-homogeneous space in a simple manner. The para-differential calculus on $\mathbb{S}^2$ will be constructed in this way. We start by formulating a global symbolic calculus on $\SU(2)$, and then pass to $\mathbb{S}^2$.

\subsection{Representation Theory of \texorpdfstring{$\SU(2)$}{c}}
Let us first sketch the representation theory of $\SU(2)$. An element of $\SU(2)$ will be denoted as
$$
x=\left(\begin{matrix}
x_1 & x_2 \\
-\bar{x}_2 & \bar{x}_1
\end{matrix}\right),
\quad
|x_1|^2+|x_2|^2=1.
$$
The Lie algebra $\mathfrak{su}(2)$, being the set of all skew-Hermitian matrices with trace 0, can be identified with the tangent space at $\mathrm{Id}\in\SU(2)$ and also the Lie algebra of left-invariant vector fields. It is spanned by the real-linear combination of (scaled) Pauli matrices
$$
\frac{1}{2}\left(\begin{matrix}
0 & i \\
i & 0
\end{matrix}\right),
\quad
\frac{1}{2}\left(\begin{matrix}
0 & -1 \\
1 & 0
\end{matrix}\right),
\quad
\frac{1}{2}\left(\begin{matrix}
i & 0 \\
0 & -i
\end{matrix}\right).
$$
They correspond to the 1-dimensional subgroups
$$
\omega_1(t)=\left(\begin{matrix}
\cos\frac{t}{2} & i\sin\frac{t}{2} \\
i\sin\frac{t}{2} & \cos\frac{t}{2}
\end{matrix}\right),
\quad
\omega_2(t)=\left(\begin{matrix}
\cos\frac{t}{2} & -\sin\frac{t}{2} \\
\sin\frac{t}{2} & \cos\frac{t}{2}
\end{matrix}\right),
\quad
\omega_3(t)=\left(\begin{matrix}
e^{it/2} & 0 \\
0 & e^{-it/2}
\end{matrix}\right)
$$
respectively. We denote these subgroups as $\mathbf{T}_{1,2,3}$. 

We employ the Euler angles as parameterization of $\SU(2)$. For 
$$
\varphi\in[0,2\pi),\quad\theta\in[0,\pi],\quad\psi\in[-2\pi,2\pi),
$$
an element of $\SU(2)$ outside a lower dimensional closed subset is factorized as
\begin{equation}\label{EulerAngle}
\left(\begin{matrix}
e^{i(\varphi+\psi)/2}\cos\frac{\theta}{2} & ie^{i(\varphi-\psi)/2}\sin\frac{\theta}{2} \\
ie^{-i(\varphi-\psi)/2}\sin\frac{\theta}{2} & e^{-i(\varphi+\psi)/2}\cos\frac{\theta}{2}
\end{matrix}\right)
=\left(\begin{matrix}
e^{i\varphi/2} & 0 \\
0 & e^{-i\varphi/2}
\end{matrix}\right)
\left(\begin{matrix}
\cos\frac{\theta}{2} & i\sin\frac{\theta}{2} \\
i\sin\frac{\theta}{2} & \cos\frac{\theta}{2}
\end{matrix}\right)
\left(\begin{matrix}
e^{i\psi/2} & 0 \\
0 & e^{-i\psi/2}
\end{matrix}\right).
\end{equation}
The parameters $\varphi\in[0,2\pi),\,\theta\in(0,\pi),\,\psi\in[-2\pi,2\pi)$ are in 1-1 correspondence with elements of $\SU(2)$ excluding a lower dimensional closed subset (which corresponds to $\theta=\pm\pi$). We shall denote the element with Euler angles $\varphi,\theta,\psi$ as $\Omega(\varphi,\theta,\psi)$. The factorization (\ref{EulerAngle}) can be re-written as
\begin{equation}\label{EulerAngle1}
\Omega(\varphi,\theta,\psi)=\omega_3(\varphi)\omega_1(\theta)\omega_3(\psi)\in\mathbf{T}_3\cdot\mathbf{T}_2\cdot\mathbf{T}_3.
\end{equation}

The Killing form of $\mathfrak{su}(2)$ is $4\Tr(XY)$ for $X,Y\in\mathfrak{su}(2)$. If we set the negative of this Killing form as the inner product on $\mathfrak{su}(2)$, then the matrices $\omega_{1,2,3}'(0)$ form an orthonormal basis. Thus, the Riemann metric on $\SU(2)$ under the Euler angles is
$$
G_0=d\theta^2+d\varphi^2+2\cos\theta d\varphi d\psi+d\psi^2,
$$
or in matrix form
$$
\left(\begin{matrix}
(G_0)_{\theta\theta} & (G_0)_{\theta\varphi} & (G_0)_{\theta\psi} \\
* & (G_0)_{\varphi\varphi} & (G_0)_{\varphi\psi} \\
* & * & (G_0)_{\psi\psi}
\end{matrix}\right)
=\left(\begin{matrix}
1 & 0 & 0 \\
0 & 1 & \cos\theta \\
0 & \cos\theta & 1
\end{matrix}\right).
$$
If we identify $\SU(2)$ with the unit 3-sphere in $\mathbb{C}^2$, then $G_0$ is 4 times the inherited metric on unit 3-sphere. The Laplacian of $G_0$ is
$$
\Delta_{G_0}
=\frac{\partial^2}{\partial\theta^2}+\frac{\cos\theta}{\sin\theta}\frac{\partial}{\partial\theta}
+\frac{1}{\sin^2\theta}\frac{\partial^2}{\partial\varphi^2}
-\frac{2\cos\theta}{\sin^2\theta}\frac{\partial^2}{\partial\varphi\partial\psi}+\frac{1}{\sin^2\theta}\frac{\partial^2}{\partial\psi^2}.
$$

In Chapter 3 of \cite{Vilenkin1978}, Vilenkin explicitly computed the left-invariant vector fields on $\SU(2)$ corresponding to the subgroups $\omega_{1,2,3}(t)$. We denote them by $X_{1,2,3}$ respectively. The action of $X_j$ is 
$$
(X_jf)(x):=\frac{d}{dt}f(x\omega_j(t))\Big|_{t=0}.
$$
They are given by, respectively,
\begin{equation}\label{LISU(2)}
\begin{aligned}
X_1&=\cos\psi\frac{\partial}{\partial\theta}+\frac{\sin\psi}{\sin\theta}\frac{\partial}{\partial\varphi}-\frac{\cos\theta}{\sin\theta}\sin\psi\frac{\partial}{\partial\psi},\\
X_2&=-\sin\psi\frac{\partial}{\partial\theta}+\frac{\cos\psi}{\sin\theta}\frac{\partial}{\partial\varphi}-\frac{\cos\theta}{\sin\theta}\cos\psi\frac{\partial}{\partial\psi},\\
X_3&=\frac{\partial}{\partial\psi}.
\end{aligned}
\end{equation}
It is directly verified that these vector fields form an orthonormal basis in the tangent space of any element in $\SU(2)$, and $\Delta_{G_0}=X_1^2+X_2^2+X_3^2$. However, for symbolic calculus, it is more convenient to consider certain linear combination of the $X_j$'s since it is more convenient for the symbols to contain more 0's. We thus introduce the \emph{creation, annihilation and neutral} operators
\begin{equation}\label{CANSU(2)}
\begin{aligned}
\Pa[+]&=iX_1-X_2=e^{-i\psi}\left(i\frac{\partial}{\partial\theta}-\frac{1}{\sin\theta}\frac{\partial}{\partial\varphi}+\frac{\cos\theta}{\sin\theta}\frac{\partial}{\partial\psi}\right),\\
\Pa[-]&=iX_1+X_2=e^{i\psi}\left(i\frac{\partial}{\partial\theta}+\frac{1}{\sin\theta}\frac{\partial}{\partial\varphi}-\frac{\cos\theta}{\sin\theta}\frac{\partial}{\partial\psi}\right),\\
\Pa[0]&=iX_3=i\frac{\partial}{\partial\psi},
\end{aligned}
\end{equation}
just as in \cite{Vilenkin1978} or \cite{RT2013}. Thus 
$$
\Delta_{G_0}=-\frac{1}{2}(\Pa[-]\Pa[+]+\Pa[+]\Pa[-])-\Pa[0]^2,
$$
and the commutator relations for $\mathfrak{su}(2)$ hold:
$$
[\Pa[0],\Pa[+]]=\Pa[+],
\quad
[\Pa[-],\Pa[0]]=\Pa[-],
\quad
[\Pa[+],\Pa[-]]=2\Pa[0].
$$

The table of irreducible unitary representations of $\SU(2)$ can be found in any standard textbook on representation theory. We sketch the result as follows. Since the group $\SU(2)$ has rank 1, one can label irreducible unitary representations of $\SU(2)$ by a parameter $l\in\hN$, the set of half integers. The irreducible unitary representation with label $l$ has dimension $2l+1$. The $l$'th representation is realized as follows: setting $\Hh[l]$ to be the space of complex homogeneous polynomials in a two-dimensional variable $z\in\mathbb{C}^2$ with degree $2l$, the action of $\SU(2)$ on $\Hh[l]$ is realized as
$$
\big(T^l(x)f\big)(z):=f(xz).
$$
The Hermite structure on $\Hh[l]$ is defined by fixing an orthonormal basis
$$
\frac{z_1^{l-k}z_2^{l+k}}{\sqrt{(l-k)!(l+k)!}},
\quad
k=-l,\,-l+1,\,\cdots,\,l-1,\,l.
$$
Following \cite{Vilenkin1978}, we use $T^l=(T^l_{nm})$, $m,n=-l,\,-l+1,\,\cdots,\,l-1,\,l$, to denote the matrix elements of the representation labeled by $l$ under the basis fixed above. We fix the convention that $n$ labels the rows of $T^l$ and $m$ labels the columns of $T^l$. In Euler angles,
$$
T^l_{nm}\big(\Omega(\varphi,\theta,\psi)\big)=e^{-i(n\varphi+m\psi)}P^l_{nm}(\cos\theta),
$$
where 
$$
P^l_{nm}(z)=\frac{2^{-l}(-1)^{l-m}i^{m-n}}{\sqrt{(l-m)!(l+m)!}}\sqrt{\frac{(l+n)!}{(l-n)!}}
\cdot\frac{(1-z)^{(m-n)/2}}{(1+z)^{(m+n)/2}}\left(\frac{d}{dz}\right)^{l-n}\left((1-z)^{l-m}(1+z)^{l+m}\right).
$$
For example, $l=0$ gives the trivial representation $\SU(2)\to\{1\}$; for $l=1/2$, we obtain the fundamental representation, which is also a faithful: $T^{1/2}(\Omega(\varphi,\theta,\psi))=\Omega(\varphi,\theta,\psi)$, and the entries are
$$
\bordermatrix{%
 & m=-\frac{1}{2} & m=\frac{1}{2} \cr
n=-\frac{1}{2} & e^{i(\varphi+\psi)/2}\cos\frac{\theta}{2} & ie^{i(\varphi-\psi)/2}\sin\frac{\theta}{2} \cr
n=\frac{1}{2} & ie^{-i(\varphi-\psi)/2}\sin\frac{\theta}{2} & e^{-i(\varphi+\psi)/2}\cos\frac{\theta}{2}
};
$$
for $l=1$, the entries of $T^{1}(\Omega(\varphi,\theta,\psi))$ are
$$
\bordermatrix{%
 & m=-1 & m=0 & m=1 \cr
n=-1 & e^{i(\varphi+\psi)/2}\cos^2\frac{\theta}{2} & ie^{i\varphi}\frac{\sin\theta}{\sqrt{2}} & -e^{i(\varphi-\psi)/2}\sin^2\frac{\theta}{2} \cr
n=0 & ie^{i\psi}\frac{\sin\theta}{\sqrt{2}} & \cos\theta & ie^{-i\psi}\frac{\sin\theta}{\sqrt{2}}  \cr
n=1 & -e^{-i(\varphi-\psi)/2}\sin^2\frac{\theta}{2} & ie^{-i\varphi}\frac{\sin\theta}{\sqrt{2}} & e^{-i(\varphi+\psi)/2}\cos^2\frac{\theta}{2}
};
$$

From now on, we shall write $\Ft{f}(l)$ for the Fourier transform of $f\in\mathcal{D}'(\SU(2))$:
\begin{equation}\label{FtSU(2)}
\begin{aligned}
\Ft{f}(l)_{mn}
&=\int_{\SU(2)}f(x)\overline{T^l_{nm}}(x)dx\\
&=\frac{1}{16\pi^2}\int_0^{4\pi}\int_{0}^\pi\int_0^{2\pi} f\big(\Omega(\varphi,\theta,\psi)\big)
\overline{T^l_{nm}}\big(\Omega(\varphi,\theta,\psi)\big)\sin\theta \cdot d\varphi d\theta d\psi.
\end{aligned}
\end{equation}
The Peter-Weyl theorem \ref{PeterWeyl} now takes the form
$$
f(x)=\sum_{l\in\hN}(2l+1)\sum_{m,n=-l}^l \Ft{f}(l)_{mn}T^l_{nm}(x).
$$
Here the convention of summation for $m,n$ is that they exhaust all numbers $-l,\,-l+1,\,\cdots,\,l-1,\,l$. The functions $\{T^l_{nm}\}_{m,n=-l}^l$ span the eigenspace $\Hh[l]$ of $\Delta_{G_0}$ with eigenvalue $-l(l+1)$. 

Since $\SU(2)$ can be identified as the unit 3-sphere in $\mathbb{C}^2$, the functions $\{T^l_{nm}\}_{m,n=-l}^l$ are exactly the spherical harmonics on 3-sphere of degree $2l+1$. The spectral localization property, i.e. Corollary \ref{SpecPrd}, now becomes the well-known fact that the product of spherical harmonics of degree $p$ and $q$ is a linear combination of spherical harmonics of degree between $|p-q|$ and $p+q$. However, it should be emphasized that the Euler angle parameterization \emph{does not} coincide with the spherical coordinate parameterization of spherical harmonics.

\subsection{Symbolic Calculus on \texorpdfstring{$\SU(2)$}{d}}

We now construct pseudo-differential calculus on $\SU(2)$. For 
$$
x=\left(\begin{matrix}
x_1 & x_2 \\
-\bar{x}_2 & \bar{x}_1
\end{matrix}\right)\in \SU(2),
$$
the functions
$$
x_1-1,\quad \bar{x}_1-1,\quad x_2,\quad -\bar{x}_2 
$$
form a strongly RT-admissible tuple of $\SU(2)$ defined in \ref{RTAdm}, which is also the fundamental tuple of $\SU(2)$ defined in (\ref{QFund}). The only disadvantage is that this tuple consists of 4 elements and can bring unnecessary tedium for calculation. Following \cite{RT2013}, we introduce instead
\begin{equation}\label{SU(2)q}
\mathscr{Q}_+(x)=-\bar{x}_2,\quad \mathscr{Q}_-(x)=x_2,\quad \mathscr{Q}_0(x)=x_1-\bar{x}_1.
\end{equation}
The tuple $\mathscr{Q}=(\mathscr{Q}_+,\mathscr{Q}_-,\mathscr{Q}_0)$ is RT-admissible but not strongly RT-admissible, since $-\mathrm{Id}$ is the other common zero of them besides $\mathrm{Id}$. A direct computation gives 
$$
(\Pa[+]\mathscr{Q}_+)(\mathrm{Id})=(\Pa[-]\mathscr{Q}_-)(\mathrm{Id})=(\Pa[0]\mathscr{Q}_0)(\mathrm{Id})=-1.
$$
Taylor's formula, i.e. Proposition \ref{TaylorGp} is then valid for $q$ and the differential operators $\Pa[+],\Pa[-],\Pa[0]$. Note that the higher order left-invariant differential operators in Taylor's formula, denoted by $\partial_x^{(\alpha)}$ from now on, have to be defined inductively, since $\Pa[+],\Pa[-],\Pa[0]$ do not commute. 

Since the dual of $\SU(2)$ is labelled by $l\in\hN$, we write $a(x,l)\in\mathrm{End}(\Hh[l])$ for a symbol $a$ on $\SU(2)$. The symbolic calculus on $\SU(2)$ then starts with the symbol of composition
\begin{equation}\label{SU(2)Compo}
(a\#_{r;\mathscr{Q}} b)(x,l)=\sum_{|\alpha|\leq r}\Df^\alpha_{\mathscr{Q},l}a(x,l)\cdot\partial_x^{(\alpha)}b(x,l),
\end{equation}
and the symbol of adjoint
\begin{equation}\label{SU(2)Adj}
a^{\bullet;r,\mathscr{Q}}(x,l)=\sum_{|\alpha|\leq r}\Df^\alpha_{\mathscr{Q},l}\partial_x^{(\alpha)}a^*(x,l).
\end{equation}
We directly cite the following results from \cite{RT2009}:
\begin{theorem}\label{SymPa}
The symbols $\sigma_{+}$, $\sigma_{-}$ and $\sigma_0$ of $\Pa[+]$, $\Pa[-]$ and $\Pa[0]$ respectively are all Fourier multipliers:
$$
\begin{aligned}
\sigma_+(l)_{mn}&=-\sqrt{(l - n)(l + n + 1)}\delta_{m,n+1}\\
\sigma_-(l)_{mn}&=-\sqrt{(l + n)(l - n + 1)}\delta_{m,n-1}\\
\sigma_0(l)_{mn}&=n\delta_{mn}.
\end{aligned}
$$
\end{theorem}
\begin{theorem}\label{Df_q}
Suppose $a$ is a Fourier multiplier on $\SU(2)$, i.e. for each $l\in\hN$, the value $a(l)\in\mathrm{End}(\Hh[l])$. Then the RT difference operators $\Df_+$, $\Df_-$, $\Df_0$ corresponding to $\mathscr{Q}_+$, $\mathscr{Q}_-$, $\mathscr{Q}_0$ respectively are 
$$
(\Df_+a)(l)_{nm}
=\frac{\sqrt{(l + m)(l - n)}}{2l + 1}a(l^-)_{n^+m^-}-\frac{\sqrt{(l - m+1)(l + n+1)}}{2l + 1}a(l^+)_{n^+m^-},
$$
$$
(\Df_-a)(l)_{nm}
=\frac{\sqrt{(l - m)(l + n)}}{2l + 1}a(l^-)_{n^-m^+}-\frac{\sqrt{(l + m+1)(l - n+1)}}{2l + 1}a(l^+)_{n^-m^+},
$$
$$
\begin{aligned}
(\Df_0a)(l)_{nm}
&=\frac{\sqrt{(l - m)(l - n)}}{2l + 1}a(l^-)_{n^+m^+}+\frac{\sqrt{(l + m+1)(l + n+1)}}{2l + 1}a(l^+)_{n^+m^+}\\
&\quad -\frac{\sqrt{(l + m)(l + n)}}{2l + 1}a(l^-)_{n^-m^-}-\frac{\sqrt{(l - m+1)(l - n+1)}}{2l + 1}a(l^+)_{n^-m^-}
\end{aligned}
$$
Here $k^{\pm}=k\pm1/2$. In particular,
$$
\Df_\mu \sigma_\nu=\delta_{\mu\nu}\cdot\sigma_{\mathrm{Id}},\quad \mu,\nu=+,\,-,\,0.
$$
\end{theorem}
\begin{remark}\label{Para-Classical}
We are actually able to manipulate para-differential operators arising from classical differential operators on $\SU(2)$ fairly easily. For example, let $Y=a\Pa[+]$ and $Z=b\Pa[-]$ be vector fields with coefficients $a,b\in C^r$ with $r>1$. Then the commutator 
$$
[Y,Z]
=a\Pa[+]b\Pa[-]-b\Pa[-]a\Pa[+]
+2ab\Pa[0]
$$
On the other hand, the symbols of $Y,Z$ are respectively
$$
\sigma[Y]=a\sigma_+,
\quad
\sigma[Z]=b\sigma_-,
$$
so a direct computation using Theorem \ref{Df_q} implies a precise equality
$$
\begin{aligned}
\sigma[Y]\#_{r;\mathscr{Q}}\sigma[Z]-\sigma[Z]\#_{r;\mathscr{Q}}\sigma[Y]
&=ab[\sigma_+,\sigma_-]+a\Df_+\sigma_+\Pa[+]b\sigma_--b\Df_-\sigma_-\Pa[-]a\sigma_+\\
&=2ab\sigma_0+a\Pa[+]b\sigma_--b\Pa[-]a\sigma_+.
\end{aligned}
$$
Consequently, Theorem \ref{Compo2} ensures that $[T_Y,T_Z]$ is exactly the para-differential operator corresponding to $[Y,Z]$, modulo a smoothing operator in $\Op\Sigma_{<1/2}^{1-r}$. Consequently, with $\{Y,Z\}$ being the symbol of $[Y,Z]$,
$$
[T_Y,T_Z]=T_{\{Y,Z\}}+\Op\Sigma_{<1/2}^{1-r}.
$$
\end{remark}

\subsection{Passage to the 2-Sphere via Hopf Fiberation}
We finally describe how pseudo-differential calculus can be constructed on $\mathbb{S}^2$. Although $\mathbb{S}^2$ does not carry any Lie group structure, it is still a homogeneous space corresponding to $\SU(2)$, and the projection $\SU(2)\to \mathbb{S}^2$ is \emph{Hopf fiberation}, well-known in geometric topology. To be consistent with \emph{left-invariant differential operators} we have been employing so far, following \cite{RT2013}, we consider $\mathbb{S}^2$ as a \emph{right} $\SU(2)$-homogeneous space, that is, the set of \emph{right cosets} 
$$
\mathbf{T}_3\setminus\SU(2):=\Big\{\mathbf{T}_3\cdot x:\,x\in\SU(2)\Big\}
$$
endowed with the quotient topology, on which $\SU(2)$ \emph{acts from the right}. A point that one should keep in notice is that, in the Euler angle representation as in (\ref{EulerAngle1}), the angle $\varphi\in[0,2\pi)$ only exhausts half of the subgroup $\mathbf{T}_3$, so the orbit of a given element $\Omega(\varphi_0,\theta_0,\psi_0)\in\SU(2)$ is
$$
\Big\{\pm\Omega(\varphi,\theta_0,\psi_0):\,
\varphi\in[0,2\pi)\Big\}.
$$
We thus realize the Hopf map $\SU(2)\to \mathbb{S}^2$ as
$$
\left(\begin{matrix}
x_1 & x_2 \\
-\bar{x}_2 & \bar{x}_1
\end{matrix}\right)
\to
\left(\begin{matrix}
-2\mathrm{Im}(x_1\bar x_2) \\
2\mathrm{Re}(x_1\bar x_2) \\
|x_1|^2-| x_2|^2
\end{matrix}\right)\in\mathbb{R}^3,
\quad
\text{where}
\quad
|x_1|^2+|x_2|^2=1,
$$
or in Euler angles,
$$
\Omega(\varphi,\theta,\psi)=\left(\begin{matrix}
e^{i(\varphi+\psi)/2}\cos\frac{\theta}{2} & ie^{i(\varphi-\psi)/2}\sin\frac{\theta}{2} \\
ie^{-i(\varphi-\psi)/2}\sin\frac{\theta}{2} & e^{-i(\varphi+\psi)/2}\cos\frac{\theta}{2}
\end{matrix}\right)
\to
\left(\begin{matrix}
\sin\theta\cos\psi \\
\sin\theta\sin\psi \\
\cos\theta
\end{matrix}\right).
$$
We note that the Hopf fiberation is a Riemannian submersion if $\SU(2)$ is endowed with the metric $G_0$ and $\mathbb{S}^2$ is endowed with the standard metric $g_0$.

A function on $\mathbb{S}^2$ is lifted to $\SU(2)$ as a function independent to the Euler angle $\varphi$ and symmetric with respect to the reflection map $x\to -x$. Given a function $f$ (distribution) on $\mathbb{S}^2$, we denote by 
\begin{equation}\label{Fsharp}
f^\sharp\big(\Omega(\varphi,\theta,\psi)\big)
:=f\big(\sin\theta\cos\psi,\sin\theta\sin\psi,\cos\theta\big)
\end{equation}
its lift to the $\mathbf{T}_3$-invariant function (distribution) on $\SU(2)$. We refer smooth functions (distributions respectively) on $\SU(2)$ with such symmetry as being \emph{$\mathbf{T}_3$-invariant}, and denote the space of them by $C^\infty(\mathbf{T}_3\setminus\SU(2))$ ($\mathcal{D}'(\mathbf{T}_3\setminus\SU(2))$ respectively). We emphasize that sometimes it is necessary to distinguish between $C^\infty(\mathbf{T}_3\setminus\SU(2))$ and $C^\infty(\mathbb{S}^2)$. Further definitions can be formulated for operators and symbols:
\begin{definition}
A linear operator $A:C^\infty(\SU(2))\to \mathcal{D}'(\SU(2))$ is said to be $\mathbf{T}_3$-invariant, if it maps $C^\infty(\mathbf{T}_3\setminus\SU(2))$ to $\mathcal{D}'(\mathbf{T}_3\setminus\SU(2))$.
\end{definition}
\begin{definition}
A symbol $a(x,l)$ on $\SU(2)$ is said to be $\mathbf{T}_3$-invariant, if for any element $\omega_3(t)\in\mathbf{T}_3$, there holds $a(\omega_3(t)x,l)=a(x,l)$, for all $x\in\SU(2)$ and $l\in\hN$. In Euler angles, $\mathbf{T}_3$-invariant symbols are exactly those that do not depend on the Euler angle $\varphi$ and remain invariant under reflection $x\to-x$. 
\end{definition}
Obviously an operator is $\mathbf{T}_3$-invariant if and only if its symbol is $\mathbf{T}_3$-invariant. We can thus lift pseudo-differential operators on $\mathbb{S}^2$ to $\SU(2)$ via the realization of Hopf fiberation that we fixed above. There are necessarily infinitely many ways of lifting and there is no ``natural" one among them. However, for classical differential operators and pseudo-differential operators constructed out of the Laplacian $\Delta_{g_0}$, lifting to $\SU(2)$ as $\Delta_{G_0}$ is direct. The advantage of this approach is that operators on $\SU(2)$ admit an explicitly manipulable symbolic calculus. In the formulas for symbolic calculus, if $a,b$ are both $\mathbf{T}_3$-invariant, then the symbols
\begin{equation}\label{ParaCompoSU2}
(a\#_{r;\mathscr{Q}} b)(x,l)=\sum_{|\alpha|\leq r}\Df^\alpha_{\mathscr{Q},l}a(x,l)\cdot\partial_x^{(\alpha)}b(x,l),
\end{equation}
and
\begin{equation}\label{ParaAdjSU2}
a^{\bullet;r,\mathscr{Q}}(x,l)=\sum_{|\alpha|\leq r}\Df^\alpha_{\mathscr{Q},l}\partial_x^{(\alpha)}a^*(x,l).
\end{equation}
are still $\mathbf{T}_3$-invariant, since the differential operators $\partial_x^{(\alpha)}$ are left-invariant, hence specifically invariant under left translation via $\omega_3(\varphi)$. Briefly speaking, \emph{symbolic calculus for $\mathbf{T}_3$-invariant operators preserves $\mathbf{T}_3$-invariance}. Consequently, symbolic calculus on $\mathbb{S}^2$ can be constructed by lifting operators on $\mathbb{S}^2$ to $\mathbf{T}_3$-invariant operators on $\SU(2)$, and manipulated without concern for $\mathbf{T}_3$-invariance. For a $\mathbf{T}_3$-invariant rough symbol $a\in\mathcal{A}^m_r$, the corresponding para-differential operator $T_a:=\Op(a^\chi)$ is still $\mathbf{T}_3$-invariant, since spectral cut-off preserves $\mathbf{T}_3$-invariance.

It is more illustrative to write down the Fourier series development. For example, given a function $f\in C^\infty(\mathbb{S}^2)$, we denote its lift to $\SU(2)$ as $f^\sharp\in C^\infty(\mathbf{T}_3\setminus\SU(2))$. By (\ref{FtSU(2)}), The Fourier transform $\Ft{f^\sharp}(l)$ does not vanish only when $l\in\mathbb{N}_0$, and for $l\in\mathbb{N}_0$, the only possible non-vanishing entries are those with $m=0$, i.e. the column with index 0. In that case, we have
$$
\Ft{f^\sharp}(l)_{0n}=\int_{\mathbb{S}^2}f(x)Y^{-n}_l(x)d\mu_0,
$$
with $Y^n_l(x)$ being the standard complex spherical harmonic on $\mathbb{S}^2$. Furthermore, for a $\mathbf{T}_3$-invariant symbol $a(x,l)$, the action $\Op(a)f^\sharp$ reads
$$
\sum_{l\in\mathbb{N}_0}(2l+1)\Tr\big[a(x,l)\Ft{f^\sharp}(l)T^l(x)\big].
$$
Under Euler angle representation, the matrix $\Ft{f^\sharp}(l)T^l(x)$ consists of entries depending only on $\psi$ and $\theta$, so each summand, hence $\Op(a)f^\sharp$, is a function of Euler angles $\theta$ and $\psi$ only. The reflection invariance is checked similarly. As a result, $\Op(a)f^\sharp\in C^\infty(\mathbf{T}_3\setminus \SU(2))$, hence projects to a function defined on $\mathbb{S}^2$. 

\section{Para-linearization of Dirichlet-Neumann Operator for a Distorted Sphere}\label{6}
In this section, we prove the para-linearization formula for the Dirichlet-Neumann operator of a distorted sphere in $\mathbb{R}^3$. We are going to lift our discussion to $\SU(2)$, so that the toolbox of para-differential calculus is available. The idea is to ``factorize" the Laplacian and then reduce to the boundary.

\subsection{The Dirichlet-Neumann Operator}
Recall that we set $\iota_0:\mathbb{S}^2\hookrightarrow\mathbb{R}^3$ to be the standard embedding. We write $\rho$ for the Euclidean distance function to $0$ in $\mathbb{R}^3$, and define
$$
\mathcal{U}=\left\{\frac{1}{2}<\rho<\frac{3}{2}\right\},
$$
so that $\mathcal{U}$ is a tubular neighbourhood of the standard unit sphere, with smooth boundary. Obviously ${\mathcal{U}}$ is diffeomorphic to the product manifold $\mathbb{S}^2\times(-1/2,1/2)$, with smooth boundary $\mathbb{S}^2\times\{-1/2,1/2\}$. 

We shall use $x$ to mark points on $\mathbb{S}^2$ or $\SU(2)$ (there is usually no risk of confusion). We already used $g_0$ to denote the induced metric and $N_0$ to denote the outward pointing normal vector field of $\mathbb{S}^2$, also known as the Gauss map. Note that $N_0$ coincides with $\iota_0$.

Just as shown in Figure \ref{Height}, the distorted sphere $M_\zeta$ we shall consider will be given by the graph of a ``height function" $\zeta\in C^r(\mathbb{S}^2)$ with $|\zeta|_{L^\infty}<1/2$, i.e.
$$
M_\zeta:=\{(1+\zeta(x))N_0(x):x\in \mathbb{S}^2\}.
$$
We may thus abbreviate $\iota_\zeta=(1+\zeta) N_0$, and $M_0=\iota_0(\mathbb{S}^2)$. If $r\geq1$, $M_\zeta$ is $C^r$-diffeomorphic to $\mathbb{S}^2$ and is itself a $C^r$ differentiable hypersurface, but this is the highest regularity to expect for generic $\zeta\in C^r(\mathbb{S}^2)$. From now on, we shall fix $r>2$. 

The most convenient coordinate system that we shall use is the distorted normal coordinate with respect to $M_0$, that is a $C^r$ diffeomorphism mapping a point $(x,y)\in \mathbb{S}^2\times[-1/2,1/2]$ to 
$$
(1+\zeta(x)+y)N_0(x)\in\bar{\mathcal{U}}.
$$
Thus, under our choice of coordinate, we have 
$$
\rho(x,y)=1+\zeta(x)+y.
$$

We then pull everything on $\bar{\mathcal{U}}$ back to $\mathbb{S}^2\times[-1/2,1/2]$ through this specific diffeomorphism. As a starting point, the pulled-back Euclidean metric reads
\begin{equation}\label{g_E}
g_{\mathrm{E}}=\rho^2g_0(x)+(d\zeta+dy)\otimes (d\zeta+dy),
\quad 
\rho=1+\zeta+y.
\end{equation}
In deriving the expression for $g_{\mathrm{E}}$, we used the second fundamental form $\langle d\iota_0(x), dN_0(x)\rangle$, where the bracket denotes the pairing $\langle a_idx^i, c_jdx^j\rangle:=(a_i\cdot c_j)dx^i\otimes dx^j$ between $\mathbb{R}^{3}$-valued differential 1-forms on $\mathbb{S}^2$. Note that for the sphere, the second fundamental form coincides with $g_0$. The following $C^r$-Riemannian manifold (with boundary) is thus isometric to $\bar{\mathcal{U}}$:
$$
\bar{\mathcal{W}}:=\Big(\mathbb{S}^2\times\left[-\frac{1}{2},\frac{1}{2}\right],g_{\mathrm{E}}\Big).
$$
The induced metric on $M_\zeta$, being the hypersurface $y=0$ in $\bar{\mathcal{W}}$, also takes a simple form
$$
g_\zeta=(1+\zeta)^2g_0+d\zeta\otimes d\zeta.
$$
We use the following formulas for block matrices to compute the dual metric of $g_{\mathrm{E}}$: 
$$
\left(
\begin{matrix}
A+c c^\top & c \\
c^{\top} & 1
\end{matrix}
\right)^{-1}
=\left(
\begin{matrix}
A^{-1} & -A^{-1}c \\
-(A^{-1}c)^{\top} & 1+(A^{-1}c)\cdot c
\end{matrix}
\right),
\quad
\det\left(
\begin{matrix}
A+c c^\top & c \\
c^{\top} & 1
\end{matrix}
\right)=\det A
$$
for an $2\times 2$ positive definite symmetric matrix $A$  and $c\in\mathbb{R}^2$. We find the dual metric of $g_{\mathrm{E}}$ reads
\begin{equation}\label{g_E^{-1}}
\begin{aligned}
g_{\mathrm{E}}^{-1}
&=\frac{g_0^{-1}(x)}{\rho^{2}}
+\left(1+\frac{|\nabla_0\zeta|_{g_0}^2}{\rho^{2}}\right)\partial_y\otimes\partial_y
-\frac{\nabla_0\zeta\otimes\partial_y+\partial_y\otimes\nabla_0\zeta}{\rho^{2}}.
\end{aligned}
\end{equation}
Here $g_0^{-1}$ stands for the dual metric of $g_0$.

For a function $\Phi$ defined on $\bar{\mathcal{U}}$, we pull it back to $\bar{\mathcal{W}}$ as
\begin{equation}\label{Psi}
\Psi(x,y):=\Phi\big((1+\zeta(x)+y)N_0(x)\big).
\end{equation}
Regarding $\Psi$ as a function defined on $\mathbb{S}^2$ depending on parameter $y\in[-1/2,0]$, the gradient of $\Psi$ is computed as
$$
\nabla_\mathrm{E}\Psi
=\frac{\nabla_0\Psi-\partial_y\Psi\nabla_0\zeta}{\rho^{2}}
+\left[\left(1+\frac{|\nabla_0\zeta|_{g_0}^2}{\rho^{2}}\right)\partial_y\Psi-\frac{\nabla_0\zeta\cdot\nabla_0\Psi}{\rho^{2}}\right]\partial_y.
$$
The hypersurface $y=0$ corresponds to $M_\zeta$ in $\mathcal{U}$, so the vector field
$$
\nabla_{\mathrm{E}}y=-\frac{\nabla_0\zeta}{\rho^{2}}
+\left(1+\frac{|\nabla_0\zeta|_{g_0}^2}{\rho^{2}}\right)\partial_y\Bigg|_{y=0}
$$
along $\{y=0\}$ is perpendicular to the hypersurface. Similarly, the hypersurface $y+\zeta(x)=0$ corresponds to the undisturbed sphere in $\mathcal{U}$, so 
$$
N_0(x)=\partial_y|_{y=-\zeta(x)}.
$$
Consequently, using the expression (\ref{g_E}) of $g_{\mathrm{E}}$, the Dirichlet-Neumann operator ${D[\zeta]\phi}$ in (\ref{EQ}) is computed as
\begin{equation}\label{Dzetaphi}
\begin{aligned}
D[\zeta]\phi
&=\frac{\nabla_{\mathrm{E}}\Psi\cdot\nabla_{\mathrm{E}}y}{N_0\cdot\nabla_{\mathrm{E}}y}\Bigg|_{y=0}\\
&=\left(1+\frac{|\nabla_0\zeta|_{g_0}^2}{\rho^{2}}\right)\partial_y\Psi\Bigg|_{y=0}-\frac{\nabla_0\zeta\cdot\nabla_0\Psi}{\rho^{2}}\Bigg|_{y=0}.
\end{aligned}
\end{equation}
The $g_{\mathrm{E}}$-Laplacian of $\Psi$, which is just the Euclidean Laplacian of $\Psi$ pulled back to $\bar{\mathcal{W}}$, is computed as
$$
\begin{aligned}
\Delta_{\mathrm{E}}\Psi
&=\frac{\Delta_0\Psi}{\rho^{2}}
+\left(1+\frac{|\nabla_0\zeta|_{g_0}^2}{\rho^{2}}\right)\partial_y^2\Psi
-\frac{2\nabla_0\zeta\cdot(\nabla_0\partial_y\Psi)}{\rho^{2}}-\frac{\Delta_0\zeta-2\rho}{\rho^{2}}\partial_y\Psi.
\end{aligned}
$$
Note that we used the conformality of $\rho^{2}g_0$ with $g_0$ and the fact that $\mathbb{S}^2$ has dimension 2.

The regularity of $D[\zeta]\phi$ can be deduced by a standard elliptic regularity argument using the expressions of $\Delta_{\mathrm{E}}$ under the coordinate system $(x,y)$. The following proposition is proved similarly as Proposition 2.7. of \cite{ABZ2011}:
\begin{proposition}\label{RegDN}
Suppose $s>3$, $s'\leq s$, and $\zeta\in H^{s+0.5}$, $\phi\in H^{s'}$, such that $|\zeta|_{L^\infty}<1/2$. Then there is an increasing function $K$ such that
$$
\|D[\zeta]\phi\|_{H^{s'-1}}
\leq K\big(\|\zeta\|_{H^{s+0.5}}\big)\|\phi\|_{H^{s'}}.
$$
\end{proposition}

Furthermore, the linearization of $D[\zeta]\phi$ can be computed similarly as in the Euclidean case. We directly state the following proposition, which resembles Proposition 2.11. of \cite{ABZ2011}; the proof is quite standard as in \cite{Lannes2005}, and is omitted.
\begin{proposition}\label{LinDN}
Suppose $s>3$, $s'\leq s$, and $\zeta\in H^{s+0.5}$, $\phi\in H^{s'}$, such that $|\zeta|_{L^\infty}<1/2$. There is a neighbourhood $\mathfrak{U}$ of $\zeta$ in $H^{s+0.5}$ in which the mapping 
$$
\mathfrak{U}\ni\zeta\to D[\zeta]\phi\in H^{s'-1}
$$
is differentiable, and the differential along direction $\zeta_1$ is 
$$
-D[\zeta]\big(\mathfrak{b}\zeta_1\big)-\mathrm{div}_{g_0}(\mathfrak{v}\zeta_1),
$$
where 
$$
\mathfrak{b}=\left(1+\frac{|\nabla_{0}\zeta|^2}{(1+\zeta)^{2}}\right)^{-1}\left(
D[\zeta]\phi+\frac{\nabla_{0}\zeta\cdot\nabla_{0}\phi}{(1+\zeta)^{2}}\right),
\quad
\mathfrak{v}=\frac{\nabla_0\phi-\mathfrak{b}\nabla_0\zeta}{(1+\zeta)^2},
$$
as in Theorem \ref{Thm1}.
\end{proposition}

In order to para-linearize the Dirichlet-Neumann operator, we will lift all the quantities of interest from $\mathbb{S}^2\times[-1/2,0]$ to $\SU(2)\times[-1/2,0]$. Suppose $\Phi$ is the velocity potential inside the region enclosed by $M_\zeta$, and $\phi(x):=\Phi\big((1+\zeta(x))N_0(x)\big)$ is the Dirichlet boundary value of $\Phi$ on $M_\zeta$. Then the corresponding pulled-back $\Psi$ satisfies $\Delta_{\mathrm{E}}\Psi=0$ on $\mathbb{S}^2\times[-1/2,0)$, and $\Psi(x,0)=\phi(x)$. Writing $\rho^\sharp=1+\zeta^\sharp+y$, defining the elliptic differential operator 
\begin{equation}\label{L_zeta}
\begin{aligned}
\mathcal{L}_\zeta
&:=\Delta_{G_0}
+\left((\rho^\sharp)^{2}+|\nabla_{G_0}\zeta^\sharp|_{G_0}^2\right)\partial_y^2
-2\nabla_{G_0}\zeta^\sharp\cdot(\nabla_{G_0}\partial_y)
-\left(\Delta_{G_0}\zeta^\sharp-2\rho^\sharp\right)\partial_y
\end{aligned}
\end{equation}
on $\SU(2)\times[-1/2,0]$, the lift $\Psi^\sharp$ satisfies $\mathcal{L}_\zeta\Psi^\sharp=0$ on $\SU(2)\times[-1/2,0)$, and $\Psi^\sharp(x,0)=\phi^\sharp(x)$. In fact, $\mathcal{L}_\zeta$ is just the lift of $\rho^2\Delta_{\mathrm{E}}$. The quantity $D[\zeta]\phi$ is lifted as the boundary value of
\begin{equation}\label{N_zeta}
\mathcal{N}_\zeta\Psi^\sharp
:=\left(1+\frac{|\nabla_{G_0}\zeta^\sharp|_{G_0}^2}{(\rho^\sharp)^{2}}\right)\partial_y\Psi^\sharp
-\frac{\nabla_{G_0}\zeta^\sharp\cdot\nabla_{G_0}\Psi^\sharp}{(\rho^\sharp)^{2}}
\end{equation}
at $y=0$. Our task now becomes the following:

\begin{task}\label{TaskDN}
Let $\mathcal{L}_\zeta$ and $\mathcal{N}_\zeta$ be as in (\ref{L_zeta}) and (\ref{N_zeta}). Find a suitable $\mathbf{T}_3$-invariant symbol $a$ on $\SU(2)$ such that if $\mathcal{L}_\zeta\Psi^\sharp=0$ on $\SU(2)\times[-1/2,0)$, and $\Psi^\sharp(x,0)=\phi^\sharp(x)$, then 
$$
\mathcal{N}_\zeta\Psi^\sharp\big|_{y=0}=T_a\phi^\sharp+\mathbf{T}_3\text{-invariant smoothing terms}.
$$
\end{task}
The rest of this section is devoted to finding out this paralinearization formula. The idea of accomplishing Task \ref{TaskDN} is simple: considering the ``height" coordinate $y$ as a parameter of evolution, the quantity $\mathcal{N}_\zeta\Psi^\sharp$ is the final state of an elliptic evolutionary problem $\mathcal{L}_\zeta\Psi^\sharp=0$; by suitable factorization of the operator $\mathcal{L}_\zeta$, the boundary value $\mathcal{N}_\zeta\Psi^\sharp$ then admits an explicit expression. 

The idea is exactly the one employed by, for example, Appendix C of Chapter 12 of Taylor's book \cite{Taylor2013}. Its para-differential version is exactly the argument employed by \cite{AM2009} and \cite{ABZ2011}. We will basically follow \cite{AM2009} and \cite{ABZ2011}, to whose idea our argument is very close.

From now on, we will fix a real number $s>3$. We assume that the height function $\zeta\in H^{s+0.5}(\mathbb{S}^2)$, so that the surface $M_\zeta$ is of class $C^{s-0.5}_*\subset C^{2.5+}$. We also assume that the boundary value $\phi\in H^{s}(\mathbb{S}^2)\subset C^{s-1}_*(\mathbb{S}^2)\subset C^{2+}(\mathbb{S}^2)$. Although the lifting operator $\sharp$ from functions on $\mathbb{S}^2$ to functions on $\SU(2)$ does not improve Sobolev regularity, it obviously does not undermine Zygmund regularity either: 
\begin{proposition}\label{SobolevForT3}
For $\mathbf{T}_3$-invariant functions on $\SU(2)$, the Sobolev embedding
$$
H^s\big(\mathbf{T}_3\setminus\SU(2)\big)
\subset C^{s-1}_*\big(\mathbf{T}_3\setminus\SU(2)\big)
$$
is valid for $s>1$.
\end{proposition}
For example, we still have 
$$
\begin{aligned}
&\zeta^\sharp\in C^{s-0.5}_*\big(\mathbf{T}_3\setminus\SU(2)\big)
\subset C^{2.5+}\big(\mathbf{T}_3\setminus\SU(2)\big),\\
&\phi^\sharp\in C^{s-1}_*\big(\mathbf{T}_3\setminus\SU(2)\big)
\subset C^{2+}\big(\mathbf{T}_3\setminus\SU(2)\big),
\end{aligned}
$$
although $H^s(\SU(2))\not\subset C^{s-1}_*(\SU(2))$. From now on we we will always be dealing with symbols or functions on $\SU(2)$ with $\mathbf{T}_3$-invariance described above, so this improved Sobolev embedding always holds.

\subsection{Alinhac's Good Unknown}
We turn to the para-linearization of equation $\mathcal{L}_\zeta\Psi^\sharp=0$. We shall look at the Laplace equation $\Delta_{\mathrm{E}}\Psi=0$ inside $M_\zeta$, and after lifting to $\SU(2)\times[-1/2,0]$, we shall find the corresponding \emph{good unknown} in the sense of Alinhac \cite{Alinhac1989}; that is, we need to find the optimal unknown function which eliminates all loss of regularity due to the non-smooth coordinate change. The idea was originally introduced by Alinhac in \cite{Alinhac1989}, Subsection 3.1.. Lannes \cite{Lannes2005} noticed its importance to the study of the Dirichlet-Neumann operator. The para-differential version of the good unknown was employed by \cite{AM2009} and \cite{ABZ2011}. We will basically follow the approach of \cite{AM2009} and \cite{ABZ2011}. 

We start by finding out a harmonic function inside $M_\zeta$ that envelopes all information of the normal direction velocity. We observe a nice commutation property between $\Delta_{\mathrm{E}}$ with $\rho\Pa[\rho]$, the infinitesimal generator of dilation:
$$
\big[\rho^2\Delta_\mathrm{E},\rho\Pa[\rho]\big]
=(2-2\rho^2)\Delta_{\mathrm{E}},
$$
we find that $\rho\partial_\rho\Phi$ is a harmonic function inside $M_\zeta$. Since $\rho=1+\zeta(x)+y$ for our choice of coordinates $(x,y)\in\mathbb{S}^2\times[-1/2,0]$, by pulling back to $\bar{\mathcal{W}}$, it follows that 
$$
\Delta_\mathrm{E}\mathfrak{B}=0,
\quad\text{where}\quad
\mathfrak{B}=\rho\partial_y\Psi.
$$
Lifting to $\SU(2)\times[-1/2,0]$, setting $\rho^\sharp=1+\zeta^\sharp+y$, we find
$$
\mathcal{L}_\zeta\mathfrak{B}^\sharp=0,
\quad\text{where}\quad
\mathfrak{B}^\sharp=\rho^\sharp\partial_y\Psi^\sharp.
$$
We define scalar functions
\begin{equation}\label{Beta13}
\begin{aligned}
\beta_1(y;x)&=\big(\rho^\sharp(x,y)\big)^{2}+|\nabla_{G_0}\zeta^\sharp(x)|_{G_0}^2,
\\
\beta_3(y;x)&=-\Delta_{G_0}\zeta^\sharp(x)+2\rho^\sharp(x,y),
\end{aligned}
\end{equation}
and then define a classical differential symbol
\begin{equation}\label{Beta2}
\begin{aligned}
\beta_2(x,l)
&=\sum_{j=1}^3X_j\zeta^\sharp(x)\cdot\sigma[X_j](l)\\
&=\text{the symbol of }\nabla_{G_0}\zeta^\sharp\cdot\nabla_{G_0},
\end{aligned}
\end{equation}
where the left-invariant vector fields $X_j$ are as in (\ref{LISU(2)}). The operator $\mathcal{L}_\zeta$ is now re-written as
$$
\mathcal{L}_\zeta
=\Delta_{G_0}+\beta_1\partial_y^2-2\nabla_{G_0}\zeta^\sharp\cdot\nabla_{G_0}\partial_y+\beta_3\partial_y.
$$
We are now at the place to introduce the good unknown in the sense of Alinhac:
\begin{proposition}\label{GoodUnknown}
Write $\mathfrak{B}=\rho\partial_y\Psi$. Defining the good unknown $W^\sharp:=\Psi^\sharp-T_{1/\rho^\sharp}T_{\mathfrak{B}^\sharp}\zeta^\sharp$ and the para-differential operator on $\SU(2)\times[-1/2,0]$
$$
\mathcal{P}_\zeta
:=\Delta_{G_0}+T_{\beta_1}\partial_y^2-2T_{\beta_2}\partial_y+T_{\beta_3}\partial_y
$$
corresponding to $\mathcal{L}_\zeta$, we have
$$
\big\|\mathcal{P}_\zeta W^\sharp\big\|_{C^0_yH_x^{s+0.5}}
\lesssim K\big(\|\zeta\|_{H^{s+0.5}}\big)\|\phi\|_{H^s}
$$
on $\SU(2)\times[-1/2,0]$, where $K$ is an increasing function, approximately linear when the argument is small.
\end{proposition}
\begin{remark}
Another reason that $W^\sharp$ has better regularity than $\Psi^\sharp$ is that, by Theorem \ref{Compo2}, we have $T_{1/\rho^\sharp}T_{\mathfrak{B}^\sharp}=T_{\partial_y\Psi^\sharp}\mod\Op\Sigma_{<1/2}^{-(s-2)}$, so
$$
W^\sharp=\Psi^\sharp-T_{\partial_y\Psi^\sharp}\zeta^\sharp
\mod C^0_yH_x^{2s-1.5},
$$
and the difference resembles the remainder term in Bony's para-linearization theorem, i.e. Theorem \ref{Bony}, for the composition $\big[\Phi\big((1+\zeta+y)N_0\big)\big]^\sharp$.
\end{remark}

\begin{notation}\label{Mod}
For any index $s'>0$, we use
$$
\mod C^0_yH_x^{s'}
\quad\text{or simply}\quad
\mod H_x^{s'}
$$
to refer to equality between quantities whose difference has $C^0_yH_x^{s'}$ (or $H^{s'}_x$) norm on $\mathbf{T}_3\setminus\SU(2)$ controlled by $K\big(\|\zeta\|_{H^{s+0.5}}\big)\|\phi\|_{H^s}$.
\end{notation}

\begin{proof}[Proof of Proposition \ref{GoodUnknown}]
Throughout the proof, we employ the ``mod" notation defined in \ref{Mod}.

The proof is divided into three steps.

\textbf{Step 1: intermediate equation for $\mathcal{P}_\zeta w^\sharp$.} Since $\Delta_{\mathrm{E}}\Psi=0$, by exactly the same argument as in \cite{ABZ2011}, Section 3, we obtain
$$
\nabla_0\Psi,\,\partial_y\Psi\in C^0_yH^{s-1}_x(\mathbb{S}^2),
\quad
\partial_y^2\Psi\in C^0_yH^{s-2}_x(\mathbb{S}^2).
$$
Upon lifting to $\SU(2)$, it follows that
$$
\nabla_{G_0}\Psi^\sharp,\,\partial_y\Psi^\sharp\in C^0_yH^{s-1}_x\big(\mathbf{T}_3\setminus\SU(2)\big),
\quad
\partial_y^2\Psi^\sharp\in C^0_yH^{s-2}_x\big(\mathbf{T}_3\setminus\SU(2)\big).
$$
In fact
$$
\big\|\nabla_{G_0}\Psi^\sharp\big\|_{C^0_yH^{s-1}_x}+
\big\|\partial_y\Psi^\sharp\big\|_{C^0_yH^{s-1}_x}
\lesssim \|\phi\|_{H^s},
\quad
\big\|\partial_y^2\Psi^\sharp\big\|_{C^0_yH^{s-2}_x}
\lesssim \|\phi\|_{H^s}.
$$
Using the para-product estimate, i.e. Theorem \ref{ParaPrd} and Remark \ref{R(a)Symbol}, since $\zeta^\sharp\in H^{s+0.5}\big(\mathbf{T}_3\setminus\SU(2)\big)$, we obtain
$$
\begin{aligned}
\mathcal{P}_\zeta\Psi^\sharp
&=\big(\mathcal{P}_\zeta-\mathcal{L}_\zeta\big)\Psi^\sharp
\quad\text{(precise equality)}\\
&=-\left(T_{\partial_y^2\Psi^\sharp}\beta_1-2\sum_{j=1}^3T_{X_j\partial_y\Psi^\sharp}X_j\zeta^\sharp+T_{\partial_y\Psi^\sharp}\beta_3\right)
\mod C_y^0H_x^{s+0.5}.
\end{aligned}
$$

Thus we compute
\begin{equation}\label{P_zetaU}
\begin{aligned}
\mathcal{P}_\zeta w^\sharp
&=\mathcal{P}_\zeta\left(\Psi^\sharp-T_{1/\rho^\sharp}T_{\mathfrak{B}^\sharp}\zeta^\sharp\right)
\quad\text{(precise equality)}\\
&=-\left(T_{1/\rho^\sharp}\mathcal{P}_\zeta T_{\mathfrak{B}^\sharp}\zeta^\sharp +T_{\partial_y^2\Psi^\sharp}\beta_1-2\sum_{j=1}^3T_{X_j\partial_y\Psi^\sharp}X_j\zeta^\sharp
+T_{\partial_y\Psi^\sharp}\beta_3\right)\\
&\quad-\left[\mathcal{P}_\zeta,T_{1/\rho^\sharp}\right]T_{\mathfrak{B}^\sharp}\zeta^\sharp
\mod C_y^0H_x^{s+0.5}.
\end{aligned}
\end{equation}

\textbf{Step 2: elimination of commutator.} To eliminate the commutator $\left[\mathcal{P}_\zeta,T_{1/\rho^\sharp}\right]T_{\mathfrak{B}^\sharp}\zeta^\sharp$ in (\ref{P_zetaU}), we need the composition formula for para-differential operators, i.e. Theorem \ref{Compo2}. The ``classical version" of $[\mathcal{P}_\zeta,T_{1/\rho^{\sharp}}]$ is just $[\mathcal{L}_\zeta,1/\rho^\sharp]$, whose projection to $\mathbb{S}^2\times[-1/2,0]$ is simply 
$$
\left[\rho^2\Delta_{\mathrm{E}},\frac{1}{\rho}\right]
=-2\partial_y\cdot\nabla_{\mathrm{E}}
=-2\partial_y.
$$
Here we used $\Delta_{\mathrm{E}}(1/\rho)=0$. Lifting to $\SU(2)\times[-1/2,0]$, it follows that $[\mathcal{L}_\zeta,1/\rho^\sharp]=-2\partial_y$. 

To pass to the para-differential version $[\mathcal{P}_\zeta,T_{1/\rho^{\sharp}}]$, we shall keep exploiting the fact that $\mathcal{L}_\zeta$ is a classical differential operator. Here with a little abuse of notation, given symbols $a,b$ on $\SU(2)$ leading to classical differential operators, we write $\{a,b\}$ for \emph{the symbol of the commutator} $[\Op(a),\Op(b)]$ (in usual notation the Poisson bracket $\{\cdot,\cdot\}$ only involves first order differentiation of symbols). Now writing down
\begin{equation}\label{[P_zeta,T]}
\begin{aligned}
\big[\mathcal{P}_\zeta,T_{1/\rho^{\sharp}}\big]
&=\big[\Delta_{G_0},T_{1/\rho^{\sharp}}\big]+\big[T_{\beta_1}\partial_y^2,T_{1/\rho^{\sharp}}\big]
-2\big[T_{\beta_2}\partial_y,T_{1/\rho^{\sharp}}\big]
+\big[T_{\beta_3}\partial_y,T_{1/\rho^{\sharp}}\big],
\end{aligned}
\end{equation}
it is not hard to see, with the aid of Theorem \ref{Compo2}, Formula (\ref{SU(2)Compo})  and also Remark \ref{Para-Classical}, that each term in the right-hand-side corresponds to exactly its classical counterpart in $[\mathcal{L}_\zeta,1/\rho^\sharp]$ with an affordable regularizing error. 

For example, by assumption, the number $r:=s-1.5>1.5$, so with $\#_{r;\mathscr{Q}}$ as in Formula (\ref{SU(2)Compo}), the third term in (\ref{[P_zeta,T]}) becomes
$$
\begin{aligned}
\big[T_{\beta_2}\partial_y,T_{1/\rho^{\sharp}}\big]
&=T_{\beta_2}T_{\partial_y(1/\rho^{\sharp})}
+T_{\beta_2}T_{1/\rho^{\sharp}}\partial_y
-T_{1/\rho^\sharp}T_{\beta_2}\partial_y
\quad\text{(precise equality)}\\
&=T_{\beta_2\#_{r;\mathscr{Q}}\partial_y(1/\rho^{\sharp})}
+T_{\{\beta_2,1/\rho^\sharp\}}\partial_y
\mod\Op\Sigma_{<1/2}^{-r}+\Op\Sigma_{<1/2}^{-r}\partial_y,
\end{aligned}
$$
while on the other hand 
$$
\begin{aligned}
\left[\nabla_{G_0}\zeta^\sharp\cdot\nabla_{G_0}\partial_y,\frac{1}{\rho^\sharp}\right]
&=\nabla_{G_0}\zeta^\sharp\cdot\nabla_{G_0}\left(\partial_y\left(\frac{1}{\rho^\sharp}\right)\cdot\right)
+\left[\nabla_{G_0}\zeta^\sharp\cdot\nabla_{G_0},\frac{1}{\rho^\sharp}\right]\partial_y\\
&=\Op\left(\beta_2\#_{r;\mathscr{Q}}\partial_y(1/\rho^{\sharp})\right)+\Op\left(\{\beta_2,1/\rho^\sharp\}\right)\partial_y
\end{aligned}
$$
in $[\mathcal{L}_\zeta,1/\rho^\sharp]$. Here we used the fact that $\beta_2$ is the symbol of a classical first order differential operator. In other words, Remark \ref{Para-Classical} ensures that $[\mathcal{P}_\zeta,T_{1/\rho^\sharp}]$ must coincide with the para-differential operator corresponding to $[\mathcal{L}_\zeta,1/\rho^\sharp]$, which is merely $-2\partial_y$, up to an error term:
$$
\begin{aligned}
\big[\mathcal{P}_\zeta,T_{1/\rho^{\sharp}}\big]
=-2\partial_y
\mod\Op\Sigma_{<1/2}^{-(s-1.5)}\partial_y^2+\Op\Sigma_{<1/2}^{-(s-2.5)}\partial_y+\Op\Sigma_{<1/2}^{-(s-2.5)},
\end{aligned}
$$
The reason that there is an error term of order $-(s-2.5)$ is that the $\mathbf{T}_3$-invariant symbols are $\beta_1,\beta_2$ are of $H^{s-0.5}$ regularity in $x$, while $\beta_3\in H^{s-1.5}\subset C_*^{s-2.5}$.

Since $\partial_y^k\Psi^\sharp\in C_y^0H^{s-k}_x\big(\mathbf{T}_3\setminus\SU(2)\big)$, we find that $\partial_yT_{\mathfrak{B}^\sharp}\zeta^\sharp=T_{\partial_y\mathfrak{B}^\sharp}\zeta^\sharp$ is still of class $C_y^0H^{s+0.5}_x$. On the other hand, the function $\partial_y^2T_{\mathfrak{B}^\sharp}\zeta^\sharp=T_{\partial_y^2\mathfrak{B}^\sharp}\zeta^\sharp$ is of class $C_y^0H^{s-0.5}_x$, since $\partial^2_y\mathfrak{B}^\sharp\in H_x^{s-3}\subset C^{-1}_*$, so the para-product operator $T_{\partial^2_y\mathfrak{B}^\sharp}$ has order 1, according to Proposition \ref{T_aNegIndex}. Taking into account the regularizing operators, we successfully proved that the commutator $\left[\mathcal{P}_\zeta,T_{1/\rho^\sharp}\right]T_{\mathfrak{B}^\sharp}\zeta^\sharp\in C_y^0H^{s+0.5}_x$. Summarizing, we find that (\ref{P_zetaU}) in fact implies
\begin{equation}\label{GoodUnknownTemp1}
\mathcal{P}_\zeta w^\sharp
=-\left(T_{1/\rho^\sharp}\mathcal{P}_\zeta T_{\mathfrak{B}^\sharp}\zeta^\sharp +T_{\partial_y^2\Psi^\sharp}\beta_1-2\sum_{j=1}^3T_{X_j\partial_y\Psi^\sharp}X_j\zeta^\sharp+T_{\partial_y\Psi^\sharp}\beta_3\right)
\mod C_y^0H_x^{2s-2.5}.
\end{equation}
Note that $s>3$, hence $2s-2.5>s+0.5$.

\textbf{Step 3: final elimination.} Now similarly as in the proof of Lemma 3.17. in \cite{ABZ2011}, we compute
\begin{equation}\label{GoodUnknownTemp2}
\begin{aligned}
\mathcal{P}_\zeta T_{\mathfrak{B}^\sharp}\zeta^\sharp
&=\left(\Delta_{G_0}+T_{\beta_1}\partial_y^2-2T_{\beta_2}\partial_y+T_{\beta_3}\partial_y\right)T_{\mathfrak{B}^\sharp}\zeta^\sharp\\
&=2\sum_{j=1}^3T_{X_j\mathfrak{B}^\sharp}X_j\zeta^\sharp
+\sum_{j=1}^3T_{X_j\zeta^\sharp\partial_y\mathfrak{B}^\sharp}X_j\zeta^\sharp
+T_{\mathfrak{B}^\sharp}\Delta_{G_0}\zeta^\sharp
\mod C_y^0H_x^{s+0.5}.
\end{aligned}
\end{equation}
Here we used the $XT_au=T_{Xa}u+T_aXu$ for a left-invariant vector field $X\in\mathfrak{su}(2)$, and the fact that $\mathcal{L}_\zeta\mathfrak{B}^\sharp=0$. Thus
\begin{equation}\label{GoodUnknownTemp3}
\begin{aligned}
\mathcal{P}_\zeta T_{\mathfrak{B}^\sharp}\zeta^\sharp
&=2\sum_{j=1}^3T_{X_j\mathfrak{B}^\sharp}X_j\zeta^\sharp
-T_{\partial_y\mathfrak{B}^\sharp}\beta_1
-T_{\mathfrak{B}^\sharp}\beta_3
&\mod C_y^0H_x^{s+0.5}\\
&=2\sum_{j=1}^3T_{\rho^\sharp X_j\partial_y\Psi^\sharp}X_j\zeta^\sharp
-T_{\rho^\sharp\partial_y^2\Psi^\sharp}\beta_1
-T_{\rho^\sharp\partial_y\Psi^\sharp}\beta_3
&\mod C_y^0H_x^{s+0.5}
\end{aligned}
\end{equation}
Here in the first equality of (\ref{GoodUnknownTemp3}), when computing 
$$
\begin{aligned}
\beta_1
&=|\nabla_{G_0}\zeta^\sharp|_{G_0}^2
&\mod H^{s+0.5}_x\\
&=\sum_{j=1}^3|X_j\zeta^\sharp|^2
&\mod H^{s+0.5}_x\\
&=2\sum_{j=1}^3T_{X_j\zeta^\sharp}X_j\zeta^\sharp
&\mod H^{s+0.5}_x,
\end{aligned}
$$
we used the para-linearization theorem of Bony, i.e. Theorem \ref{Bony}; in the third equality, we used 
$$
\begin{aligned}
2\sum_{j=1}^3T_{\partial_y\Psi^\sharp X_j\rho^\sharp}X_j\zeta^\sharp-T_{\partial_y\Psi^\sharp}\beta_1
&=2\sum_{j=1}^3T_{\partial_y\Psi^\sharp}T_{X_j\zeta^\sharp}X_j\zeta-T_{\partial_y\Psi^\sharp}\sum_{j=1}^3|X_j\zeta^\sharp|^2
&\mod C_y^0H^{2s-2.5}_x\\
&=0 
&\mod C_y^0H^{2s-2.5}_x.
\end{aligned}
$$
The proof is completed by noting that the right-hand-side of (\ref{GoodUnknownTemp2}) cancels with (\ref{GoodUnknownTemp1}) modulo $C_y^0H_x^{s+0.5}$, since $T_{1/\rho^\sharp}T_{\rho^\sharp}=\mathrm{Id}+\Op\Sigma_{<1/2}^{-(s-0.5)}$ by Theorem \ref{Compo2}.
\end{proof}

\subsection{Factorization}
With the elliptic para-differential equation satisfied by the good unknown $w^\sharp$ at hand, we are now at the place to factorize the equation to obtain the behaviour of $\partial_y\Psi^\sharp$ near the boundary $y=0$, to accomplish Task \ref{TaskDN}. This is the essential step for which an explicit, global symbolic calculus on $\SU(2)$ is necessary. The framework of our argument is parallel to Lemma 3.18. of \cite{ABZ2011}, but the details are quite different, since we are obviously working with a different geometry.

So we still assume $s>3$, $\zeta\in H^{s+0.5}(\mathbb{S}^2)$, and set $\beta_1$, $\beta_2$ and $\beta_3$ as in (\ref{Beta13})-(\ref{Beta2}). Imitating Lemma 3.18. of \cite{ABZ2011}, we seek for first-order rough symbols $a,A$ on $\SU(2)$, depending on the parameter $y\in[-1/2,0]$, such that
\begin{equation}\label{Factorization}
\begin{aligned}
\mathcal{P}_\zeta
&=\Delta_{G_0}+T_{\beta_1}\partial_y^2-2T_{\beta_2}\partial_y+T_{\beta_3}\partial_y\\
&=T_{\beta_1}(\partial_y-T_a)(\partial_y-T_A)+R_0+R_1\partial_y.
\end{aligned}
\end{equation}
Here we write $\delta=\min(0.5,s-3)>0$, and suppose $R_0\in\Op\Sigma_{<1/2}^{0.5-\delta}$, $R_1\in\Op\Sigma_{<1/2}^{-0.5-\delta}$.

We need to apply Theorem \ref{Compo2} and the composition formula (\ref{SU(2)Compo}) on $\SU(2)$, with the left-invariant vector fields $\Pa[+,-,0]$ given in (\ref{CANSU(2)}) and RT-admissible tuple $\mathscr{Q}_{+,-,0}$ given in (\ref{SU(2)q}). Suppose that $a,A$ are both of class $C_y^1\mathcal{A}^1_{0.5+\delta}(\SU(2))$. Write $a=a_1+a_0$, $A=A_1+A_0$, with subscript $1,0$ denoting symbols in $\mathcal{A}^1_{1.5+\delta}(\SU(2))$ and $\mathcal{A}^0_{0.5+\delta}(\SU(2))$ respectively. Then we propose the system to be solved as
\begin{equation}\label{AaSystemFull}
\begin{aligned}
a_1A_1+\sum_{\mu\in\Ind}\Df_{\mu}a_1\partial_{\mu}A_1
+a_1A_0+a_0A_1-\partial_yA_1
&=\frac{\sigma[\Delta_{G_0}]}{\beta_1}\\
a_1+A_1+a_0+A_0
&=\frac{2\beta_2-\beta_3}{\beta_1}.
\end{aligned}
\end{equation}
By Theorem \ref{Compo2} and Formula (\ref{SU(2)Compo}), we find that $T_{1/\beta_1}T_{\beta_1}=\mathrm{Id}\mod\Sigma_{<1/2}^{-(s-1.5)}$, so comparing symbols of same order, we find that the resulting $R_0$ and $R_1$ match with our requirement in (\ref{Factorization}).

Now we are at the place to solve (\ref{AaSystemFull}). Preserving the highest order symbols, we obtain
$$
\begin{aligned}
a_1A_1
=\frac{\sigma[\Delta_{G_0}]}{\beta_1},
\quad
a_1+A_1
=\frac{2\beta_2}{\beta_1}.
\end{aligned}
$$
Note that $\sigma[\Delta_{G_0}]=-l(l+1)\I[l]$, so for each label of representation $l\in\hN$ we solve, with the convention that the Hermitian part of $A_1(y;x,l)$ is positive definite for $l$ large,
\begin{equation}\label{AsSystem1}
\begin{aligned}
a_1(y;x,l)&=\frac{\beta_2(x,l)}{\beta_1(y;x)}-\tilde{\lambda}_1(y;x,l)\\
A_1(y;x,l)&=\frac{\beta_2(x,l)}{\beta_1(y;x)}+\tilde{\lambda}_1(y;x,l).
\end{aligned}
\end{equation}
Here we set
$$
\tilde{\lambda}_1(y;x,l)
:=\frac{\sqrt{\beta_2(x,l)^2+\beta_1(y;x)l(l+1)}}{\beta_1(y;x)}
:=\sqrt{\frac{l(l+1)}{\beta_1(y;x)}}\sum_{k=0}^\infty\binom{k}{1/2}\left(\frac{\beta_2(x,l)^2}{\beta_1(y;x)l(l+1)}\right)^k.
$$
At this level there is no issue of commutativity, as the matrix $\beta_2(x,l)$ commutes with $\beta_1(y;x)l(l+1)$, which is simply a scaling transformation on each $\Hh[l]$. Corollary \ref{CoroPSSymbol} immediately implies that the power series in the definition of square root sums to a symbol of class $\mathcal{A}^0_{s-1.5}$. Thus the symbols $\tilde{\lambda}_1(y;x,l)$, $a_1(y;x,l)$ and $A_1(y;x,l)$ are indeed of class $\mathcal{A}^1_{s-1.5}$. This justifies our search for symbols $a,A$ up to highest order. 

Now the equations for $a_0,A_0$ read
$$
\begin{aligned}
a_1A_0+a_0A_1
&=\partial_yA_1-\sum_{\mu\in\Ind}\Df_{\mu}a_1\partial_{\mu}A_1\\
a_0+A_0
&=-\frac{\beta_3}{\beta_1}.
\end{aligned}
$$
The issue here is that the symbol $\partial_\mu A_1$ does not commute with the symbol $\beta_2$ anymore, though all other terms do. Fortunately, we may content ourselves with solutions up to error of order $-1$. Since $A_1$ and $a_0$ are of order 1 and 0 in the sense of Definition \ref{2Order}, we find that $[a_0,A_1]$ should be a symbol of order $0$, thus negligible. Replacing $a_0A_1$ by $A_1a_0+[a_0,A_1]$, we can then re-write the system as
$$
\begin{aligned}
a_1A_0+A_1a_0
&=\partial_yA_1-\sum_{\mu\in\Ind}\Df_{\mu}a_1\partial_{\mu}A_1
\mod \mathcal{A}^{0}_{s-2.5}\\
a_0+A_0
&=-\frac{\beta_3}{\beta_1}.
\end{aligned}
$$
Multiplying the second equation by $A_1$ from the left, subtracting it from the first, we can explicitly solve $A_0$, and then $a_0$, up to an error in $\mathcal{A}^{-1}_{s-2.5}$.

It is thus safe to choose
\begin{equation}\label{AsSystem0}
\begin{aligned}
a_0&=\frac{\tilde{\lambda}_1^{-1}}{2}\partial_yA_1
+\frac{\tilde{\lambda}_1^{-1}\beta_3\beta_2-2\beta_1\beta_3+1}{2\beta_1^2}
-\frac{\tilde{\lambda}_1^{-1}}{2}\sum_{\mu\in\Ind}\Df_{\mu}a_1\partial_{\mu}A_1\\
A_0&=-\frac{\tilde{\lambda}_1^{-1}}{2}\partial_yA_1
-\frac{\tilde{\lambda}_1^{-1}\beta_3\beta_2+1}{2\beta_1^2}
+
\frac{\tilde{\lambda}_1^{-1}}{2}\sum_{\mu\in\Ind}\Df_{\mu}a_1\partial_{\mu}A_1.
\end{aligned}
\end{equation}
Applying Proposition \ref{PSSymbol} and the Corollary, we justify $a_0,A_0\in\mathcal{A}^0_{s-2.5}$, and this final result justifies the legitimacy of neglecting $[a_0,A_1]$.

To sum up, we have successfully obtained the desired factorization (\ref{Factorization}), with $a,A$ given by (\ref{AsSystem1})-(\ref{AsSystem0}), and 
$$
\|R_0\|_{H^{s+0.5-\delta}\to H^s},\,
\|R_1\|_{H^{s-0.5-\delta}\to H^s}
\lesssim K\big(\|\zeta\|_{H^{s+0.5}}\big)\|\phi\|_{H^s}.
$$

The next step is to find out what the boundary value $\partial_yw^\sharp\big|_{y=0}$ looks like. We state the following lemma, whose proof is identical to that of Proposition 3.19. in \cite{ABZ2011}:
\begin{lemma}[Regularity for Evolutionary Elliptic Para-Differential Equation]\label{MaxReg}
Suppose $r>1$, the $y$-dependent symbol $a\in C^1_y\mathcal{A}^1_{r}(\SU(2))$ for $y\in[-1/2,0]$, such that the Hermitian part of $a$ is $\geq c|l|$ for some $c>0$ when $l$ is sufficiently large. Suppose $N<0$, $W\in C^0_yH^{-N}_x$ solves the evolutionary elliptic para-differential equation
$$
\partial_yW+T_aW=f\in C^0_yH^{s}_x,\quad s\in\mathbb{R},
$$
then in fact $W(0)\in H^{s+1-\varepsilon}$:
$$
\|W(0)\|_{H^{s+1-\varepsilon}}
\lesssim_{\varepsilon}\|f\|_{C^0_yH^{s}_x}+\|w\|_{C^0_yH^{-N}_x}.
$$
\end{lemma}
With the factorization (\ref{Factorization}), we set $\tilde W=(\partial_y-T_A)W^\sharp$ and find
$$
\partial_y\tilde W+T_a\tilde W=R_0\tilde W+R_0\partial_y\tilde W
\mod C^0_yH^{2s-1.5}.
$$
From the expression $a=a_1+a_0$, where $a_1,a_0$ are as in (\ref{AsSystem1})-(\ref{AsSystem0}), we find that $a$ in fact depends smoothly on $y$. On the other hand, $\tilde W\in C^0_yH^{s-1}_x$, with norm controlled in terms of $\|\zeta\|_{H^{s+0.5}}$ and $\|\phi\|_{H^s}$, so we can apply Lemma \ref{MaxReg} to conclude
\begin{equation}\label{AUy=0}
\left\|\partial_yW^\sharp\big|_{y=0}-T_AW^\sharp\big|_{y=0}\right\|_{H^{s+0.5}}
\lesssim K\big(\|\zeta\|_{H^{s+0.5}}\big)\|\phi\|_{H^s},
\end{equation}
where 
$$
A=A_1+A_0,
$$
with $A_1$ given in (\ref{AsSystem1}) and $A_0$ given in (\ref{AsSystem0}). All the symbols and functions of concern are $\mathbf{T}_3$-invariant.

\subsection{Concluding the Proof of Theorem \ref{Thm1}}
We are at the stage to complete Task \ref{TaskDN}. With a little abuse of notation, if $b=\sum_{j}b^jX_j$ is a vector field, then we write
$$
T_b\cdot\nabla_{G_0}=\sum_{j}T_{b^j}X_j.
$$
Thus we compute, with Bony's para-linearization theorem (and still the ``mod" notation in \ref{Mod}),
\begin{equation}\label{GoodUnknownTemp4}
\begin{aligned}
(\rho^\sharp)^2\mathcal{N}_\zeta\Psi^\sharp
&=\beta_1\partial_y\Psi^\sharp-\nabla_{G_0}\zeta^\sharp\cdot\nabla_{G_0}\Psi^\sharp
\quad\text{(precise equality)}\\
&=T_{\beta_1}\partial_y\Psi^\sharp
+2T_{\partial_y\Psi^\sharp\nabla_{G_0}\zeta^\sharp}\cdot\nabla_{G_0}\zeta^\sharp
-T_{\nabla_{G_0}\zeta^\sharp}\cdot\nabla_{G_0}\Psi^\sharp-T_{\nabla_{G_0}\Psi^\sharp}\cdot\nabla_{G_0}\zeta^\sharp
\mod C_y^0H^{s+0.5}_x
\end{aligned}
\end{equation}
Substituting in $\Psi^\sharp=W^\sharp+T_{1/\rho^\sharp}T_{\mathfrak{B}^\sharp}\zeta^\sharp$, we find that the second and third term in (\ref{GoodUnknownTemp4}) cancel out the $\nabla_{G_0}\zeta^\sharp$ within them:
$$
T_{\partial_y\Psi^\sharp\nabla_{G_0}\zeta^\sharp}\cdot\nabla_{G_0}\zeta^\sharp
-T_{\nabla_{G_0}\zeta^\sharp}\cdot\nabla_{G_0}\Psi^\sharp
=-T_{\nabla_{G_0}\zeta^\sharp}\cdot\nabla_{G_0}W^\sharp
\mod C_y^0H^{s+0.5}_x.
$$
Thus using the fact $\mathfrak{B}^\sharp,\partial_y\mathfrak{B}^\sharp\in C^0_yH^{s+0.5}_{x}$, (\ref{GoodUnknownTemp4}) reduces to
$$
(\rho^\sharp)^2\mathcal{N}_\zeta\Psi^\sharp
=T_{\beta_1}\partial_yW^\sharp-T_{\nabla_{G_0}\zeta^\sharp}\cdot\nabla_{G_0}W^\sharp
-T_{\nabla_{G_0}\Psi^\sharp-\partial_y\Psi^\sharp\nabla_{G_0}\zeta^\sharp}\cdot\nabla_{G_0}\zeta^\sharp
\mod C_y^0H^{s+0.5}_x.
$$
Using (\ref{AUy=0}), substituting in (\ref{AsSystem1})-(\ref{AsSystem0}) at $y=0$, we obtain the desired result: the Dirichlet-Neumann operator ${D[\zeta]\phi}$ is lifted to $\SU(2)$ as
\begin{equation}\label{DNParaLin}
\big(D[\zeta]\phi\big)^\sharp
=T_{\lambda}\big(\phi^\sharp-T_{\mathfrak{b}^\sharp}\zeta^\sharp\big)
-T_{\mathfrak{v}^\sharp}\cdot\nabla_{G_0}\zeta^\sharp
\mod H^{s+0.5}_x.
\end{equation}
Here with $\beta_1,\beta_2,\beta_3$ given in (\ref{Beta13})-(\ref{Beta2}), $\tilde{\lambda}_1$ and $A_0$ given in (\ref{AsSystem1})-(\ref{AsSystem0}) (evaluated at $y=0$), we have
\begin{equation}\label{DNlambda}
\begin{aligned}
\lambda(x,l)
&=\frac{\beta_1(x)\tilde{\lambda}_1(x,l)}{(1+\zeta^\sharp)^2}
+\frac{\beta_1(x)A_0(x,l)}{(1+\zeta^\sharp)^2}\\
&=\frac{\sqrt{\beta_2(x,l)^2+\beta_1(x)l(l+1)}}{(1+\zeta^\sharp)^2}
+\frac{\beta_1(x)A_0(x,l)}{(1+\zeta^\sharp)^2}\\
&=:\lambda^{(1)}(x,l)+\lambda^{(0)}(x,l).
\end{aligned}
\end{equation}
The quantities $\mathfrak{b},\mathfrak{v}$ on $\mathbb{S}^2$ are
\begin{equation}\label{bvExpression}
\begin{aligned}
\mathfrak{b}
&=\partial_y\Psi\big|_{y=0}
=\left(1+\frac{|\nabla_{0}\zeta|^2}{(1+\zeta)^{2}}\right)^{-1}\left(
{D[\zeta]\phi}+\frac{\nabla_{0}\zeta\cdot\nabla_{0}\phi}{(1+\zeta)^{2}}\right),\\
\\
\mathfrak{v}
&=\frac{\nabla_0\Psi-\partial_y\Psi\nabla_0\zeta}{\rho^2}\Bigg|_{y=0}
=\frac{\nabla_0\phi-\mathfrak{b}\nabla_0\zeta}{(1+\zeta)^2}.
\end{aligned}
\end{equation}
Note that we used $T_{\beta_2}=T_{\nabla_{G_0}\zeta^\sharp}\cdot\nabla_{G_0}$ and $\beta_1(y;x)A_1(y;x,l)-\beta_2(x,l)=\tilde{\lambda}_1(y;x,l)$. When $\zeta$ is close to 0, the error term is obviously quadratic in $(\zeta,\phi)\in H^{s+0.5}\times H^s$.

This finishes the proof of Theorem \ref{Thm1}.

\section{Para-differential Form of the System}\label{7}
In this section, we conclude the proof of Theorem \ref{Thm2} using the results we have obtained so far. Although the method is quite standard, the symmetrization process is more complicated than that in \cite{ABZ2011} due to non-commutativity of symbolic calculus on $\SU(2)$. We thus mainly focus on what is different from \cite{ABZ2011} due to this non-commutativity, and sketch everything else that is parallel.

\subsection{Para-linearizing the Full System}
We start by deriving the para-linearized form of the full system (\ref{EQ}). We still write $\rho=1+\zeta$, $\beta_1=\rho^2+|\nabla_0\zeta|_{g_0}^2$ at the level set $\{y=0\}$. As a consequence of the para-linearization formula for the Dirichlet-Neumann operator, we already know that 
$$
\begin{aligned}
\partial_t\zeta^\sharp
&=\big(D[\zeta]\phi\big)^\sharp\\
&=T_{\lambda}\big(\phi^\sharp-T_{\mathfrak{b}^\sharp}\zeta^\sharp\big)
-T_{\mathfrak{v}^\sharp}\cdot\nabla_{G_0}\zeta^\sharp
+f_1(\zeta,\phi),
\end{aligned}
$$
where $\lambda,\mathfrak{b},\mathfrak{v}$ are as in (\ref{DNlambda})-(\ref{bvExpression}),  $f_1(\zeta,\phi)$ is $\mathbf{T}_3$-invariant and 
$$
\|f_1(\zeta,\phi)\|_{H_x^{s+0.5}}
\lesssim \left(\|\zeta\|_{H_x^{s+0.5}}+\|\phi\|_{H_x^s}\right)^2.
$$

To proceed further, we introduce the actual \emph{good unknown} of interest, being the boundary value of $W^\sharp$ in Proposition \ref{GoodUnknown} modulo $H_x^{s+0.5}$:
\begin{equation}\label{ActualGUknown}
w^\sharp:=\phi^\sharp-T_{\mathfrak{b}^\sharp}\zeta^\sharp.
\end{equation}
Then with $\lambda,\mathfrak{b},\mathfrak{v}$ as in (\ref{DNlambda})-(\ref{bvExpression}), we find that the mapping
$$
\left(\begin{matrix}
\zeta^\sharp \\
\phi^\sharp
\end{matrix}\right)
\to
\left(\begin{matrix}
\zeta^\sharp \\
w^\sharp
\end{matrix}\right)
$$
is a diffeomorphism in a neighbourhood of $(0,0)\in H^{s+0.5}_x\times H^s_x$. In fact, the mapping $\phi^\sharp\to\phi^\sharp-T_{\mathfrak{b}^\sharp}\zeta^\sharp$ is linear in $\phi^\sharp$, and estimates in Section \ref{2} on the para-product operator imply
$$
\|T_{\mathfrak{b}^\sharp}\zeta^\sharp\|_{H^s_x}
\lesssim \|\zeta\|_{H_x^{s+0.5}}\|\phi\|_{H_x^s}.
$$
Thus when $\zeta$ is close to 0 in $H^{s+0.5}_x$, the linear operator $\phi^\sharp\to\phi^\sharp-T_{\mathfrak{b}^\sharp}\zeta^\sharp$ is invertible, and the inverse is smooth in $\zeta^\sharp$. Consequently, the passage of unknown from $(\zeta^\sharp,\phi^\sharp)$ to $(\zeta^\sharp,w^\sharp)$ does not lose any information. It then suffices to derive an equation governing $(\zeta^\sharp,w^\sharp)$.

To derive the evolution equation for $w^\sharp$, we need to study the terms in $\partial_t\phi$ separately. We start with the mean curvature.

\begin{proposition}\label{H(zeta)}
Suppose $s>3$, $\zeta\in H^{s+0.5}$. Still write $\rho=1+\zeta$, and $\beta_1$ for the lift of $\rho^2+|\nabla_0\zeta|_{g_0}^2$. The mean curvature $H(\zeta)$ admits a para-linearization formula
$$
\big(H(\zeta)\big)^\sharp=-T_{h}\zeta^\sharp
\mod H_x^{2s-3}.
$$
Here $h=h^{(2)}+h^{(1)}$ is a classical differential symbol, with $h^{(2)}=(\rho^\sharp)^3\beta_1^{-3/2}(\lambda^{(1)})^2$ being a multiplication of $(\lambda^{(1)})^2$, and $h^{(1)}$ takes the form
$$
h^{(1)}
=\sum_{j=1}^3 h_{1j}\big(\zeta^\sharp,\nabla_{G_0}\zeta^\sharp,\nabla_{G_0}^2\zeta^\sharp\big)\sigma[X_j]
+h_{10}\big(\zeta^\sharp,\nabla_{G_0}\zeta^\sharp,\nabla_{G_0}^2\zeta^\sharp\big),
$$
where the coefficients $h_{1j}$ are all $C^\infty$ in its arguments.
\end{proposition}
\begin{proof}
We are going to use the first variation formula to compute the mean curvature. The area of $M_\zeta$ is 
$$
A(\zeta)=\int_{\mathbb{S}^2}\rho\sqrt{\rho^2+|\nabla_0\zeta|_{g_0}^2}d\mu_0.
$$
Given any $\zeta_1\in C^\infty(\mathbb{S}^2)$, we consider the slightly varied surfaces $M_{\zeta+\theta\zeta_1}$, i.e. the surfaces given by the graphs $x\to (1+\zeta(x)+\theta\zeta_1(x))N_0(x)$. Recall that under the coordinate system $(x,y)$ introduced in the previous section, the hypersurface $y=0$ corresponds to the distorted sphere $M_\zeta$, and the vector field
$$
\nabla_{\mathrm{E}}y=-\frac{\nabla_0\zeta}{\rho^{2}}
+\left(1+\frac{|\nabla_0\zeta|_{g_0}^2}{\rho^{2}}\right)\partial_y
$$
along $\{y=0\}$ is perpendicular to the hypersurface. Thus $$
N(\iota)=\frac{\nabla_{\mathrm{E}}y}{|\nabla_{\mathrm{E}}y|}
$$
is the unit outer normal vector field along $M_\zeta$. Note that
$$
|\nabla_{\mathrm{E}}y|=\sqrt{1+\frac{|\nabla_0\zeta|_{g_0}^2}{\rho^2}}.
$$
Under this choice of coordinate, the variational vector field of this variation is $\zeta_1\partial_y\big|_{y=0}$, so 
$$
N(\iota)\cdot \big(\zeta_1\partial_y)
=\frac{\zeta_1\rho}{\sqrt{\rho^2+|\nabla_0\zeta|_{g_0}^2}}. 
$$

The first variation formula of area then gives
$$
\begin{aligned}
\frac{d}{d\theta}A(\zeta+\theta\zeta_1)\Bigg|_{\theta=0}
=\int_{\mathbb{S}^2} H(\zeta)\zeta_1\rho^2d\mu_0.
\end{aligned}
$$
On the other hand, a direct computation with integration by parts gives that the left-hand-side in the first variation formula reads
$$
\begin{aligned}
\frac{d}{d\theta}A(\zeta+\theta\zeta_1)\Bigg|_{\theta=0}
&=\int_{\mathbb{S}^2}
\left(-\rho\Delta_{g_0}\zeta
+\frac{\rho\mathrm{Hess}_{g_0}(\zeta)(\nabla_0\zeta,\nabla_0\zeta)}{\rho^2+|\nabla_0\zeta|_{g_0}^2}
+2\rho^2\right)
\frac{\zeta_1}{\sqrt{\rho^2+|\nabla_0\zeta|_{g_0}^2}}d\mu_0,
\end{aligned}
$$
Thus the mean curvature operator reads
\begin{equation}\label{H(zeta)Expr}
H(\zeta)
=\frac{1}{\rho\sqrt{\rho^2+|\nabla_0\zeta|_{g_0}^2}}\left(\Delta_{g_0}\zeta
-\frac{\mathrm{Hess}_{g_0}(\zeta)(\nabla_0\zeta,\nabla_0\zeta)}{\rho^2+|\nabla_0\zeta|_{g_0}^2}\right)
+\frac{2}{\rho\sqrt{\rho^2+|\nabla_0\zeta|_{g_0}^2}}.
\end{equation}

We lift (\ref{H(zeta)Expr}) to $\SU(2)$. Recalling that $\sum_{j=1}^3X_j\zeta^\sharp\sigma[X_j]$, where the left-invariant vector fields $X_i$ are as in (\ref{LISU(2)}), we compute the lift of the Hessian form as
$$
\begin{aligned}
\mathrm{Hess}_{G_0}(\zeta^\sharp)(\nabla_{G_0}\zeta^\sharp,\nabla_{G_0}\zeta^\sharp)
&=\sum_{i,j=1}^3\left(X_iX_j\zeta^\sharp-(\nabla_{X_i}X_j)\zeta^\sharp\right)\cdot X_i\zeta^\sharp\cdot X_j\zeta^\sharp\\
&=\Op(\beta_2(x,l)^2)\zeta^\sharp.
\end{aligned}
$$
Here we used that $\nabla_{X_i}X_j=[X_i,X_j]/2$ (with respect to the bi-invariant metric) on the Lie group $\SU(2)$. Recalling the symbol $\lambda^{(1)}=\sqrt{\beta_2^2+\beta_1l(l+1)}/(\rho^\sharp)^2$ in (\ref{DNlambda}), we find that the first term in (\ref{H(zeta)Expr}) are lifted as
\begin{equation}\label{H(zeta)2nd}
\begin{aligned}
\frac{\Delta_{G_0}\zeta^\sharp}{\rho^\sharp(x)\beta_1(x)^{1/2}}
-\frac{\mathrm{Hess}_{G_0}(\zeta^\sharp)(\nabla_{G_0}\zeta^\sharp,\nabla_{G_0}\zeta^\sharp)}{\rho^\sharp(x)\beta_1(x)^{3/2}}
&=\Op\left(\frac{-l(l+1)}{\rho^\sharp(x)\beta_1(x)^{1/2}}
-\frac{\beta_2^2(x,l)}{\rho^\sharp(x)\beta_1(x)^{3/2}}\right)\zeta^\sharp\\
&=-\Op\big(h^{(2)}(x,l)\big)\zeta^\sharp.
\end{aligned}
\end{equation}
By Bony's para-linearization theorem, (\ref{H(zeta)2nd}) is then para-linearized as (with the ``mod" notation in \ref{Mod} used)
$$
-T_{h^{(2)}}\zeta^\sharp
+T_{\tilde h^{(1)}}\cdot\nabla_{G_0}\zeta^\sharp
+T_{\tilde h_0}\zeta^\sharp
\mod H^{2s-3}_x,
$$
where $\tilde h^{(1)},\tilde h_0$ are $C^\infty$ function of $\big(\zeta^\sharp,\nabla_{G_0}\zeta^\sharp,\nabla_{G_0}^2\zeta^\sharp\big)$. Finally, the last term in (\ref{H(zeta)Expr}) is also lifted and para-linearized similarly. This completes the proof.
\end{proof}

We are now at the place to para-linearize the second order terms in the expression of $\partial_t\phi$, and derive the evolution equation of $w^\sharp$.
\begin{proposition}\label{Quad}
Suppose $s>3$ and $\zeta\in H^{s+0.5}$, $\phi\in H^s$. The evolution of $w^\sharp$ is given by
$$
\partial_tw^\sharp
=-T_{\mathfrak{v}^\sharp}\cdot\nabla_{G_0}w^\sharp
-T_{h}\zeta^\sharp
+f_2(\zeta,\phi).
$$
The error term satisfies 
$$
\|f_2(\zeta,\phi)\|_{H^{s}_x}\leq K\big(\|\zeta\|_{H^{s+0.5}_x}\big)\|\phi\|_{H^s_x},
$$
where $K$ is an increasing function vanishing linearly at zero.
\end{proposition}
\begin{proof}
We start by looking at the quadratic terms in the equation for $\partial_t\phi$. By a direct computation using our coordinate system $(x,y)$ introduced in last section, we argue just as in Lemma 3.26 of \cite{ABZ2011} as follows. Still write $\rho=1+\zeta$. Then with
$$
F(\mu,a,b,c)=\frac{\left(\mu^2a+b\cdot c\right)^2}{2\mu^2\big(\mu^2+|b|_{g_0}^2\big)},
$$
the quadratic part in the right-hand-side of the equation for $\phi$ in (\ref{EQ}) may be written as
\begin{equation}\label{RHSphi}
\begin{aligned}
-\frac{|\nabla_0\phi|_{g_0}^2}{2\rho^2}
+\frac{\left(\rho^2D[\zeta]\phi+\nabla_0\zeta\cdot\nabla_0\phi\right)^2}{2\rho^2\big(\rho^2+|\nabla_0\zeta|_{g_0}^2\big)}
=-\frac{|\nabla_0\phi|_{g_0}^2}{2\rho^2}
+F(\rho,D[\zeta]\phi,\nabla_0\zeta,\nabla_0\phi).
\end{aligned}
\end{equation}
Bony's para-linearization theorem implies that
\begin{equation}\label{|Dphi|^2}
\begin{aligned}
\left(-\frac{|\nabla_0\phi|_{g_0}^2}{2\rho^2}\right)^\sharp
&=-T_{(\rho^\sharp)^{-2}\nabla_{G_0}\phi^\sharp}\cdot\nabla_{G_0}\phi^\sharp\\
&=-T_{(\rho^\sharp)^{-2}\mathfrak{v}^\sharp}\cdot\nabla_{G_0}\phi^\sharp
-T_{(\rho^\sharp)^{-2}\mathfrak{b}^\sharp\nabla_{G_0}\zeta^\sharp}\cdot\nabla_{G_0}\phi^\sharp
\mod H^{2s-3}_x,
\end{aligned}
\end{equation}
and (here subscript for $F$ stands for partial differentiation)
$$
\big(F(\rho,D[\zeta]\phi,\nabla_0\zeta,\nabla_0\phi)\big)^\sharp
=T_{F_\mu^\sharp}\rho^\sharp
+T_{F_a^\sharp}\big(D[\zeta]\phi\big)^\sharp
+T_{F_b^\sharp}\cdot\nabla_{G_0}\zeta^\sharp
+T_{F_c^\sharp}\cdot\nabla_{G_0}\phi^\sharp
\mod H^{2s-3}_x.
$$
Since $\rho\in H^{s+0.5}_x$, we find $T_{F_\mu^\sharp}\rho^\sharp\in H^{s+0.5}_x$ and thus can be disregarded. At $(\mu,a,b,c)=(\rho,D[\zeta]\phi,\nabla_0\zeta,\nabla_0\phi)$, we compute, recalling (\ref{bvExpression}),
$$
\begin{aligned}
F_a
&=\frac{\mu^2a+b\cdot c}{\mu^2+|b|_{g_0}^2}
=\mathfrak{b},\\
F_b
&=\frac{\mu^2a+b\cdot c}{\mu^2\big(\mu^2+|b|_{g_0}^2\big)}
\left(c-\frac{(\mu^2a+b\cdot c)b}{\mu^2+|b|_{g_0}^2}\right)
=\mathfrak{bv},\\
F_c
&=\frac{\mu^2a+b\cdot c}{\mu^2\big(\mu^2+|b|_{g_0}^2\big)}b
=\frac{\mathfrak{b}\nabla_0\zeta}{\rho^2}.
\end{aligned}
$$
Thus
\begin{equation}\label{F^sharp}
\big(F(\rho,D[\zeta]\phi,\nabla_0\zeta,\nabla_0\phi)\big)^\sharp
=T_{\mathfrak{b}^\sharp}\big(D[\zeta]\phi\big)^\sharp
+T_{\mathfrak{b}^\sharp}T_{\mathfrak{v}^\sharp}\cdot\nabla_{G_0}\zeta^\sharp
+T_{(\rho^\sharp)^{-2}\mathfrak{b}^\sharp\nabla_{G_0}\phi^\sharp}\cdot\nabla_{G_0}\phi^\sharp
\mod H^{s}_x.
\end{equation}
Combining (\ref{|Dphi|^2}) and (\ref{F^sharp}), after suitable cancellation and commutation (causing acceptable error), we find that (\ref{RHSphi}) lifts to $\SU(2)$ as 
$$
-T_{\mathfrak{v}^\sharp}\cdot\nabla_{G_0}w^\sharp
+T_{\mathfrak{b}^\sharp}\big(D[\zeta]\phi\big)^\sharp
\mod H_x^{s}.
$$
Consequently, noticing that $\partial_tw^\sharp=\partial_t\phi^\sharp-T_{\partial_t\mathfrak{b}^\sharp}\phi^\sharp-T_{\mathfrak{b}^\sharp}\partial_t\zeta^\sharp$, we find
$$
\partial_t w^\sharp
=-T_{\mathfrak{v}^\sharp}\cdot\nabla_{G_0}w^\sharp
-T_{h}\zeta^\sharp
+T_{\partial_t\mathfrak{b}^\sharp}\phi^\sharp
\mod H_x^{s}.
$$
The error caused by para-linearization is quadratic in $\zeta$ and $\phi$, since the para-linearization already includes the linear part of the nonlinear expression. Finally, the regularity of $T_{\partial_t\mathfrak{b}^\sharp}\phi^\sharp$ is obtained by applying Proposition \ref{RegDN}-\ref{LinDN}, identically as in Lemma 3.27. of \cite{ABZ2011}. This finishes the proof.
\end{proof}

Just as in \cite{ABZ2011}, we obtain the para-linearization of system (\ref{EQ}), lifted to $\SU(2)$:
\begin{equation}\label{EQPara}
\begin{aligned}
\partial_t\zeta^\sharp
&=T_\lambda w^\sharp-T_{\mathfrak{v}^\sharp}\cdot\nabla_{G_0}\zeta^\sharp+f_1(\zeta,\phi)\\
\partial_tw^\sharp&=-T_{\mathfrak{v}^\sharp}\cdot\nabla_{G_0}w^\sharp
-T_{h}\zeta^\sharp
+f_2(\zeta,\phi),
\end{aligned}
\end{equation}
where the errors $f_1,f_2$ takes the explicit form
\begin{equation}\label{f1f2}
\begin{aligned}
f_1(\zeta,\phi)
&=\big(D[\zeta]\phi\big)^\sharp-T_\lambda\big(w^\sharp-T_{\mathfrak{v}^\sharp}\cdot\nabla_{G_0}\zeta^\sharp\big)\\
f_2(\zeta,\phi)
&=\left(-\frac{|\nabla_0\phi|_{g_0}^2}{2\rho^2}
+\frac{\left(\rho^2D[\zeta]\phi+\nabla_0\zeta\cdot\nabla_0\phi\right)^2}{2\rho^2\big(\rho^2+|\nabla_0\zeta|_{g_0}^2\big)}
+H(\zeta)\right)^\sharp
+T_{\mathfrak{v}^\sharp}\cdot\nabla_{G_0}w^\sharp
+T_{h}\zeta^\sharp
-T_{\partial_t\mathfrak{b}^\sharp}w^\sharp,
\end{aligned}
\end{equation}
and are affordable and in fact quadratic:
$$
\|f_1(\zeta,\phi)\|_{H_x^{s+0.5}}+\|f_2(\zeta,\phi)\|_{H_x^{s}}
\lesssim\left(\|\zeta\|_{H_x^{s+0.5}}+\|\phi\|_{H_x^s}\right)^2.
$$
We already know that (\ref{EQPara}) is equivalent to (\ref{EQ}) at least when $\|\zeta(t)\|_{H^{s+0.5}_x}$ remains sufficiently small. Thus we have successfully derived the para-linear form of (\ref{EQ}).

\subsection{Adjoint of Para-differential Operators}
In order to symmetrize the system (\ref{EQPara}), we need to understand the adjoint properties of the Dirichlet-Neumann operator and the mean curvature operator. We start from the former. We still employ the notation $\rho=1+\zeta$, and define the symbol $\lambda$ as in (\ref{DNlambda}). The following proposition reflects the self-adjointness of the Dirichlet-Neumann operator.

\begin{proposition}\label{AdjDN}
Still fix $s>3$ and $\zeta\in H^{s+0.5}$, $\phi\in H^s$. Then the symbol $\lambda$ of the Dirichlet-Neumann operator satisfies
\begin{equation}\label{AdjT_Lambda}
(\rho^\sharp)^{-2}\lambda^{\bullet;1,\mathscr{Q}}=
\lambda\#_{1;\mathscr{Q}}(\rho^\sharp)^{-2}
\mod\mathcal{A}^{-1}_{s-2.5}
\end{equation}
Here the operations $\#_{r;\mathscr{Q}}$ and ${\bullet;r,\mathscr{Q}}$ are defined in (\ref{ParaCompoSU2})-(\ref{ParaAdjSU2}). Thus the para-differential operator $T_\lambda$ satisfies
$$
T_{(\rho^\sharp)^{-2}}T_\lambda^*=T_\lambda T_{(\rho^\sharp)^{-2}}\mod\Op\Sigma^{-0.5}_{<1/2}.
$$
\end{proposition}

\begin{proof}
If we directly manipulate with (\ref{AsSystem0}) and (\ref{DNlambda}), the algebra in the proof of (\ref{AdjT_Lambda}) would be extremely involved. Thus we proceed differently by recalling the geometric interpretation of the Dirichlet-Neumann operator. 

Assume for the moment that $\zeta$ is \emph{smooth}. Recall that the linear operator taking the Dirichlet boundary value $\phi$ to the Neumann boundary value $\bar\nabla\Phi\cdot N(\iota)=\big(N_0\cdot N(\iota)\big)D[\zeta]\phi$ on $M_\zeta$ is self-adjoint with respect to the surface measure of $M_\zeta$; this is the starting point of our argument. Recall that $N_0\cdot N(\iota)
=\rho/\sqrt{\rho^2+|\nabla_0\zeta|_{g_0}^2}$, while the surface measure of $M_\zeta$ is $d\mu(\iota)=\rho\sqrt{\rho^2+|\nabla_0\zeta|_{g_0}^2}d\mu_0$. Consequently, we find, for smooth functions $\phi$, $\phi_1$ defined on $\mathbb{S}^2$,
$$
\int_{\mathbb{S}^2}D[\zeta]\phi\cdot\phi_1 \rho^2d\mu_0
=\int_{\mathbb{S}^2}\phi\cdot D[\zeta]\phi_1\cdot \rho^2d\mu_0.
$$
Lifting to $\SU(2)$, this is exactly
$$
\int_{\SU(2)}\mathfrak{D}\phi^\sharp\cdot\phi_1^\sharp \cdot(\rho^\sharp)^2d\mu(G_0)
=\int_{\SU(2)}\phi^\sharp\cdot \mathfrak{D}\phi_1^\sharp\cdot (\rho^\sharp)^2d\mu(G_0),
$$
where $\mathfrak{D}\phi^\sharp:=\big(D[\zeta]\phi\big)^\sharp$. As a result, we have
\begin{equation}\label{DNAdjTemp1}
(\rho^\sharp)^{-2}\mathfrak{D}^*
=\mathfrak{D}\circ( \rho^\sharp)^{-2}.
\end{equation}
On the other hand, if $\zeta$ is smooth, the para-linearization formula of the Dirichlet-Neumann operator shows that in fact
$$
\mathfrak{D}=\Op(\lambda)+\Op\mathscr{S}^{-1}_{1,0}.
$$
Converting (\ref{DNAdjTemp1}) to the symbolic side modulo symbols of order less than $-1$, this gives
\begin{equation}\label{DNAdjTemp2}
(\rho^\sharp)^{-2}\lambda^{\bullet;1,\mathscr{Q}}-
\lambda\#_{1;\mathscr{Q}}(\rho^\sharp)^{-2}
\in\mathscr{S}^{-1}_{1,0}.
\end{equation}
If we denote the left-hand-side of (\ref{DNAdjTemp2}) as $Q(\zeta;x,l)$, then (\ref{DNAdjTemp2}) of course implies
\begin{equation}\label{DNAdjTemp3}
\begin{aligned}
\big\|\Df_\mu^\alpha Q(\zeta;x,l)\big\|_{C^{s-2.5}_*}
&=O(l^{-1-|\alpha|}),
\quad
\text{for all multi-index }\alpha.
\end{aligned}
\end{equation}
This matrix norm estimate has been deduced for smooth $\zeta$. However, as we can see from the definition of $\lambda$ and symbolic operation, the concrete expression of entries of $Q(\zeta;x,l)$ only contains up to second order derivative of $\zeta$. Thus we can approximate a given $\zeta\in H^{s+0.5}$ by smooth functions and conclude that (\ref{DNAdjTemp3}) continues to hold for $\zeta\in H^{s+0.5}$. Thus for $\zeta\in H^{s+0.5}$, 
$$
(\rho^\sharp)^{-2}\lambda^{\bullet;1,\mathscr{Q}}-
\lambda\#_{1;\mathscr{Q}}(\rho^\sharp)^{-2}
$$
is a symbol of order $-1$. This proves (\ref{AdjT_Lambda}).
\end{proof}

We then proceed to the mean curvature.
\begin{proposition}\label{T_hAdjoint}
Suppose $s>3$, $\zeta\in H^s$. Then with the symbol $h=h^{(2)}+h^{(1)}$ as in Proposition \ref{H(zeta)}, there holds a precise equality
\begin{equation}\label{Adjh}
(\rho^\sharp)^{-2}h^{\bullet;2,\mathscr{Q}}
=h\#_{2,\mathscr{Q}}(\rho^\sharp)^{-2}.
\end{equation}
Thus
$$
T_{(\rho^\sharp)^{-2}}T_h^*=T_hT_{(\rho^\sharp)^{-2}} \mod\Op\Sigma^{0.5}_{<1/2}.
$$
\end{proposition}
\begin{proof}
Of course we may directly deal with the expression of $h(x,l)$. But there is a method that could tactfully avoid such lengthy computation. Consider a variation of the surface $M_\zeta$:
$$
x\to\big(1+\zeta(x)+\theta_1\zeta_1(x)+\theta_2\zeta_2(x)\big)N_0.
$$
By second variation formula for area functional, the bilinear form
$$
\frac{\partial^2 A}{\partial\theta_1\partial\theta_2}(\zeta+\theta_1\zeta_1+\theta_2\zeta_2)\Bigg|_{\theta=0}
$$
is symmetric in $\zeta_1,\zeta_2$. On the other hand, the second variation obviously equals
$$
\int_{\mathbb{S}^2}H'(\zeta)\zeta_2\cdot\zeta_1\rho^2d\mu_0
+2\int_{\mathbb{S}^2}H(\zeta)\zeta_1\zeta_2\rho d\mu_0.
$$
Thus the second order differential operator $\rho^2H'(\zeta)$ is self-adjoint with respect to $\mu_0$. Lifting to $\SU(2)$, it follows that the second order differential operator $(\rho^\sharp)^2\big(H'(\zeta)\big)^\sharp$ is self-adjoint with respect to $\mu(G_0)$. But we also notice that the symbol $h$ is nothing but the symbol of the second order differential operator $\big(H'(\zeta)\big)^\sharp$. By the symbolic calculus theorem of para-differential operators, this proves (\ref{Adjh}).
\end{proof}

\subsection{Symmetrization}\label{Symmetrization}
We are at the place to symmetrize the para-differential system (\ref{EQPara}). The general procedure is identical to that in \cite{ABZ2011} but with different computational details due to the non-flat geometry in our scenario. 

The symmetrizer that we are looking for should take the form
$$
\left(\begin{matrix}
T_p & \\
 & T_q
\end{matrix}\right),
$$
where $p\in\mathcal{A}^{1/2}_{s-2.5}$ should be close to the symbol of $|\nabla_{G_0}|^{0.5}$, and $q=q(x)$ is a scalar function. We require that for some elliptic symbol $\gamma$ of order 1.5, there holds
\begin{equation}\label{Symmetrizer}
\left(\begin{matrix}
T_p & \\
 & T_q
\end{matrix}\right)\left(\begin{matrix}
 & T_\lambda \\
-T_h & 
\end{matrix}\right)
=\left(\begin{matrix}
 & T_\gamma \\
-T_\gamma & 
\end{matrix}\right)\left(\begin{matrix}
T_p & \\
 & T_q
\end{matrix}\right)
\mod
\left(\begin{matrix}
 & \Op\Sigma^{0}_{<1/2} \\
\Op\Sigma^{0.5}_{<1/2} & 
\end{matrix}\right)
\end{equation}
The operator norms of the error term will be bounded linearly by $\|\zeta\|_{H_x^{s+0.5}}$. 

The elliptic symbol $\gamma$ will be the one that satisfy
\begin{equation}\label{T_gamma}
T_\gamma^2=\frac{T_hT_\lambda+T_\lambda^*T_h^*}{2}
\mod\Op\Sigma^{0.5}_{<1/2}.
\end{equation}
This ensures that $T_\gamma$ is approximately self-adjoint. Our symmetrization is finished by several steps.

\textbf{Step 1: solve $\gamma$.} We look for a solution $\gamma=\gamma^{(1.5)}+\gamma^{(0.5)}\in\mathcal{A}^{1.5}_{s-1.5}+\mathcal{A}^{0.5}_{s-2.5}$ of (\ref{T_gamma}). The highest order symbol $\gamma^{(1.5)}$ is just fixed as
$$
\gamma^{(1.5)}=\sqrt{\lambda^{(1)}h^{(2)}},
$$
which is Hermitian on each representation space $\Hh[l]$. Note that we used the commutativity of $h^{(2)}$ and $\lambda^{(1)}$, as shown in Proposition \ref{H(zeta)}. For the next order symbol $\gamma^{(0.5)}$, we just have to fix it by looking at the next order symbols in the composition formula of para-differential operators. That is, we require 
$$
\gamma\#_{1;\mathscr{Q}}\gamma
=\frac{h\#_{1,\mathscr{Q}}\lambda+\lambda^{\bullet;1,\mathscr{Q}}\#_{1,\mathscr{Q}}h^{\bullet;1,\mathscr{Q}}}{2},
$$
or equivalently,
$$
\begin{aligned}
\sqrt{\lambda^{(1)}h^{(2)}}\gamma^{(0.5)}+\gamma^{(0.5)}\sqrt{\lambda^{(1)}h^{(2)}}
&+\sum_{\mu\in\Ind}\Df_\mu\sqrt{\lambda^{(1)}h^{(2)}}\partial_\mu\sqrt{\lambda^{(1)}h^{(2)}}\\
&=
\frac{h\#_{1,\mathscr{Q}}\lambda
+\lambda^{\bullet;1,\mathscr{Q}}\#_{1,\mathscr{Q}}h^{\bullet;1,\mathscr{Q}}}{2}
-\lambda^{(1)}h^{(2)}.
\end{aligned}
$$
Since $[\sqrt{\lambda^{(1)}h^{(2)}},\gamma^{(0.5)}]$ is a symbol of order 1 by Proposition \ref{2OrderComm}, we can just replace $\gamma^{(0.5)}\gamma^{(1.5)}$ by $\gamma^{(1.5)}\gamma^{(0.5)}$, and solve $\gamma^{(0.5)}$ as
$$
\gamma^{(0.5)}
=\frac{(\lambda^{(1)}h^{(2)})^{-1/2}}{2}\left(
\frac{h\#_{1,\mathscr{Q}}\lambda
+\lambda^{\bullet;1,\mathscr{Q}}\#_{1,\mathscr{Q}}h^{\bullet;1,\mathscr{Q}}}{2}
-\lambda^{(1)}h^{(2)}
-\sum_{\mu\in\Ind}\Df_\mu\sqrt{\lambda^{(1)}h^{(2)}}\partial_\mu\sqrt{\lambda^{(1)}h^{(2)}}
\right).
$$

\textbf{Step 2: solve $q$}. We seek for a scalar function $q=q(x)$ defined on $\SU(2)$ such that 
\begin{equation}\label{Tq}
T_qT_hT_\lambda
=\frac{T_hT_\lambda+T_\lambda^*T_h^*}{2}\cdot T_q
\mod\Op\Sigma^{1.5}_{<1/2}.
\end{equation}
This is the step that we should employ the almost self-adjoint properties of $T_\lambda$ and $T_h$. Note that we have already proved, in Proposition \ref{AdjDN}-\ref{T_hAdjoint}, that
$$
\begin{aligned}
T_{(\rho^\sharp)^{-2}}T_\lambda^*&=T_\lambda T_{(\rho^\sharp)^{-2}}
&\mod\Op\Sigma^{-0.5}_{<1/2}\\
T_{(\rho^\sharp)^{-2}}T_h^*&=T_hT_{(\rho^\sharp)^{-2}}
&\mod\Op\Sigma^{0.5}_{<1/2}.
\end{aligned}
$$
So (\ref{Tq}) becomes 
$$
2T_qT_hT_\lambda=\left(T_hT_\lambda
+T_{(\rho^\sharp)^2\lambda}T_hT_{(\rho^\sharp)^{-2}}\right)T_q
\mod\Op\Sigma^{1.5}_{<1/2}.
$$
Here we used that $\rho\in H^{s+0.5}$, so $T_{(\rho^\sharp)^{2}}T_h=T_{(\rho^\sharp)^{2}h}\mod\Op\Sigma^{3.5-s}_{<1/2}$. Since only symbols of order $\geq2$ is of our concern in the above equality, we can simply use the composition formula of para-differential operators and then neglect symbols of order $\leq1$. Thus in order that (\ref{Tq}) holds, it suffices to require 
\begin{equation}\label{SymmTemp}
\begin{aligned}
2q\sum_{\mu\in\Ind}\Df_\mu h^{(2)}\partial_\mu \lambda^{(1)}
&=q\sum_{\mu\in\Ind}\Df_\mu h^{(2)}\partial_\mu \lambda^{(1)}\\
&\quad+2\sum_{\mu\in\Ind}\Df_\mu(\lambda^{(1)}h^{(2)})\partial_\mu q\\
&\quad+q(\rho^\sharp)^{2}\sum_{\mu\in\Ind}\Df_\mu\lambda^{(1)}\partial_\mu\left(\frac{h^{(2)}}{(\rho^\sharp)^{2}}\right)\\
&\quad+q(\rho^\sharp)^{2}\sum_{\mu\in\Ind}\lambda^{(1)}\Df_\mu h^{(2)}\partial_\mu(\rho^\sharp)^{-2}
\mod\mathcal{A}^{1}_{s-3}.
\end{aligned}
\end{equation}

Being tedious at a first glance, there are several key features of the above equality that can be exploited. Recall that $h^{(2)}=(\rho^\sharp)^3\beta_1^{-3/2}(\lambda^{(1)})^2$, and from the expression of $\lambda^{(1)}$, i.e. formula (\ref{DNlambda}), we known that $\lambda^{(1)}$ is a quasi-homogeneous symbol of order 1 in the sense of Definition \ref{QuasiHomoSym}, where the vector field is just $\nabla_{G_0}\zeta^\sharp/\beta_1(x)$. Proposition \ref{PSSymbol} thus ensures that the commutators $[\lambda^{(1)},\partial_\mu\lambda^{(1)}]$ and $[\lambda^{(1)},\Df_\mu\lambda^{(1)}]$ are rough symbols of order 1 and 0 respectively. Consequently, we find 
$$
\begin{aligned}
\Df_\mu h^{(2)}
&=\frac{2(\rho^\sharp)^3}{\beta_1^{3/2}}\Df_\mu\lambda^{(1)}\cdot\lambda^{(1)}
&\mod\mathcal{A}^{1}_{s-3},\\
\Df_\mu(\lambda^{(1)}h^{(2)})
&=\frac{3(\rho^\sharp)^3}{\beta_1^{3/2}}\Df_\mu\lambda^{(1)}\cdot(\lambda^{(1)})^2
&\mod\mathcal{A}^{1}_{s-3},\\
\partial_\mu\left(\frac{h^{(2)}}{(\rho^\sharp)^{2}}\right)
&=(\lambda^{(1)})^2\partial_\mu\big(\rho^\sharp\beta_1^{-3/2}\big)+\frac{2\rho^\sharp\lambda^{(1)}\partial_\mu\lambda^{(1)}}{\beta_1^{3/2}}
&\mod\mathcal{A}^{1}_{s-3},
\end{aligned}
$$
leading to cancellation between the left-hand-side and the first/thrid terms in the right-hand-side of (\ref{SymmTemp}) modulo symbol of order $1$. Thus (\ref{SymmTemp}) is in fact equivalent to
$$
\begin{aligned}
\frac{6(\rho^\sharp)^3}{\beta_1^{3/2}}\sum_{\mu\in\Ind}
\Df_\mu\lambda^{(1)}\cdot(\lambda^{(1)})^2\partial_\mu q
&+2q\big(\rho^\sharp\big)^2\sum_{\mu\in\Ind}\Df_\mu\lambda^{(1)}\cdot(\lambda^{(1)})^2\partial_\mu\big(\rho^\sharp\beta_1^{-3/2}\big)\\
&+\frac{2q(\rho^\sharp)^5}{\beta_1^{3/2}}\sum_{\mu\in\Ind}\Df_\mu\lambda^{(1)}\cdot(\lambda^{(1)})^2\partial_\mu(\rho^\sharp)^{-2}
=0\mod\mathcal{A}^{1}_{s-3}.    
\end{aligned}
$$
In order that this equality holds, it suffices to require 
$$
\frac{6(\rho^\sharp)^3}{\beta_1^{3/2}}\frac{\nabla_{G_0}q}{q}
+2\big(\rho^\sharp\big)^2\nabla_{G_0}\big(\rho^\sharp\beta_1^{-3/2}\big)
+\frac{2(\rho^\sharp)^5}{\beta_1^{3/2}}\nabla_{G_0}(\rho^\sharp)^{-2}=0,
$$
which can be solved as
\begin{equation}\label{Symmq}
q
={\beta_1^{3/2}}{(\rho^\sharp)^{1/3}}
=(1+\zeta^\sharp)^{1/3}\sqrt{(1+\zeta^\sharp)^2+|\nabla_{G_0}\zeta^\sharp|_{G_0}^2}.
\end{equation}

\textbf{Step 3: solve $p$.} Finally, the symbol $p=p^{(0.5)}+p^{(-0.5)}$ is solved directly from the equation 
$$
T_pT_\lambda=T_\gamma T_q
\mod\Op\Sigma^0_{<1/2},
$$
or equivalently
$$
p\#_{1;\mathscr{Q}}\lambda=\gamma\#_{1;\mathscr{Q}}q
\mod\Op\Sigma^0_{<1/2}.
$$
The principal symbol of $p$ is just 
\begin{equation}\label{Symmp}
p^{(0.5)}=\frac{(\rho^\sharp)^{1.5}}{\beta_1^{3/4}}\sqrt{\lambda^{(1)}}.
\end{equation}
The sub-principal symbol $p^{(-0.5)}$ is then fixed as
$$
p^{(-0.5)}
=(\lambda^{(1)})^{-1}\gamma^{(0.5)}q+(\lambda^{(1)})^{-1}\sum_{\mu\in\Ind}\Df_\mu\gamma^{(1.5)}\partial_\mu q.
$$
Thus we have had all symbols $\gamma,p,q$ defined and matching (\ref{Symmetrizer}). The operator norm of the error term in (\ref{Symmetrizer}) is bounded by $\|\zeta\|_{H_x^{s+0.5}}$ linearly when $\zeta\simeq0$, as a consequence of Theorem \ref{Compo2}-\ref{ParaAdj}.

\textbf{Step 4: symmetrization.} We now apply the symmetrizer $\left(\begin{matrix}
T_p & \\ & T_q\end{matrix}\right)$ to the para-differential sysmtem (\ref{EQPara}). We first notice that the expressions for $\partial_tp$ and $\partial_tq$ actually do not explicitly contain time derivative of $\zeta$ and $\phi$, since we assumed \emph{in a priori} that $(\zeta,\phi)$ solves the autonomous partial differential equation (\ref{EQ}). In order not to cause unnecessary confusion, we denote the symbols
\begin{equation}\label{DtpDtq}
a+b
:=\partial_tp^{(0.5)}+\partial_tp^{(-0.5)}
\in\mathcal{A}^{0.5}_{s-3}+\mathcal{A}^{-0.5}_{s-4},
\quad 
c:=\partial_tq\in\mathcal{A}^{0}_{s-2}.
\end{equation}
Here the regularity of $\partial_tp^{(-0.5)}$ is justified by noting that it is of the type indicated in Proposition \ref{SymbolVeryRough}.

Hence we compute
\begin{equation}\label{SymmEQTemp}
\begin{aligned}
\partial_t\left(\begin{matrix}
T_p\zeta^\sharp \\
T_qw^\sharp
\end{matrix}\right)
&=\left(\begin{matrix}
 & T_\gamma \\
-T_\gamma & 
\end{matrix}\right)\left(\begin{matrix}
T_p\zeta^\sharp \\
T_qw^\sharp
\end{matrix}\right)
-T_{\mathfrak{v}^\sharp}\cdot\nabla_{G_0}
\left(\begin{matrix}
T_p\zeta^\sharp \\
T_qw^\sharp
\end{matrix}\right)\\
&\quad
+\left(\begin{matrix}
(T_pT_\lambda-T_\gamma T_q)w^\sharp \\
(-T_qT_h+T_\gamma T_p)\zeta^\sharp
\end{matrix}\right)
+\left(\begin{matrix}
\left(T_{a+b}-\big[T_{\mathfrak{v}^\sharp}\cdot\nabla_{G_0},T_p\big]\right)\zeta^\sharp \\
\left(T_c-\big[T_{\mathfrak{v}^\sharp}\cdot\nabla_{G_0},T_q\big]\right)w^\sharp
\end{matrix}\right)
+\left(\begin{matrix}
T_pf_1(\zeta,\phi) \\
T_qf_2(\zeta,\phi)
\end{matrix}\right),
\end{aligned}
\end{equation}
where the mappings $f_1,f_2$ are as in (\ref{f1f2}). We just define the sum of last three terms as the mapping $f_3(\zeta,\phi)$.

Due to the para-differential reduction, we know that $f_1,f_2$ are controlled by $\left(\|\zeta\|_{H_x^{s+0.5}}+\|\phi\|_{H_x^s}\right)^2$ as (\ref{f1f2}) indicates. Due to (\ref{Symmetrizer}), the operator norms of $T_pT_\lambda-T_\gamma T_q$ and $-T_qT_h+T_\gamma T_p$ are controlled linearly by $\|\zeta\|_{H^{s+0.5}_x}$ when $\zeta$ is small.

It remains to show that given $\alpha\in\mathbb{R}$,
\begin{equation}\label{HsBdd1}
\big[T_{\mathfrak{v}^\sharp}\cdot\nabla_{G_0},T_p\big]
\quad\text{and}\quad
T_{a+b}
\end{equation}
both have operator norms in $\mathcal{L}(H^{\alpha+0.5},H^\alpha)$ controlled linearly by $\|\zeta\|_{H^{s+0.5}_x}+\|\phi\|_{H_x^s}$ when $(\zeta,\phi)\simeq0$, while
\begin{equation}\label{HsBdd2}
\big[T_{\mathfrak{v}^\sharp}\cdot\nabla_{G_0},T_q\big]
\quad\text{and}\quad
T_c\zeta^\sharp
\end{equation}
both have operator norms in $\mathcal{L}(H^{\alpha},H^\alpha)$ controlled linearly by $\|\zeta\|_{H^{s+0.5}_x}+\|\phi\|_{H_x^s}$ when $(\zeta,\phi)\simeq0$.

The claim for first operators in (\ref{HsBdd1})-(\ref{HsBdd2}) follows from the fact that
$\big[T_{\mathfrak{v}^\sharp}\cdot\nabla_{G_0},T_p\big]
$ and $\big[T_{\mathfrak{v}^\sharp}\cdot\nabla_{G_0},T_q\big]$
are para-differential operators of order 1/2 and 0 respectively, since commutator with a vector field does not increase the order of the operator. The control of the norms is a consequence of Corollary \ref{ParaComm}. AS for the claim for second operators in (\ref{HsBdd1})-(\ref{HsBdd2}), this is exactly the content of Lemma 4.10. of \cite{ABZ2011}:
\begin{lemma}[Equivalent form of Lemma 4.10., \cite{ABZ2011}]\label{ABZ4-10}
The para-differential operators $T_{a}$ and $T_{b}$ are both of order $0.5$: for all $\alpha\in\mathbb{R}$, there is an increasing function $K$ vanishing linearly at 0, such that
$$
\big\|T_{a}\big\|_{\mathcal{L}(H^\alpha,H^{\alpha-0.5})}
+\big\|T_{b}\big\|_{\mathcal{L}(H^\alpha,H^{\alpha-0.5})}
\leq K\big(\|\zeta\|_{H_x^{s+0.5}}\big).
$$
The para-differential operator $T_c$ has order 0:
for all $\alpha\in\mathbb{R}$, there is an increasing function $K$ anishing linearly at 0, such that
$$
\big\|T_c\big\|_{\mathcal{L}(H^\alpha,H^{\alpha})}
\leq K\big(\|\zeta\|_{H_x^{s+0.5}}\big).
$$
\end{lemma}
\begin{remark}
The only technical difficulty with the proof of this Lemma is that the symbol $b(x,l)$ is too rough in $x$. In fact, from the expression of $p^{(-0.5)}$, one concludes that $p^{(-0.5)}$ has $H^{s-1.5}$ regularity in $x$. Differentiating in $t$ reduces its regularity to $H^{s-3}\subset C_*^{s-4}$. Thus $b(x,l)$ is a symbol of order $-0.5$ in $l$ and $C_*^{s-4}$ regularity in $x$. By Proposition \ref{T_aNegIndex}, $T_{b}$ is a para-differential operator of order at most $0.5$.
\end{remark}

To summarize, introducing the diagonal unknown $\left(\begin{matrix}
T_p\zeta^\sharp \\ T_q w^\sharp \end{matrix}\right)$, the original system (\ref{EQ}) is equivalent to
\begin{equation}\label{EQSymm'}
\partial_t\left(\begin{matrix}
T_p\zeta^\sharp \\ T_q w^\sharp \end{matrix}\right)=\left(\begin{matrix}
 & T_\gamma \\
-T_\gamma & 
\end{matrix}\right)\left(\begin{matrix}
T_p\zeta^\sharp \\ T_q w^\sharp \end{matrix}\right)
-T_{\mathfrak{v}^\sharp}\cdot\nabla_{G_0} u
+f_3(\zeta,\phi),
\end{equation}
where the error $f_3(\zeta,\phi)$ collects the last three terms in (\ref{SymmEQTemp}), and $\|f_3(\zeta,\phi)\|_{H^s_x}\lesssim\left(\|\zeta\|_{H_x^{s+0.5}}+\|\phi\|_{H_x^s}\right)^2$. This is exactly (\ref{EQSymm}), We hence complete the proof of Theorem \ref{Thm2}.

\section{Energy Estimate and Local Well-posedness}\label{8}
We finally complete the energy estimate for (\ref{EQSymm}), and prove local well-posedness of the original system (\ref{EQ}). For simplicity, we sketch the proof of local well-posedness for small amplitude solutions. Since the argument is rather standard, and all the key estimates are already done in \cite{ABZ2011}, we will only present in detail the parts that are different from \cite{ABZ2011}; for those that are identical to \cite{ABZ2011}, we will directly cite the corresponding Lemmas from \cite{ABZ2011} and show how they aid our proof.

\subsection{Energy Estiamte}
In \cite{ABZ2011}, the authors used the $2s/3$'th order power of the principal symbol $\gamma$ to construct a suitable energy functional. The reason is that the Poisson bracket $\{(\gamma^{(1.5)})^{2\alpha/3},\gamma\}$ vanishes up to order $\alpha-0.5$, hence $[T_{(\gamma^{(1.5)})^{2\alpha/3}},T_\gamma]$ is in fact of order $\leq \alpha$. A similar result applies to the spherical system (\ref{EQSymm}) as well, although not as trivial as in the Euclidean case. 

\begin{proposition}\label{OrderS}
Given $\zeta\in H^{s+0.5}$, let $\gamma^{(1.5)}$ be the rough symbol as given in Subsection \ref{Symmetrization}. Then $\big[T_{(\gamma^{(1.5)})^{2\alpha/3}},T_\gamma\big]$ is a para-differential operator of order $\leq \alpha$, instead of just $\alpha+0.5$. Quantitatively, its operator norm for $H^s\to H^s$ is controlled linearly by $\|\zeta\|_{H_x^{s+0.5}}$ when $\zeta\simeq0$.
\end{proposition}
\begin{proof}
Recall from Subsection \ref{Symmetrization} that 
$$
\gamma^{(1.5)}=\sqrt{\lambda^{(1)}h^{(2)}}
=(\rho^\sharp)^{1.5}\beta_1^{-3/4}(\lambda^{(1)})^{1.5},
$$
which may be abbreviated as
$$
\gamma^{(1.5)}(x,\xi)=g(x)\kappa(\xi)^{1.5}f\left(\frac{b(x,\xi)}{\kappa(\xi)}\right),
$$
where $g(x)$ is some scalar function, $b$ is the symbol of a vector field, all depending on first order derivatives of $\zeta$; $\kappa$ is the symbol of $|\nabla_{G_0}|$, and $f(z)=(1+z^2)^{3/4}$. By our symbolic calculus formula for para-differential operators, namely Theorem \ref{Compo2}, the commutator $\big[T_{(\gamma^{(1.5)})^{2\alpha/3}},T_\gamma\big]$ in fact equals the para-differential operator corresponding to
\begin{equation}\label{GammaCommu}
\sum_{\mu\in\Ind}\Df_{\mu}(\gamma^{(1.5)})^{2\alpha/3}\partial_\mu\gamma^{(1.5)}
-\sum_{\mu\in\Ind}\Df_{\mu}\gamma^{(1.5)}\partial_\mu(\gamma^{(1.5)})^{2\alpha/3},
\end{equation}
modulo an operator of class $\Op\Sigma^{\alpha}_{<1/2}$. Obviously $\kappa(\xi)^{1.5}f\left(\frac{b(x,\xi)}{\kappa(\xi)^2}\right)$ is a quasi-homogeneous symbol of order 1.5 in the sense of Definition \ref{QuasiHomoSym}. We can now apply Proposition \ref{Dkappa}-\ref{PSSymbol} as follows.

We first compute, by Proposition \ref{Dkappa}, that
$$
\Df_\mu\kappa^{1.5}=\frac{3}{2}\kappa^{0.5}\Df_\mu\kappa
\mod\mathscr{S}^{-0.5}_{1,0},
\quad
\Df_\mu\kappa^{2\alpha/3}=\frac{2\alpha}{3}\kappa^{2\alpha/3-1}\Df_\mu\kappa
\mod\mathscr{S}^{\alpha-2}_{1,0}.
$$
Thus by Proposition \ref{PSSymbol}, together with the Leibniz property of $\Df_\mu$,
$$
\begin{aligned}
\Df_{\mu}(\gamma^{(1.5)})^{2\alpha/3}
&=\alpha g^{2\alpha/3}\kappa^{\alpha-1}f\left(\frac{b}{\kappa}\right)^{2\alpha/3}\Df_\mu\kappa
+\frac{2\alpha}{3}g^{2\alpha/3}\kappa^{\alpha-1}
f\left(\frac{b}{\kappa}\right)^{2\alpha/3-1}f'\left(\frac{b}{\kappa}\right)\Df_\mu b\\
&\quad+\frac{2\alpha}{3} g^{2\alpha/3}\kappa^{\alpha-2}f\left(\frac{b}{\kappa}\right)^{2\alpha/3-1}f'\left(\frac{b}{\kappa}\right)\Df_\mu \kappa b
&\mod \mathcal{A}^{\alpha-2}_{s-1.5}\\
&=\frac{2\alpha}{3}(\gamma^{(1.5)})^{2\alpha/3-1}\Df_\mu\gamma^{(1.5)}
&\mod \mathcal{A}^{\alpha-2}_{s-1.5}.
\end{aligned}
$$
On the other hand, by Proposition \ref{PSSymbol},
$$
\begin{aligned}
\partial_\mu(\gamma^{(1.5)})^{2\alpha/3}
&=\frac{2\alpha}{3}g^{2\alpha/3-1}\kappa^\alpha\partial_\mu gf\left(\frac{b}{\kappa}\right)^{2\alpha/3}
+\frac{2\alpha}{3}g^{2\alpha/3}\kappa^\alpha f\left(\frac{b}{\kappa}\right)^{2\alpha/3-1}f'\left(\frac{b}{\kappa}\right)\partial_\mu b
&\mod \mathcal{A}^{\alpha-2}_{s-2.5}\\
&=\frac{2\alpha}{3}(\gamma^{(1.5)})^{2\alpha/3-1}\partial_\mu\gamma^{(1.5)}
&\mod \mathcal{A}^{\alpha-2}_{s-2.5}.
\end{aligned}
$$
Thus (\ref{GammaCommu}) in fact equals $0\mod\mathcal{A}^{\alpha-1}_{s-2.5}$. Consequently  $\big[T_{(\gamma^{(1.5)})^{2\alpha/3}},T_\gamma\big]$ is a para-differential operator of order $\leq\alpha$, not just $\alpha+0.5$.
\end{proof}

We can now state the energy estimate:
\begin{proposition}\label{Energy}
Let
$$
(\zeta,\phi)\in \big(C_TH^{s+0.5}_x\times C_TH^s_x\big)\cap\big(C_T^1H^{s-1}_x\times C_T^1H^{s-1.5}_x\big)
$$
be given, with norm in this space less than some small $R$. We define the good unknown $w^\sharp$ as in (\ref{ActualGUknown}), the symbols $\lambda,h,\mathfrak{b}^\sharp,\mathfrak{v}^\sharp$ as in Section \ref{5}. Suppose $f\in C_TH^{s+0.5}_x\times C_TH^s_x$. 
Then the \emph{linear} Cauchy problem 
\begin{equation}\label{CauchyLin}
\partial_t\left(\begin{matrix}
\eta^\sharp \\
v^\sharp
\end{matrix}\right)
=
\left(\begin{matrix}
-T_{\mathfrak{v}^\sharp}\cdot\nabla_{G_0} & T_\lambda v^\sharp \\
-T_{h} & -T_{\mathfrak{v}^\sharp}\cdot\nabla_{G_0}
\end{matrix}\right)\left(\begin{matrix}
\eta^\sharp \\
v^\sharp
\end{matrix}\right)
+f
\end{equation}
in the unknown $(\eta,v)$ admits a unique solution 
$$
(\eta,v)\in \big(C_TH^{s+0.5}_x\times C_TH^s_x\big)\cap\big(C_T^1H^{s-1}_x\times C_T^1H^{s-1.5}_x\big),
$$
satisfying the following energy estimate for any real number $\alpha\leq s$:
$$
\begin{aligned}
\|\eta(t)\|_{H^{\alpha+0.5}_x}
+\|v(t)\|_{H^{\alpha}_x}
\leq C_\alpha e^{C_\alpha Rt}\left(\|\eta(0)\|_{H^{\alpha+0.5}_x}
+\|v(0)\|_{H^{\alpha}_x}
+\int_0^t\big\|f(\tau);{H^{\alpha+0.5}_x\times  H^\alpha_x}\big\|d\tau\right).
\end{aligned}
$$
\end{proposition}
\begin{proof}
If we define symbols $p,q,\gamma,a,b,c$ corresponding to $(\zeta,\phi)$ as in Subsection \ref{Symmetrization} (especially (\ref{DtpDtq})), and set $u=(T_p\zeta^\sharp,T_qv^\sharp)$, then the procedure in Subsection \ref{Symmetrization} transforms (\ref{CauchyLin}) into the following equivalent form:
\begin{equation}\label{CauchySymm}
\begin{aligned}
\partial_tu
&=\left(\begin{matrix}
 & T_\gamma \\
-T_\gamma & 
\end{matrix}\right)u
-T_{\mathfrak{v}^\sharp}\cdot\nabla_{G_0}
u
+\left(\begin{matrix}
\left(T_{a+b}-\big[T_{\mathfrak{v}^\sharp}\cdot\nabla_{G_0},T_p\big]\right)T_p^{-1} \\
\left(T_c-\big[T_{\mathfrak{v}^\sharp}\cdot\nabla_{G_0},T_q\big]\right)T_q^{-1}
\end{matrix}\right)u
\\
&\quad
+\left(\begin{matrix}
 & (T_pT_\lambda-T_\gamma T_q)T_q^{-1} \\
(-T_qT_h+T_\gamma T_p)T_p^{-1} & 
\end{matrix}\right)u
+\left(\begin{matrix}
T_p & \\
 & T_q
\end{matrix}\right)f.
\end{aligned}
\end{equation}

For any $\alpha\leq s$, we simply apply the operator $\Gamma_\alpha:=T_{(\gamma^{(1.5)})^{2\alpha/3}}$ to (\ref{CauchySymm}). We first notice that
$$
\big\|\Gamma_\alpha u\big\|_{L^2_x}
\simeq
\|u\|_{H^\alpha_x},
$$
and the implicit constants in this equivalence of norms depend only on $\|\zeta\|_{H^{s+0.5}_x}$. In fact, $\gamma^{(1.5)}(x,l)^{2\alpha/3}$ is a symbol of order $\alpha$ and involves only up to first order derivative of $\zeta$. Furthermore, by the symbolic calculus formula, $\Gamma_\alpha\cdot\Gamma_{-\alpha}=\mathrm{Id}+\Op\Sigma^{-1}_{<1/2}$, so at least when $\|\zeta\|_{H^{s+0.5}_x}$ stays close to 0 (which is the case by our assumption), the inverse of the operator $\Gamma_\alpha$ is a para-differential operator of order $-\alpha$. This gives the equivalence of norms. The estimate on the implicit constants is a direct consequence of Theorem \ref{Compo2}.

By Proposition \ref{OrderS}, $\big[\Gamma_\alpha,T_{\gamma}\big]$ is a para-differential operator of order $\alpha$. The symbols in this commutator involve only up to second order derivatives of $\zeta$. We next notice that $\big[\Gamma_\alpha,T_{\mathfrak{v}^\sharp}\cdot\nabla_{G_0}\big]$ is also a para-differential operators of order $\leq \alpha$. This is because commuting with a vector field does not increase the order of an operator. The commutator of $\Gamma_\alpha$ with
$$
\left(T_{a+b}-\big[T_{\mathfrak{v}^\sharp}\cdot\nabla_{G_0},T_p\big]\right)T_p^{-1},
\quad
\left(T_c-\big[T_{\mathfrak{v}^\sharp}\cdot\nabla_{G_0},T_q\big]\right)T_q^{-1}
$$
and 
$$
(T_pT_\lambda-T_\gamma T_q)T_q^{-1},
\quad
(-T_qT_h+T_\gamma T_p)T_p^{-1}
$$
are all of order $\leq\alpha$, due to Lemma \ref{ABZ4-10} and (\ref{Symmetrizer}). Furthermore, Proposition \ref{PSSymbol} ensures that
$$
\partial_t\Gamma_\alpha
=\frac{2\alpha}{3} T_{(\gamma^{(1.5)})^{2\alpha/3-1}\partial_t\gamma^{(1.5)}}
+\text{lower order para-differential operator}.
$$
Since $\gamma^{(1.5)}$ involves first order derivative of $\zeta$, while by the assumption we have $\partial_t\zeta\in C_T^0H^{s-1}_x$, it follows that $\partial_t\gamma^{(1.5)}$ is a symbol with $H^{s-2}\subset C^{s-3}_*$ regularity in $x$, so $\partial_t\Gamma_\alpha$ is still a para-differential operator of order $\alpha$. The operator norms of these are all controlled linearly by $R$, as a consequence of Theorem \ref{Compo2}.

To summarize, for any real number $\alpha\leq s$, we have
\begin{equation}\label{CauchySymmalpha}
\partial_t\Gamma_\alpha u=\left(\begin{matrix}
 & T_\gamma \\
-T_\gamma & 
\end{matrix}\right)\Gamma_\alpha u
-T_{\mathfrak{v}^\sharp}\cdot\nabla_{G_0} \Gamma_\alpha u
+\Upsilon_\alpha u
+\Gamma_\alpha\left(\begin{matrix}
T_p & \\
 & T_q
\end{matrix}\right)f,
\end{equation}
where $\Upsilon_\alpha$ is a para-differential operator of order $\alpha$, with operator norm controlled linearly by $R$. We now find, using that $T_\gamma$ and $T_{\mathfrak{v}^\sharp}\cdot\nabla_{G_0}$ are both approximately anti-self adjoint, the following a priori inequality for a solution $u$ of (\ref{CauchySymm}): for any real number $\alpha\leq s$,
$$
\frac{d}{dt}\|\Gamma_\alpha u(t)\|_{L^2_x}^2
\lesssim R\|\Gamma_\alpha u(t)\|_{L^2_x}^2
+\big\|f(\tau);{H^{\alpha+0.5}_x\times  H^\alpha_x}\big\|\cdot\|\Gamma_\alpha u(t)\|_{L^2_x}.
$$
Using the Grönwall inequality and norm equaivalence $\big\|\Gamma_\alpha u\big\|_{L^2_x}
\simeq
\|u\|_{H^\alpha_x}$, this implies the \emph{a priori} energy inequalities for a solution $u$ of (\ref{CauchySymm}), provided that it does exist:
$$
\|u(t)\|_{H^\alpha_x}
\leq C_\alpha e^{C_\alpha Rt}\left(\|u(0)\|_{H^\alpha_x}+\int_0^t\big\|f(\tau);{H^{\alpha+0.5}_x\times  H^\alpha_x}\big\|d\tau\right).
$$
The regularity $\partial_tu\in C^0_TH^{s-1.5}_x$ follows from the equation itself.

We may then apply the standard duality argument for linear hyperbolic systems. These a priori energy estimates do imply well-posedness of the linear Cauchy problem (\ref{CauchySymm}). It is important that $\alpha$ are allowed to be negative. See for example Section 6.3 in H\"{o}rmander's textbook \cite{Hormander1997}. The estimates for $u$ are then converted back to estimates for $(\eta,v)$.
\end{proof}

\subsection{Banach Fixed Point Argument}
We now formulate a Banach fixed point problem. We shall fix a sufficiently small $R>0$ to measure the size of the solution. As we shall see shortly, the major issue in closing a contraction argument is that we need to control some operator norms under \emph{weaker} regularity assumptions than $(\zeta,\phi)\in C_TH_x^{s+0.5}\times C_TH_x^s$. We thus seek aid from Proposition \ref{T_aNegIndex} constantly.

\textbf{The Metric Space and the Map.} 
We will be dealing with the good unknown $(\zeta,w)$ instead of $(\zeta,\phi)$. We set
$$
\mathcal{G}:\left(\begin{matrix}
\zeta \\
\phi
\end{matrix}\right)
\to
\left(\begin{matrix}
\zeta \\
w
\end{matrix}\right)
$$
the assignment of $(\zeta,\phi)$ to the corresponding good unknown. We already know that $\mathcal{G}$ is a diffeomorphism in a neighbourhood of zero in $C_TH^{s+0.5}_x\times C_TH^s_x$. 

Fix $\bar{\mathfrak{X}}_{R}$ to be the set of all 
$$
(\zeta,w)\in \big(C_TH^{s+0.5}_x\times C_TH^s_x\big)\cap\big(C_T^1H^{s-1}_x\times C_T^1H^{s-1.5}_x\big)
$$ 
such that $(\zeta,w)$ has norm $\leq R$ in that space, and also $(\zeta(0),\phi(0))$ has norm $\leq A_sR$ in that space, for some $A_s\ll1$ to be specified. We equip $\bar{\mathfrak{X}}_{R}$ with a \emph{weaker} metric
$$
\begin{aligned}
d\big((\zeta_1,w_1),(\zeta_2,w_2)\big)
:=\|\zeta_1-\zeta_2\|_{C_TH^{s-1}_x}
+\|w_1-w_2\|_{C_TH^{s-1.5}_x}.
\end{aligned}
$$
By weak compactness of bounded closed convex sets in Hilbert spaces, $(\bar{\mathfrak{X}}_R,d)$ is a complete metric space.

Now let $(\zeta,w)\in\bar{\mathfrak{X}}_R$ be given, and $(\zeta,\phi)=\mathcal{G}^{-1}(\zeta,w)$. We define the symbols $\lambda,h,\mathfrak{b}^\sharp,\mathfrak{v}^\sharp$ as in Section \ref{5}, and the mappings $f_1,f_2$ as in (\ref{f1f2}). Let us consider the \emph{linear} Cauchy problem for the unknown $(\eta,v)\in C_TH^{s+0.5}_x\times C_TH^s_x$:
\begin{equation}\label{ParaLin'}
\partial_t\left(\begin{matrix}
\eta^\sharp \\
v^\sharp
\end{matrix}\right)
=
\left(\begin{matrix}
-T_{\mathfrak{v}^\sharp}\cdot\nabla_{G_0} & T_\lambda \\
-T_{h} & -T_{\mathfrak{v}^\sharp}\cdot\nabla_{G_0}
\end{matrix}\right)\left(\begin{matrix}
\eta^\sharp \\
v^\sharp
\end{matrix}\right)
+\left(\begin{matrix}
f_1(\zeta,\phi)\\
f_2(\zeta,\phi)
\end{matrix}\right).
\end{equation}
Using the operator norm estimates in Subsection \ref{Symmetrization}, together with the quadratic estiamte for $f_1,f_2$ in (\ref{f1f2}), Proposition \ref{Energy} makes sure that the Cauchy problem for this linear hyperbolic system can be uniquely solved with given initial data. Thus, (\ref{ParaLin'}) has a unique solution $(\eta,v)\in C_TH_x^{s+0.5}\times C_TH^s_x$ with initial value being $(\zeta(0),w(0))$, satisfying the energy estimates
\begin{equation}\label{Energy'}
\begin{aligned}
\|\eta(t)\|_{H^{s+0.5}_x}+\|v(t)\|_{H^s_x}
&\leq C_se^{Rt}(A_sR+R^2t),\\
\|\partial_t\eta(t)\|_{H^{s-1}_x}+\|\partial_tv(t)\|_{H^{s-1.5}_x}
&\leq C_se^{Rt}(A_sR+R^2t)+C_sR^2.
\end{aligned}
\end{equation}
We set the \emph{solution operator} $\mathcal{S}$ to be the unique solution $(\eta,v)$ of (\ref{ParaLin'}) with initial data being $(\zeta(0),\phi(0))$.

The map $\mathscr{F}$ on $\bar{\mathfrak{X}}_R$ under consideration is thus defined by
$$
\mathscr{F}\left(\begin{matrix}
\zeta \\
w 
\end{matrix}\right)
:=\mathcal{S}\mathcal{G}^{-1}\left(\begin{matrix}
\zeta \\
w 
\end{matrix}\right).
$$
When $R\simeq0$ and $T$ is suitably small, if $A_s$ is also suitably small, then for $(\zeta,w)\in\bar{\mathfrak{X}}_R$, (\ref{Energy'}) implies 
$$
\begin{aligned}
\big\|\mathscr{F}(\zeta,w);{C_TH^{s+0.5}_x\times C_TH^{s}_x}\big\|
&\leq C_se^{C_sT}(A_s+C_sTR)R
\ll R\\
\big\|\partial_t\mathscr{F}(\zeta,w);{C_TH^{s-1}_x\times C_TH^{s-1.5}_x}\big\|
&\leq C_se^{C_sT}(A_s+C_sTR)R+C_sR^2
\ll R
\end{aligned}
$$
Thus a suitable choice of $A_s$ and $T$ makes sure that $\mathscr{F}$ maps $\bar{\mathfrak{X}}_{R}$ to itself.

\textbf{Contraction Argument.}
Let us now consider two different $(\zeta,w),(\zeta_1,w_1)\in\bar{\mathfrak{X}}_R$. We use void subscript or subscript 1 to denote functions or symbols constructed out of $(\zeta,w),(\zeta_1,w_1)\in\bar{\mathfrak{X}}_R$ respectively. Write $(\eta_j,v_j)=\mathscr{F}(\zeta_j,w_j)$. We would like to prove that for small $R>0$, there holds
\begin{equation}\label{Contraction}
\|\eta-\eta_1\|_{C_TH^{s-1}_x}
+\|v-v_1\|_{C_TH^{s-1.5}_x}
\ll d\big((\zeta,w),(\zeta_1,w_1)\big),
\end{equation}
which ensures $\mathscr{F}:\bar{\mathfrak{X}}_R\to\bar{\mathfrak{X}}_R$ is a contraction.

Write for simplicity $U=\eta-\eta_1$, $V=v-v_1$. We compute, by the definition of solution operator, that $U(0)=0$, $V(0)=0$, and
\begin{equation}\label{UV}
\begin{aligned}
\partial_t\left(\begin{matrix}
U \\
V
\end{matrix}\right)
&=\left(\begin{matrix}
-T_{\mathfrak{v}^\sharp}\cdot\nabla_{G_0} & T_{\lambda} \\
-T_{h} & -T_{\mathfrak{v}^\sharp}\cdot\nabla_{G_0}
\end{matrix}\right)\left(\begin{matrix}
U \\
V
\end{matrix}\right)\\
&\quad
+\left(\begin{matrix}
T_{\mathfrak{v}^\sharp-\mathfrak{v}_1^\sharp}\cdot\nabla_{G_0} & -T_{\lambda-\lambda_1} \\
T_{h-h_1} & T_{\mathfrak{v}^\sharp-\mathfrak{v}_1^\sharp}\cdot\nabla_{G_0}
\end{matrix}\right)\left(\begin{matrix}
\eta \\
v
\end{matrix}\right)
+\left(\begin{matrix}
f_1(\zeta,w)-f_1(\zeta_1,w_1)\\
f_2(\zeta,w)-f_2(\zeta_1,w_1)
\end{matrix}\right).
\end{aligned}
\end{equation}
(\ref{UV}) is exactly in the form indicated by Proposition \ref{Energy}. We just need to show that the $H^{s-1}_x\times H^{s-1.5}_x$ norm of the last two terms of the right-hand-side in (\ref{UV}) is controlled linearly by $d\big((\zeta,w),(\zeta_1,w_1)\big)$. But this is exactly the content of Lemma 6.6, Corollary 6.7 and Lemma 6.8 of \cite{ABZ2011}, which concerns with exactly the same quantities. If we want to reproduce the proof, we just constantly use Proposition \ref{SymbolVeryRough} and Proposition \ref{T_aNegIndex}. For example, the first-order symbol $h^{(1)}-h_1^{(1)}$ contains up to second order derivative of $\zeta$ and $\zeta_1$, so 
$$
\big\|h^{(1)}(x,l)-h_1^{(1)}(x,l)\big\|_{C^{s-4}_*}
\leq Cl\|\zeta-\zeta_1\|_{H^{s-1}_x}.
$$
By Proposition \ref{T_aNegIndex}, this implies that the para-differential operator $T_{h^{(1)}-h_1^{(1)}}$ is of order 2, and
$$
\big\|T_{h^{(1)}-h_1^{(1)}}\eta\big\|_{H^{s-1.5}_x}
\lesssim \|\zeta-\zeta_1\|_{H^{s-1}_x}\|\eta\|_{H^{s+0.5}_x}
\lesssim Rd\big((\zeta,w),(\zeta_1,w_1)\big).
$$

The final step then becomes simple. Applying Proposition \ref{Energy} to (\ref{UV}) with $\alpha=s-1.5$, we find
$$
\|U\|_{C_TH^{s-1}_x}+\|V\|_{C_TH^{s-1.5}_x}
\leq C_se^{C_sRT}TRd\big((\zeta,w),(\zeta_1,w_1)\big).
$$
Thus for $T$ suitably small, the mapping $\mathscr{F}$ is a contraction from $\bar{\mathfrak{X}}_R$ to itself, hence has a unique fixed point. The fixed point of $\mathscr{F}$ is exactly the solution of the para-linearized system (\ref{EQPara}), thus the original (\ref{EQ}).

\begin{remark}
In \cite{ABZ2011}, the existence and uniqueness for solution of the symmetrized capillary-gravity water waves system is proved as follows: one first prove energy estimate for a smoothed system with some parameter $\varepsilon$ (e.g. artificial viscosity), and show that a solution sequence converges to a genuine solution of the original one as $\varepsilon\to0$ by a weak compactness argument, then prove uniqueness by considering the energy estimate for the difference of two solutions, which is just (\ref{UV}). The reason that we would rather formulate a Banach fixed point theorem is that the latter is more constructive, and is in fact equivalent to the iterative method used in most literature on quasi-linear hyperbolic systems, for example Chapter 6 of \cite{Hormander1997}.
\end{remark}

\section*{Acknowledgement}
The author would like to thank Professor Carlos Kenig for weekly discussion on this project, and Professor Gigliola Staffilani for constant support. The author benefits a lot from discussion with Professor Jean-Marc Delort, Veronique Fischer, Isabelle Gallagher and David Jerison. Thanks also goes to the author's friend, Kai Xu, for comments on representation theory.

\bibliographystyle{alpha}
\bibliography{References}

\end{spacing}
\end{document}